\documentclass[11pt]{amsart}

\usepackage{amsmath}
\usepackage{amsthm}

\usepackage{amssymb}
\usepackage{amscd}
\usepackage[dvips]{graphicx}
\usepackage{xcolor}

\usepackage{epsfig}

\usepackage[margin=0.5in,top=0.6in]{geometry}
\usepackage{amsfonts}

\numberwithin{equation}{section} 

\usepackage{color}

 \usepackage{caption}
 
 \usepackage{sidecap}

\renewcommand{\theequation}{\mbox{\arabic{section}.\arabic{equation}}}

\newtheorem{lemma}{Lemma}
\newtheorem{corollary}{Corollary}
\newtheorem{proposition}{Proposition}
\newtheorem{theorem}{Theorem}
\newtheorem{remark}{Remark}

\input epsf

\begin{document}
\title{Stratospheric planetary flows \\
from the perspective of the Euler equation on a rotating sphere}
\author{A. Constantin and P. Germain}

\address{Faculty of Mathematics, University of Vienna, Oskar-Morgenstern-Platz 1, 1090 Vienna, Austria}

\email{adrian.constantin@univie.ac.at}%

\address{Courant Institute of Mathematical Sciences, New York University,
251 Mercer Street, New York, NY 10012, USA}

\email{pgermain@cims.nyu.edu}

\maketitle

\maketitle

\begin{abstract}
This article is devoted to stationary solutions of Euler's equation on a rotating sphere, and to their relevance to 
the dynamics of stratospheric flows in the atmosphere of the outer planets of our solar system and in polar regions 
of the Earth. For the Euler equation, under appropriate conditions, rigidity results are established, ensuring that the solutions are either zonal or rotated 
zonal solutions. A natural analogue of Arnold's stability criterion is proved. In both cases, the lowest mode Rossby-Haurwitz stationary solutions (more precisely, 
those whose stream functions belong to the sum of the first two eigenspaces of the Laplace-Beltrami operator) appear as limiting cases. We study the stability properties of 
these critical stationary solutions. Results on the 
local and global bifurcation of non-zonal stationary solutions from classical Rossby-Haurwitz waves are also obtained. Finally, we show that stationary solutions of the 
Euler equation on a rotating sphere are building blocks for travelling-wave solutions of the 3D system that 
describes the leading order dynamics of stratospheric planetary flows, capturing the characteristic decrease of density and increase of temperature with height in 
this region of the atmosphere.
\end{abstract}



\section{Introduction}

\subsection{Euler's equation on a rotating sphere}

The Euler equation set on the 2-sphere $\mathbb{S}^2$, with standard metric, in a frame rotating at speed $\omega \in \mathbb{R}$ about the polar axis, 
can be written in terms of the stream function $\psi$ as
\begin{equation}\label{Eomega}
\partial_t \,\Delta \psi + \frac{1}{\cos\theta}\, \left[ -\partial_\theta \psi \,\partial_\varphi + \partial_\varphi \psi \,\partial_\theta \right] \left(\Delta \psi + 2\omega\sin\theta\right) = 0\,.\tag{${\mathcal E}_\omega$}
\end{equation}
(see for instance Section 13.4.1 in~\cite{holton}). Here $(\theta,\varphi) \in (-\frac{\pi}{2},\frac{\pi}{2}) \times [0,2\pi)$ are the latitude and longitude angles (see Figure \ref{fig1}), and $\Delta$ is the Laplace-Beltrami operator, see Section~\ref{sectionstructure} for a more thorough presentation of the differential geometry of the sphere.

\begin{SCfigure}
\includegraphics[width=9cm]{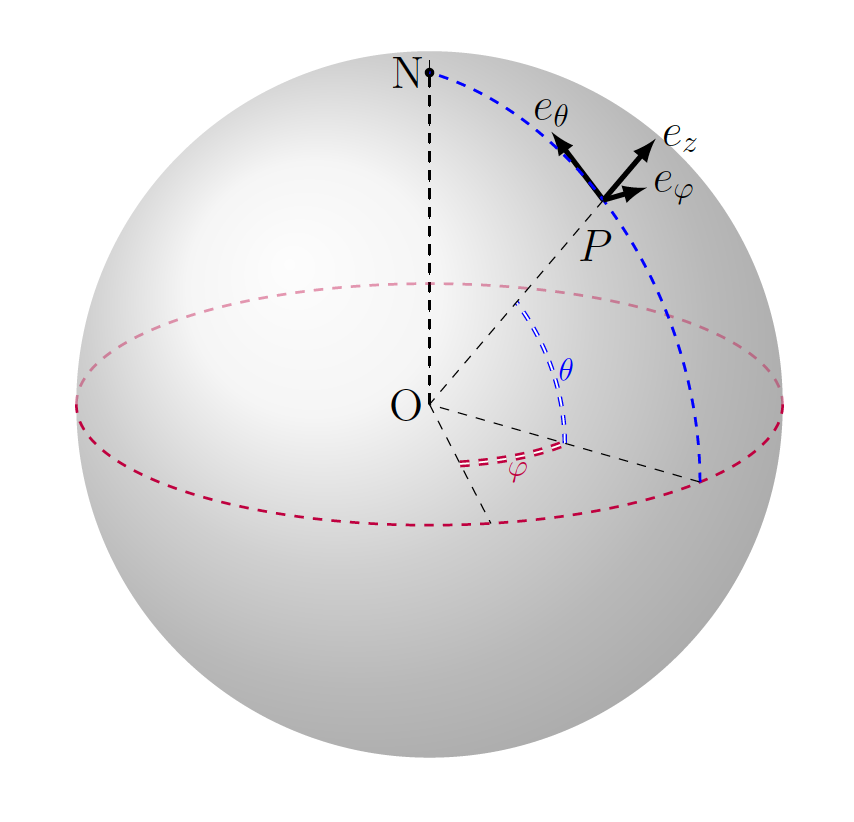}
\caption{\footnotesize{The rotating spherical coordinate system $(r',\varphi,\theta)$: $\theta \in [-\frac{\pi}{2},\,\frac{\pi}{2}]$ is
the angle of longitude, $\varphi \in [0,\pi]$ is the angle of latitude, and $r'=|OP|$ is the distance
from the origin at planet's center. The North Pole is at $\theta=\frac{\pi}{2}$, the Equator is on $\theta=0$ and the 
South Pole is at $\theta=-\frac{\pi}{2}$.}}
\label{fig1}
\end{SCfigure}

Conservation laws and energy estimates for equation \eqref{Eomega} are very similar to the more classical framework of the (two-dimensional) Euclidean space, or of the torus. Therefore, one can use energy methods to prove local well-posedness in $H^s$, $s>2$ (see for instance~\cite{MB}), and then an analogue of the Beale-Kato-Majda theorem~\cite{BKM} ensures global well-posedness in $H^s$, $s>2$. For these matters we refer to Taylor~\cite{Taylor}, where this program is explained in greater detail. Global existence being settled, we now turn to more qualitative questions.

A recent line of research has focused on the behavior of $(E_\omega)$ as $\omega \to \infty$, proving convergence to zonal flows after time averaging, see Cheng-Mahalov~\cite{ChMa}, Wirosoetisno~\cite{Wiro}, Taylor~\cite{Taylor}. 

In the present paper, the focus will be on stationary solution and their stability. Stationary solutions of \eqref{Eomega} can be found by solving the semilinear elliptic problem
\begin{equation}
\label{ellipticpsi}
\Delta \psi = F(\psi) - 2\omega \sin \theta\,,
\end{equation}
where $F$ is a smooth function; throughout regions without critical points of the stream function, any stationary solution comes about in this way (see \cite{cj1}). 
Two fundamental explicit classes of stationary solutions of \eqref{ellipticpsi} are zonal flows and Rossby-Haurwitz planetary waves:
\begin{itemize}
\item \textit{Zonal flows} correspond to stream functions which only depend on the polar angle $\theta$, $\psi = \psi(\theta)$. Their stability has been investigated in Caprino-Marchioro~\cite{CaMa} and Taylor~\cite{Taylor}. Notice that one can use the invariance of the equation through the action of $\mathbb{O}(3)$ to obtain non-zonal solutions from zonal flows. In particular, this explains the apparently 
{\it ad hoc} Ansatz made in~\cite{ahm}.
\item \textit{Rossby-Haurwitz solutions of degree $k$} are given by the stream function 
$$
\psi_0 = \alpha \sin \theta + Y(\varphi,\theta),
$$
where $Y$ belongs to the $k$-th eigenspace $\mathbb{E}_k$ of the Laplace-Beltrami operator and $\alpha \in \mathbb{R}$,  solving~\eqref{ellipticpsi} for 
$$
F(\psi) = - k(k+1) \psi \quad\text{and} \quad \omega = \alpha \left( 1 - \frac{k(k+1)}{2} \right)\,.
$$
These are the classical Rossby-Haurwitz planetary waves, due to Craig \cite{craig}, who found the complete 
nonlinear solution corresponding to the solutions of  the linearized barotropic vorticity equation obtained by Rossby \cite{rossby} and 
Haurwitz \cite{hau} on the beta-plane and on the sphere, respectively (see the discussion in \cite{verk}). 

The degree $1$ modes are either zonal (of the form $\beta \sin \theta$) or rotations of 
this zonal solution. These solutions can be thought-of as ground states and will play a key role in this article, since they are distinguished in many respects:
\begin{itemize}
\item They are minimizers of the Dirichlet energy for fixed $L^2$ norm. In more hydrodynamical terms, they minimize the enstrophy $\int_{{\mathbb S}^2} |\Delta \psi|^2\,{\rm d}\sigma$ 
for fixed kinetic energy $\int_{{\mathbb S}^2} |U|^2\,{\rm d}\sigma$, $U$ being the associated velocity field and ${\rm d}\sigma=\cos\theta\,{\rm d}\theta{\rm d}\varphi$ being the surface element on the sphere. This gives a first proof of its stability.
\item The $L^2$-projection of any solution $\psi$ onto the subspace of ground states is conserved by the flow of~\eqref{Eomega}. This is assertion $(iii)$ in Proposition~\ref{conserved}, giving a second proof of the 
stability of this ground state.
\item This solution is isochronal, in other words its Lagrangian flow is periodic.
\end{itemize}

The next-gravest modes have degree 2, and they comprise waves with a more intricate latitude variation. These solutions are of considerable interest in meteorology. For example, the wave obtained by 
setting $Y$ proportional to the spherical harmonics $Y_2^1$ is commonly observed in the terrestrial atmosphere, being known as the $5$-day wave since it travels westwards with a period of about 
5 days (see the field data in \cite{he}). These waves are also preponderant in the atmospheres of the outer planets of our solar system (Jupiter, Saturn, Uranus, Neptune); see \cite{dow}. 
The instability of the Rossby-Haurwitz waves is a key factor in the lack of predictability of the weather in long-term forecasts \cite{ben}. Earlier attempts to study the stability of the degree 2 waves 
by numerical means are somewhat inconclusive, the results being partly contradicting (see the discussion in \cite{ben}). Since the linear stability of the zonal mode-two solutions was proved by Taylor \cite{Taylor}, the 
issue of the nonlinear stability or instability of the mode-two Rossby-Haurwitz waves, settled by Theorem~\ref{thmRH}, is of outmost importance.
\end{itemize}

\subsection{Stratospheric planetary flows}

The dynamics of the planetary atmospheres in our solar system is intertwined with its thermal structure, with temperature gradients 
driving specific atmospheric flows which, in turn, modify the temperature field. To a large extent, Earth's weather is conditioned by the redistribution of the excess
solar insolation received by the tropical regions towards the poles, whereas sunlight is not the main driver of the atmospheric motions of 
Jupiter, Saturn and Neptune, all these planets radiating about twice as much energy as they receive from the Sun (see \cite{dow, lun}). The atmosphere of 
these three planets is primarily made up of a hydrogen-helium gas mixture and the dynamics is dominated by zonal flows that feature a banded structure -- 
flows of this type are also common in terrestrial polar regions but the Earth's atmospheric circulation at midlatitudes is much more complicated (see \cite{cj2}). Latitudinal bands are also 
the main atmospheric features on Uranus but they are considerably fewer and only visible in the infrared (see \cite{lun}):  the lack of an internal energy source results in less drastic changes of 
the atmospheric flow pattern with latitude and an overall rather bland atmosphere -- an additional factor being that Uranus lies sideways, with its poles 
where its equator should be, so that its icy core makes the temperatures globally uniformly low, impeding the formation of localized flow patterns. 
Note that the strongest atmospheric flows on any planet were measured on Neptune (reaching supersonic speeds in equatorial regions). The 
wind patterns on Neptune and Uranus lack Jupiter's multiple zonal winds that flow alternatively in opposite directions: there 
is a westward atmospheric flow at low latitudes and an eastward flow at higher latitudes in each hemisphere (see Figure \ref{fig2}).

\begin{figure}[ht]
\begin{center}
\begin{minipage}{199mm}{
\resizebox*{19cm}{!}{\includegraphics{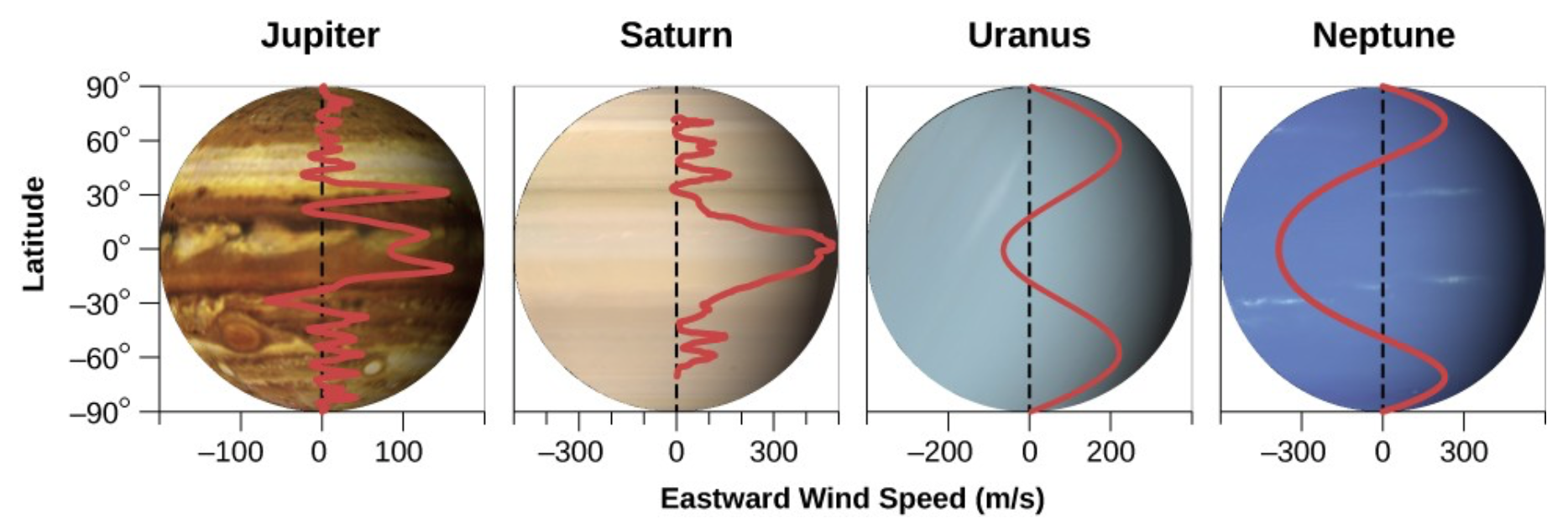}}}
\caption{\footnotesize{Variation of the mean zonal winds with latitude on the giant planets of our solar system, 
measured relative to the planet's rotation speed about its polar axis (Credit: OpenStax CNX). The traces of methane (which absorbs red light) in their upper atmosphere gives Uranus and Neptune a 
blue hue, obscuring the visibility of specific flow patterns. These pictures show the high altitude clouds just beneath the stratosphere (at the top of the troposphere) -- the only planetary atmosphere in our solar system 
transparent enough to see through from space being that of the Earth.}}
\label{fig2}
\end{minipage}
\end{center}
\end{figure}

The stratosphere of the Earth and of the outer planets of our solar system is thermally stably stratified and presents a rapid decrease of density with height. In large-scale atmospheric flow conditions, gas 
parcels move adiabatically (i.e. without loosing or gaining heat), thus conserving potential temperature (see \cite{holton}). Since stratospheric isentropic surfaces 
(level sets of potential temperature) are practically of constant height (see the discussions in \cite{cf, knox}), taking a spherical model for each 
specific planet, we see that the motion of inviscid fluids on the surface of a rotating sphere is relevant to the dynamics of the stratosphere. Consequently, the study 
of the Euler equation on a rotating sphere offers insight into the dynamics of the stratosphere of the outer planets in our solar system. We will pursue this aspect in some detail in Section 6. 
Note that this issue is not relevant for the inner planets of our solar system: Mercury has no atmosphere, while the atmospheres of Venus and Mars lack a stratosphere. Without 
a stable stratification (which is the hallmark of the stratosphere, where the energy balance is primarily determined by absorption and emission of radiation), 
two-dimensional flows on a rotating sphere fail to be pertinent for atmospheric dynamics. 

It is well-established that the Rossby-Haurwitz waves  play an important role in the large-scale dynamics of atmospheric flows. In particular, the zonal spherical harmonics capture the pattern of the band 
structure in the upper atmosphere of the outer planets, comprised of zonal flows whose direction alternates from westward to eastward. But 
superimposed on these flows one often observes non-zonal features (referred to as ``eddies" in the atmospheric sciences), like Jupiter's Great Red Spot. First 
thought of as an exotic occurrence, it is now appreciated that such long-lived vortices are frequently encountered throughout the solar system. For example, 
several spots occur in Saturn's and Neptune's atmosphere (like the Polar Hexagon and the Great Dark Spot, respectively), while in Jupiter's atmosphere 
there are upwards of a dozen smaller vortices near the latitude 22$^\circ$S (where 
the Great Red Spot is centred); see the discussion in \cite{dow}. The spherical harmonics by themselves cannot cover this plethora of flows. It is therefore of interest 
to develop an approach that can provide large families of Rossby-Haurwitz-like solutions. This is precisely our aim when pursuing bifurcation from Rossby-Haurwitz waves. 
Let us also point out that the scarcity of stable atmospheric flows makes flows likely to be unstable also of great interest. In this context, note that 
Jupiter's Great Red Spot is confined by an eastward jet stream to its 
south and a westward one to its north (which explains why it rotates counterclockwise), and while a change of direction in zonal flows if often indicative of their linear instability by means of a variant of 
Rayleigh's criterion (see \cite{Taylor}), only small changes were noted in the dynamics of the Great Red Spot since 1831. This shows that even potentially unstable flow patterns can be 
(relatively) long-lived. We would also like to emphasise the importance of not performing the analysis within the flat geometry of the $f$-plane or $\beta$-plane approximation. 
Even the Great Red Spot (large enough to engulf Earth) is not an isolated vortex but rather a global system involving involving one large
anticyclone, and several smaller ones in the same anticyclone zone, and a filamentary region in the adjacent equatorward shear zone (see \cite{dow}).

\subsection{Main results and organization of the article} 

\label{MR}

\subsubsection{Symmetry} Our first result addresses symmetries of the solutions of~\eqref{ellipticpsi}.

\begin{theorem}
\label{theosymmetry}
Consider $\psi$ a solution of~\eqref{ellipticpsi} for some $\omega \in \mathbb{R}$. If $F' > -6$, then $\psi$ is a zonal flow, modulo a rotation in $\mathbb{O}(3)$.
\end{theorem}

This theorem is an immediate consequence of Theorem~\ref{theo1} in Section~\ref{sectionrigidity}. Note that the number $-6$ is the second eigenvalue of the Laplace-Beltrami operator.
Such symmetry results are proved in Constantin-Drivas-Ginsberg~\cite{CDG} for general Riemannian surfaces with Killing field, under the condition that $F'$ is larger than the smallest eigenvalue of the Laplace-Beltrami operator, which on the sphere amounts to $F'>-2$. The symmetric structure of the sphere explains the improvement that we are able to obtain.

The sharpness of this result can be seen from the Rossby-Haurwitz solutions in $\mathbb{E}_1 + \mathbb{E}_2$ (the first and second eigenspaces of $-\Delta$). Modulo $\mathbb{O}(3)$, they are of the type $\psi_0 = \alpha \sin \theta + Y$, where $\alpha \in \mathbb{R}$ and $Y \in \mathbb{E}_2$, and correspond to $F'=-6$; they are not zonal in general. This is a first indication that Rossby-Haurwitz solutions in $\mathbb{E}_1 + \mathbb{E}_2$ play a key role in the qualitative study of the 2D Euler problem on the sphere.

\subsubsection{Stability of zonal solutions}

Persistent zonal (east-west) flows are ubiquituous in planetary atmospheres. While the Earth's atmosphere typically exhibits meandering jets, the zonal flows on the outer planets present a long-lived 
coherence that might be indicative of stability. In Section~\ref{SOSS}, after a brief discussion of the Rayleigh and Fjortoft necessary criteria for the linear instability of zonal flows on a sphere, we present the proof of an Arnold-type 
nonlinear stability result for zonal flows. While variations of this result can be found in the research literature, with the possibly most comprehensive approach for general rotating bodies due to Taylor \cite{Taylor}, 
we present a simple proof for the case of a rotating sphere. We also show that this result yields the stability of the mean zonal flow patterns on Uranus and Neptune.

\subsubsection{Stability of Rossby-Haurwitz solutions of degree 2} 
The classical theory of Arnold~\cite{Arnold1,Arnold2} adresses the stability of stationary solutions of the Euler equation on a planar domain $\Omega$. Let $\lambda$ denote the largest eigenvalue of the negative Dirichlet Laplacian $-\Delta$ on $\Omega$. Then a solution of $\Delta \psi = F(\psi)$ is called Arnold stable of type I if  $-\lambda < F' < 0$, and Arnold stable of type II if $F'>0$. These solutions can be showed to be nonlinearly stable for the 2D Euler problem, through the construction of an energy functional.

Due to its rich symmetry structure, the relevant threshold for the sphere becomes the second eigenvalue of $-\Delta$, instead of the first as in the classical Arnold theory. Namely, the following result holds.

\begin{theorem}
If $\psi$ is a solution of~\eqref{ellipticpsi} with $-6<F'<0$, then it is stable in $H^2({\mathbb S}^2)$.
\end{theorem}

This theorem corresponds to Theorem~\ref{thmarnold} in Section~\ref{SOSS}.
It is natural to conjecture that the transition to instability occurs as $\min\{ F'\}$ crosses the threshold of $-6$. This points to the importance of Rossby-Haurwitz solutions in $\mathbb{E}_1 + \mathbb{E}_2$, for which $F'=-6$. It was observed following Theorem~\ref{theosymmetry} that these flows are in some sense the "first" genuine non-zonal stationary flows; we now see that they appear at the transition to instability since 
we are able to analyze the stability of Rossby-Haurwitz solutions of degree 2 precisely.

\begin{theorem}
\begin{itemize}
\item[(i)] The set $\mathbb{E}_1 + \mathbb{E}_2$ is stable in $H^2({\mathbb S}^2)$.
\item[(ii)] Non-zonal Rossby-Haurwitz solutions of degree 2 are unstable in $H^2({\mathbb S}^2)$
\item[(iii)] Zonal Rossby-Haurwitz solutions of degree 2 are stable in $H^2({\mathbb S}^2)$.
\end{itemize}
\end{theorem}

 This theorem is a short version of Theorem~\ref{thmRH} in Section~\ref{sectionRH}.
The notion of stability considered here is always Lyapunov stability. 

\subsubsection{Bifurcation} Exact solutions of the vorticity equation \eqref{Eomega} are very useful for gaining insight into the dynamics, for validating models and for evaluating numerical 
discretizations. Due to the scarcity of the available explicit exact non-zonal solutions, we develop a bifurcation approach that permits the construction of families of exact solutions to \eqref{ellipticpsi} for some classes of parameter-dependent nonlinearities $F$ with specific structural properties. To ensure that these solutions comprise non-zonal flows, we implement a symmetry-breaking approach. We also derive {\it a priori} bounds of the velocity 
and of the vorticity along the continuum of solutions constructed by means of bifurcation (see Theorems \ref{exi1}-\ref{exi2}).

\subsubsection{Stratospheric flows} We develop a methodology for embedding the 2D flows studied hitherto into the 3D dynamics of the stratosphere. More precisely, solutions to $\Delta\psi=F(\psi)$ on 
${\mathbb S}^2$ correspond to geostationary solutions $\widehat{\psi}=\psi(\varphi + \omega t,\theta)$ of the vorticity equation \eqref{Eomega} which are restrictions to the surface of the sphere of solutions 
to the 3D system that describes the leading-order dynamics of the stratosphere (see Theorem \ref{zf}), where the density decreases and the temperature increases with height above the tropopause. This feature 
appears to be replicated by solutions to \eqref{ellipticpsi} for suitable forcings of the gravity acting in the radial direction (typically encountered if the geopotential surfaces are not spheres) but we do not 
pursue this direction in the present paper.

\subsection{Comparison with the torus} It is instructive to compare the results which have been stated above to what is known for 2D-Euler on the torus, for which we adopt the parametrization $(x,y) \in [0,2\pi]^2$ with periodic boundary conditions, the metric being the standard one. 

From a geometric and analytic viewpoint, the sphere and the torus share many similarities: the eigenspaces of the Laplacian can be described explicitly 
and isometries act transitively in both cases. However, the sphere is more symmetric than the torus, since, for instance, all geodesics are equivalent up to isometries; this ultimately results in a different picture for the stability of stationary solutions of the Euler equation.

On the torus, there are two important classes of explicit stationary solutions: shear flows, whose stream function only depends on $y$, and eigenfunctions of the Laplacian; these are analogous to zonal flows and Rossby-Haurwitz solutions, respectively.

The first eigenvalue of $-\Delta$ is $1$, with a distinguished eigenfunction provided by $\sin y$, and also a shear flow (sometimes called Kolmogorov flow). It is analogous to the stationary solution $\sin \theta$ on the sphere. Just like for the sphere, solutions of $-\Delta \psi = F(\psi)$ with $F'>-1$ are constant. However, there are elements of the first eigenspace of the Laplacian which are not shear, and there is furthermore a rich and nontrivial family of stationary solutions bifurcating from $\sin y$, see~ Coti-Zelati-Elgindi-Widmayer~\cite{CZEW}.

It is instructive to consider other aspect ratios: if we change the parameterization to $[0,2\pi L] \times [0,2\pi]$, the Kolmogorov flow $\sin y$ is still a stationary solution, but new phenomena occur. First, the only stationary solutions in a small neighborhood are shear flows~\cite{CZEW}. Second, the stability properties of the Kolmogorov flow depend on the aspect ratio $L$:
\begin{itemize}
\item the flow is stable for $L>1$ for obvious energy reasons; 
\item it can also be proved to be stable for $L=1$, see~\cite{Arnold2};
\item linear instability for $L<1$ has been the subject of a number of works, starting with the sketch in Arnold and Meshalkin~\cite{AM}, followed by Belenkaya-Friedlander-Yudovich~\cite{BFY} and Butt{\`a}-Negrini~\cite{BN}.
\end{itemize}
Finally, we mention that inviscid damping for the linearized problem around the Kolmogorov flows was studied in Wei-Zhang-Zhao~\cite{WZZ}; such questions are certainly harder for the sphere, due to the more involved eigenfunction decomposition for the Laplacian.

\subsection*{Acknowledgements} The authors thank Vladimir \v{S}verak for a very helpful communication on the prescribed Gauss curvature equation on the sphere. 
While working on this project, PG was visiting the University of Vienna, to which he is very grateful. PG was supported by the NSF grant DMS-1501019, by 
the Simons collaborative grant on weak turbulence, and by the Center for Stability, Instability and Turbulence (NYUAD).
AC was supported by the Austrian Science Foundation (FWF) "Wittgenstein-Preis Z 387-N".

\section{The vorticity equation for inviscid flow on a rotating sphere}

\label{sectionstructure}

\subsection{Differential geometry of the sphere}

Since by the ``hairy ball theorem" the $2$-sphere $\mathbb{S}^2$ does not possess a continuously differentiable field of unit tangent vectors 
(see \cite{gp}), it is not possible to cover $\mathbb{S}^2$ with one chart. However, the standard longitude-latitude spherical coordinates 
$(\varphi,\theta) \in (0,2\pi) \times (-\frac{\pi}{2},\frac{\pi}{2})$ provide us (see Fig. \ref{fig1}) with a chart 
$$(\varphi,\theta) \in (0,2\pi) \times \big(-\tfrac{\pi}{2},\tfrac{\pi}{2}\big) \mapsto  (\cos\varphi\cos\theta,\,\sin\varphi\cos\theta,\,\sin\theta) \in \mathbb{S}^2$$
covering $\mathbb{S}^2$ with the half-circle $\varphi=0$ (the international date line, including the poles) excised. A smooth atlas for $\mathbb{S}^2$ 
is obtained by coupling this with the chart
$$(\tilde{\varphi},\tilde{\theta}) \in (0,2\pi) \times \big(-\tfrac{\pi}{2},\tfrac{\pi}{2}\big) \mapsto  
(-\cos\tilde{\varphi}\cos\tilde{\theta},\,\sin\tilde{\theta},\,\sin\tilde{\varphi}\cos\tilde{\theta}) \in \mathbb{S}^2$$
covering $\mathbb{S}^2$ with the equatorial half-circle parametrized in spherical coordinates by $\{\theta=0\,,\,\varphi \in [\pi,2\pi]\}$ excised: the 
bijective transformation $(x,y,z) \mapsto (-x,z,y)$ between the above parametrizations is equivalent to first rotating the Euclidean coordinate system by $\frac{\pi}{2}$ about the $x$-axis and then 
by $\pi$ about the $z$-axis.

Throughout this paper we will mostly rely only on spherical coordinates. The double-valued ambiguity along the international date 
line of the chart provided by the spherical coordinates can be resolved by assuming a periodic dependence on the azimuthal angle $\varphi$. 
At several places in the manuscript the use of spherical coordinates (more precisely, the fact that longitude is not well-defined and latitude 
circles degenerate into a single point at the poles) introduces artificial 
singularities at the poles that can be ruled out either by switching to the chart that covers the 
sphere with the equatorial half-circle removed or by taking smoothness into account -- see relation \eqref{bu} below. 

The $4$-dimensional tangent bundle $T\mathbb{S}^2$ of the $2$-sphere $\mathbb{S}^2$ is not parallelizable as a consequence 
of the hairy ball theorem. However, at every point $X$ of $\mathbb{S}^2 \setminus \{N,S\}$, having spherical coordinates 
$(\varphi,\theta) \in [0,2\pi] \times (-\frac{\pi}{2},\frac{\pi}{2})$, the tangent vectors 
$$\mathbf{e}_\varphi = \frac{1}{\cos \theta}\, \partial_\varphi, \qquad \mathbf{e}_\theta = \partial_\theta\,,$$
provide us with a basis of the tangent space $T_X\mathbb{S}^2$ at $X \in \mathbb{S}^2$. In these coordinates, the Riemannian volume element is
$${\rm d}\sigma = \cos \theta \,{\rm d}\varphi\,{\rm d}\theta\,,$$
and the classical differential operators (gradient and Laplace-Beltrami for scalar functions $\psi: \mathbb{S}^2 \to {\mathbb R}$, divergence for vector fields 
$F: \mathbb{S}^2 \to T\mathbb{S}^2$) are given by
\begin{align*}
& \operatorname{grad} \psi = \partial_\theta \psi \,\mathbf{e}_\theta + \frac{\partial_\varphi \psi}{\cos \theta} \,\mathbf{e}_\varphi \,,\\
& \operatorname{div} (F_\varphi \, \mathbf{e}_\varphi + F_\theta\, \mathbf{e}_\theta) = \frac{1}{\cos \theta}\,[\partial_\varphi F_\varphi + \partial_\theta (\cos \theta \,F_\theta)] \,,\\
& \operatorname{\Delta} \psi = \operatorname{div}  \operatorname{grad} \psi = \partial_\theta^2 \psi - \tan \theta\, \partial_\theta \psi + \frac{1}{(\cos \theta)^2}\, \partial_\varphi^2 \psi\,,
\end{align*}
with the formula for $\operatorname{grad}$ following from the definition $\partial_s \psi(\gamma(s)) = \operatorname{grad} \psi \cdot \gamma'(s)$ 
for a path $\gamma$ on $\mathbb{S}^2$, while the formula for $\operatorname{div}$ follows by duality (see \cite{richt}). The 
covariant derivatives have the form 
$$ \nabla_{\mathbf{e}_\theta} \mathbf{e}_\theta = \nabla_{\mathbf{e}_\theta}   \mathbf{e}_\varphi = 0 \,,\qquad 
 \nabla_{\mathbf{e}_\varphi}  \mathbf{e}_\theta = - \tan \theta \, \mathbf{e}_\varphi \,,\qquad 
 \nabla_{\mathbf{e}_\varphi} \mathbf{e}_\varphi = \tan \theta  \, \mathbf{e}_\theta\,,$$
and can be computed by projecting Euclidean derivatives on the tangent space to the $2$-sphere: for instance $ \nabla_{\mathbf{e}_\varphi}  \mathbf{e}_\theta = P \partial_\varphi  \mathbf{e}_\theta$, where $P$ is the projection operator. Note also that the $2$ sphere $\mathbb{S}^2$ 
admits the complex structure $J$ (corresponding to a rotation in the tangent space) defined by
$$J \mathbf{e}_\varphi = \mathbf{e}_\theta, \qquad J \mathbf{e}_\theta = - \mathbf{e}_\varphi \,.$$
The Laplace-Beltrami operator $\Delta$ on $\mathbb{S}^2$, operating in the Hilbert space $L^2({\mathbb S}^2)$ obtained as the 
completion of the smooth functions $f: {\mathbb S}^2 \to {\mathbb C}$ of zero mean 
(i.e., with $\iint_{{\mathbb S}^2} f\,{\rm d\sigma}=0$) with respect to the inner product
$$\langle f_1,f_2\rangle =\iint_{{\mathbb S}^2} f_1\,\overline{f_2}\,{\rm d\sigma}\,,$$
(where the overbar denotes complex conjugation) is negative, self-adjoint and its spectrum is the discrete set of eigenvalues 
$\bigcup_{j \ge 1} \{-j(j+1)\}$,  the spherical harmonics $\{Y_j^m\}_{j \ge 1,\,|m| \le j}$ being an orthonormal basis of eigenfunctions in 
$L^2({\mathbb S}^2)$, with 
$$\Delta Y_j^m=-j(j+1) Y_j^m\,,\qquad j \ge 1\,,\ m \in \{-j,\dots,j\}$$ 
(see the Appendix).

\subsection{Euler equation} We start from the stream function $\psi$, from which the velocity field $U$ is obtained as
$$U = J \operatorname{grad} \psi \quad\text{or}\quad U= u \mathbf{e}_\varphi + v \mathbf{e}_\theta\,,$$
with the geostrophic relations
\begin{equation}\label{uv}
\begin{cases}
\displaystyle u = - \partial_\theta \psi \,,\\
\displaystyle v = \tfrac{1}{\cos \theta}\,\partial_\varphi \psi\,.
\end{cases}
\end{equation}
The vorticity is then given by
\begin{equation}
\label{utoomega}
\Omega = \Delta \psi = - \operatorname{div} J U\,,
\end{equation}
while the material derivative, describing the transport by the velocity field $U$, can be expressed in the form
$$D_t = \partial_t+  \nabla_U  =\partial_t + u \nabla_{\mathbf{e}_\varphi} + v \nabla_{\mathbf{e}_\theta}\,.$$
The material derivative can be applied to scalars or vectors; when applied to scalars, it becomes
$$D_t = \partial_t + \frac{1}{\cos \theta} \left[ -\partial_\theta \psi \,\partial_\varphi + \partial_\varphi \psi \,\partial_\theta \right]\,.$$

The Euler equation on a sphere rotating at speed $\omega$ about the polar axis can then be written for either of the 
variables $\psi$, $U$, or $\Omega$. At the level of the stream function, it is equation \eqref{Eomega}:
$$
\partial_t \,\Delta \psi + \frac{1}{\cos\theta}\, \left[ -\partial_\theta \psi \partial_\varphi + \partial_\varphi \psi \partial_\theta \right] \left(\Delta \psi + 2\omega\sin\theta\right) = 0\,.
$$
To express the Euler equation in terms of the velocity field we have to add the divergence-free condition to the evolution equation \eqref{Eomega}, together with an auxiliary scalar pressure field $p$ (that arises as a Lagrange multiplier for the divergence-free constraint)
\begin{equation}
\label{eqU}
\left\{
\begin{array}{l}
D_t U + 2 \omega \sin \theta \,JU = - \operatorname{grad} p\,, \\
\operatorname{div} U = 0\,.
\end{array}
\right.
\end{equation}
At the level of the vorticity, the evolution equation becomes\footnote{Sometimes termed the two-dimensional
baroclinic Ertel equation for the material conservation of potential vorticity -- see \cite{holton}.}
\begin{equation}\label{ertel}
D_t (\Omega + 2\omega \sin \theta) = 0\,,
\end{equation}
and has to be complemented with the Biot-Savart law, which recovers at every instant $t$ the stream function (and thus the velocity field) from the vorticity:
$$\psi(\xi_0)=\iint_{\mathbb{S}^2} {\mathcal G}(\xi,\xi_0) \Omega(\xi)\,{\rm d}\sigma(\xi)\,,$$
where (see \cite{drit, mp})
$${\mathcal G}(\xi,\xi_0)=\frac{1}{2\pi}\,\ln\Big(\frac{|\xi-\xi_0|}{2} \Big) \quad\text{with}\quad |\xi-\xi_0| \quad\text{the distance in}\ 
{\mathbb R}^3\ \text{between}\ \xi \neq \xi_0\ \text{on}\ \mathbb{S}^2$$ 
is the Green function, satisfying 
\begin{equation}\label{gfe}
\Delta \mathcal{G}(\xi,\xi_0)=\delta(\xi-\xi_0) - \frac{1}{4\pi}\,,
\end{equation}
with $\delta$ the Dirac delta distribution corresponding to a point vortex located at $\xi_0 \in \mathbb{S}^2$. Note that since the velocity field 
is divergence-free, an immediate consequence of the divergence theorem is the validity of the Gauss constraint
\begin{equation}\label{gaussc}
\iint_{\mathbb{S}^2} \Omega\,{\rm d}\sigma=0\,,
\end{equation}
so that the factor $- \frac{1}{4\pi}$ in \eqref{gfe} plays the role of a compensating uniform 
vorticity distribution on $\mathbb{S}^2$ to guarantee the validity of \eqref{gaussc}. Regarding the apparent singularity of the meridional velocity component $v$ in \eqref{uv} at 
the poles, let  us point out that for any $C^1$-function $\psi: {\mathbb S}^2 \to {\mathbb R}$ the continuity of the gradient with 
respect to the spherical coordinates $(\varphi,\theta)$ implies
\begin{equation}\label{bu}
\lim_{\theta \to \pm \frac{\pi}{2}} \partial_\varphi\psi(\varphi,\theta)=0
\end{equation}
since on any genuine circle of latitude $\theta \in \big(-\tfrac{\pi}{2},\tfrac{\pi}{2}\big)$ the periodicity of $\psi$ in the longitudinal direction ensures the existence of a point where 
$\partial_\varphi\psi$ vanishes.

Generally, insight in the flow dynamics 
is more readily available working with the  stream function $\psi$, rather than with the vorticity $\Omega=\Delta \psi$. 
Equation \eqref{Eomega} is the barotropic vorticity equation, describing the motion
of an inviscid, unforced, incompressible, homogeneous fluid on a rotating sphere (see \cite{gill}). Note that with respect to the 
symplectic structure on $\mathbb{S}^2$, whose Poisson bracket is given in spherical coordinates by 
\begin{equation}\label{pobr}
\{ f, h\}= \frac{1}{\cos\theta}\,\big( \partial_\theta h\, \partial_\varphi f - \partial_\theta f\, \partial_\varphi h \big)\,,
\end{equation}
the vorticity equation \eqref{Eomega} can be expressed as the Hamiltonian flow
$$\partial_t (\Delta\psi + 2\omega\sin\theta) = \{ \Delta\psi + 2\omega\sin\theta, \psi\}\,.$$

\subsection{Symmetries} 

\label{subsectionsymmetries}
The Euler equations \eqref{Eomega} with different rotation speeds $\omega$ are related through the change-of-frame transformation
\begin{equation}
\label{changeofframe}
\psi_0(\varphi,\theta,t) \longleftrightarrow \psi_\omega(\varphi,\theta,t) = \psi_0(\varphi+\omega t,\theta,t) + \omega \sin \theta\,.
\end{equation}
More precisely, $\psi_\omega$ solves \eqref{Eomega} if and only if $\psi_0$ solves the Euler equation on a fixed sphere, (${\mathcal E}_0$).
\medskip

The classical scaling of the Euler equation in a fixed frame has to be modified to take $\omega$ into account: 
namely, if $\psi(\varphi,\theta,t)$ solves \eqref{Eomega}, 
then $\lambda \psi(\varphi,\theta,\lambda t)$ solves (${\mathcal E}_{\lambda\omega}$), for any $\lambda>0$.\medskip

Another invariance is related to the symmetries of the 2-sphere, given by the orthogonal group $\mathbb{O}(3)$, a compact Lie group of 
dimension 3, consisting of the isometries of ${\mathbb R}^3$ which fix the origin: one can think of $\mathbb{O}(3)$ as the group of 
orthogonal real $3 \times 3$ matrices or as a group of transformations of ${\mathbb R}^3$. The action of ${\mathbb O}(3)$ is defined by
$$G f(X)=f(GX)\,,\qquad X \in {\mathbb S}^2\,,\quad G \in {\mathbb O}(3)\,,$$ 
for a scalar function $f: {\mathbb S}^2 \to {\mathbb R}$ and the following transformations leave the set of solutions of \eqref{Eomega} invariant:
$$\begin{cases}
& \psi(X,t) \mapsto \psi (GX,t)\,, \\
& U(X,t) \mapsto G U (GX,t)\,, \\
& \Omega(X,t) \mapsto \Omega(GX,t)\,,
\end{cases}\qquad X \in {\mathbb S}^2\,.$$
Note that the non-abelian subgroup of $\mathbb{O}(3)$ of all orthogonal $3 \times 3$ real matrices $R$ with $\text{det}(R)=1$, itself a compact Lie group of 
dimension 3, is called the rotation group $\mathbb{SO}(3)$ since each transformation $X \mapsto RX$ with $R \in \mathbb{SO}(3)$ 
can be obtained by first choosing a fixed direction through the origin and subsequently rotating the coordinate system through a suitable angle about this direction as an axis (see \cite{richt} and the Appendix).

\subsection{Conservation laws} The identification of integrals of motion provides insight into the flow dynamics.

\begin{proposition}[Integrals of motion for a fixed sphere]
\label{conserved}
The following quantities are conserved by smooth solutions of the vorticity equation~(${\mathcal E}_0$):
\begin{itemize}
\item[(i)] (kinetic energy) $\displaystyle\frac{1}{2}\iint_{{\mathbb S}^2} |U|^2 \,{\rm d}\sigma$,
\item[(ii)] (Casimir invariants) $\displaystyle\iint_{{\mathbb S}^2}  F(\Omega)\,{\rm d}\sigma$ for any differentiable function $F$,
\item[(iii)] (first eigenspace of the Laplace-Beltrami operator) the vorticity components in the direction of each of the three spherical harmonics of 
degree $1$.
\end{itemize}
\end{proposition}

\begin{proof}
(i) Taking the time derivative of the energy gives, with the help of~\eqref{eqU},
$$
\frac{d}{dt}\Big(  \frac{1}{2}\iint_{{\mathbb S}^2} |U|^2 \,{\rm d}\sigma \Big)= \iint_{{\mathbb S}^2} [- \nabla_U U - 2 \omega \sin \theta J U - \operatorname{grad} p ] \cdot U \,{\rm d}\sigma\,.
$$
It is immediate to see that the second term on the right side vanishes (due to the antisymmetry of $J$), and the third term as well 
(since $U$ is divergence-free). But the first term on the right side is also zero since
$$
\iint_{{\mathbb S}^2} \nabla_U U \cdot U\,{\rm d}\sigma = \iint_{{\mathbb S}^2} U \cdot \operatorname{grad} \frac{|U|^2}{2} \,{\rm d}\sigma = -  
\iint_{{\mathbb S}^2} \operatorname{div} U \, \frac{|U|^2}{2} \,{\rm d}\sigma = 0\,.
$$

\smallskip

\noindent 
(ii) This is an immediate consequence of the vorticity equation \eqref{ertel}, due to a simple change of variables in phase-space 
since the flow-map is area preserving ($U$ being divergence-free). Note that these Casimir functionals reflect the underlying noncanonical 
Hamiltonian structure induced by \eqref{pobr}: their Poisson bracket with any other functional vanishes.\medskip

\noindent
(iii) Since the spherical harmonics $Y_1^{\pm 1}$ can be obtained from $Y_1^0$ by a rotation, due to the 
invariance of the vorticity equation under the action of $\mathbb{SO}(3)$, it suffices to show that 
$\iint_{{\mathbb S}^2} \Omega \sin \theta \,{\rm d}\sigma$ is conserved. Since $\operatorname{grad} \sin \theta = \cos \theta \, \mathbf{e}_\varphi$, an integration by parts reveals that
$$
\iint_{{\mathbb S}^2} \Omega \sin \theta \,{\rm d}\sigma = - \iint_{{\mathbb S}^2} \operatorname{div} JU \sin \theta \,{\rm d}\sigma 
= \iint_{{\mathbb S}^2} U \cdot \mathbf{e}_\varphi \cos \theta\,{\rm d}\sigma\,.
$$
Taking the time derivative of this quantity and using the velocity evolution equation~\eqref{eqU} gives
$$
\frac{d}{dt} \iint_{{\mathbb S}^2} U \cdot \mathbf{e}_\varphi \cos \theta \,{\rm d}\sigma = \iint_{{\mathbb S}^2} \left[ -\nabla_U U - 2\omega \sin \theta J U - \operatorname{grad} p \right] \cdot \mathbf{e}_\varphi \cos \theta \,{\rm d}\sigma\,.
$$
The third term on the above right-hand side is zero, since $\operatorname{div} (\cos \theta \, \mathbf{e}_\varphi ) =0$. The second term is also zero, as can be seen by using the defition of $U$ in terms of $\psi$, and the fact that $\operatorname{div} (\sin \theta \cos \theta \, \mathbf{e}_\varphi ) =0$:
$$
\iint_{{\mathbb S}^2} \sin \theta JU \cdot \mathbf{e}_\varphi \, \cos \theta \,{\rm d}\sigma = \iint_{{\mathbb S}^2} \operatorname{grad} \psi \cdot (\sin \theta \cos \theta \, \mathbf{e}_\varphi ) \,{\rm d}\sigma = 0\,.
$$
Finally, using the identity $\nabla_U (\cos \theta \, \mathbf{e}_\varphi) = JU  \sin \theta$, the fact that $U$ is divergence-free, and the antisymmetry of $J$, one obtains 
$$
\iint_{{\mathbb S}^2} \nabla_U U \cdot  \mathbf{e}_\varphi \cos \theta \,{\rm d}\sigma 
= - \iint_{{\mathbb S}^2} U \cdot \nabla_U (\mathbf{e}_\varphi \cos \theta) \,{\rm d}\sigma 
= - \iint_{{\mathbb S}^2} U \cdot JU \sin \theta \,{\rm d}\sigma = 0\,,
$$
so that the third term vanishes as well.\end{proof}

Due to the transformation~\eqref{changeofframe}, Proposition \ref{conserved} gives invariants for~\eqref{Eomega} with $\omega \neq 0$.

\begin{corollary}[Integrals of motion for a rotating sphere]\label{iom} 
The following quantities are constant in time for smooth solutions of~\eqref{Eomega} with $\omega \neq 0$:
\begin{itemize}
\item[(i)] (kinetic energy) $\displaystyle \int |U|^2 \,dx$,
\item[(ii)] (Casimir invariants) $\displaystyle \int  F(\Omega + 2\omega \sin \theta)\,dx$ for any differentiable function $F$,
\item[(iii)] (first eigenspace of the Laplace-Beltrami operator) ${\rm e}^{{\rm i}m\omega t} c_1^m(t)$ for the coefficients
$$c_1^m=  \iint_{{\mathbb S}^2} \Omega\, \overline{Y_1^j} \,{\rm d}\sigma\,,\qquad m \in \{-1,0,1\}\,,$$
of the $L^2({\mathbb S}^2)$-expansion of $\Omega$ in terms of the spherical harmonics $\{Y_j^m\}_{j \ge 1,\,|m| \le j}$. In particular, the real 
number $c_1^0(t)$ and the absolute values of the complex numbers $c_1^{\pm 1}(t)$ are flow-invariants. 
\end{itemize}
\end{corollary}

\begin{proof}
The proof of (i)-(ii) being rather straightforward, we only discuss that of (iii). Since $\sin\theta=2\sqrt{\frac{\pi}{3}}\,Y_1^0(\theta)$, using 
\eqref{changeofframe} and the explicit dependence of the spherical harmonics $Y_1^m$ on the longitude angle $\varphi$, for a solution 
$\psi_\omega(\varphi,\theta,t)$ of \eqref{Eomega} we get
\begin{align*}
{\rm e}^{{\rm i}m\omega t} c_1^m(t) &= {\rm e}^{{\rm i}m\omega t} \iint_{{\mathbb S}^2}  \psi_\omega(\cdot,\cdot,t)\,\overline{Y_1^m}\,{\rm d}\sigma\\
&= {\rm e}^{{\rm i}m\omega t}  \iint_{{\mathbb S}^2}  \psi_0(\cdot + \omega t,\cdot,t)\,\overline{Y_1^m}\,{\rm d}\sigma 
+ 2\omega\sqrt{\tfrac{\pi}{3}}\,{\rm e}^{{\rm i}m\omega t} \iint_{{\mathbb S}^2}  Y_1^0\,\overline{Y_1^m}\,{\rm d}\sigma \\
&= \iint_{{\mathbb S}^2}  \psi_0(\cdot,\cdot,t)\,\overline{Y_1^m}\,{\rm d}\sigma 
+ 2\omega\sqrt{\tfrac{\pi}{3}}\,{\rm e}^{{\rm i}m\omega t} \iint_{{\mathbb S}^2}  Y_1^0\,\overline{Y_1^m}\,{\rm d}\sigma
\end{align*}
and we can conclude by Proposition \ref{conserved}, the spherical harmonics being orthonormal.\end{proof}

\subsection{Stationary solutions} Stationary solutions of~\eqref{Eomega} satisfy
\begin{equation}
\label{ellipticpsi0}
\left[ -\partial_\theta \psi\, \partial_\varphi + \partial_\varphi \psi \,\partial_\theta \right] \left(\Delta \psi + 2\omega\sin\theta\right) = 0
\end{equation}
In view of the transformation~\eqref{changeofframe}, these solutions, which are stationary for a fixed $\omega$, correspond to uniform rotation around the polar axis for other values of $\omega$.

Geometrically, \eqref{ellipticpsi0} means that the gradients of the stream function $\psi$ and of the potential vorticity $\Delta \psi + 2\omega\sin\theta$ 
are parallel. Since the gradient is orthogonal to the level set, in regions of ${\mathbb S}^2$ where $\text{grad} \,\psi \neq (0,0)$ the rank theorem ensures that 
 \eqref{ellipticpsi0} is locally equivalent to the elliptic problem \eqref{ellipticpsi}, namely 
 $$\Delta \psi + 2\omega\sin\theta=F(\psi)$$
holds for some $C^1$-function $F$; see the discussion in \cite{cj1}. It is easy to check that any solution of the elliptic 
problem~\eqref{ellipticpsi} on ${\mathbb S}^2$ will also solve~\eqref{ellipticpsi0}, but the converse is not true in general. 
For example, any zonal function $\psi$ solves ~\eqref{ellipticpsi0} but does not 
have to be a solution of~\eqref{ellipticpsi}, as shown by the case of constant functions.

Two classes of explicit solutions of \eqref{ellipticpsi0} are known:
\begin{itemize}
\item zonal solutions $\psi(\theta)$;
\item Rossby-Haurwitz waves of the form
\begin{equation}\label{srhw}
\psi(\varphi,\theta)= \tfrac{2\omega}{2-j(j+1)}\,\sin\theta + \beta \,Y(\varphi,\theta)\,,\qquad j \ge 2\,,\quad \beta \in {\mathbb R}\,,\quad Y \in E_j \,,
\end{equation}
where $E_j$ is the $(2j+1)$-dimensional eigenspace of the Laplace-Beltrami operator associated to the eigenvalue $-j(j+1)$. These are solutions 
of \eqref{ellipticpsi0} for $F(s)=-j(j+1)s$.
\end{itemize}
It is easy to verify that functions of this type solve \eqref{ellipticpsi0}. Using the symmetry~\eqref{changeofframe}, one obtains explicit non-trivial 
travelling-wave solutions of \eqref{Eomega} of the form
\begin{equation}\label{eotw}
\psi(\varphi-ct,\theta)=\alpha\,\sin\theta + \beta\,Y(\varphi -ct,\theta) \quad\text{with}\quad 
\begin{cases}
&Y \in E_j \ (j \ge 1)\,,\qquad \quad \alpha \in {\mathbb R}\,,\quad \beta \in {\mathbb R} \setminus \{0\}\,,\\
&c=\frac{2\omega}{j(j+1)} + \alpha\,\frac{j(j+1)-2}{j(j+1)}  \,,
 \end{cases}
\end{equation}
as one can easily check. Two particular cases are of great interest:
\begin{itemize}
\item for $\alpha=\tfrac{2\omega}{2-j(j+1)}$ with $j \ge 2$ we obtain the stationary waves \eqref{srhw} with wave speed $c=0$;
\item for $\alpha=\omega$ we obtain geostationary waves that propagate azimuthally with wave speed $c=\omega$, which can be subsumed into 3D stratospheric flows 
(see Section 6).
\end{itemize}
Some other explicit solutions are listed in the next section and an approach providing us with further classes of (non-explicit) stationary solutions is provided in Section 6.

\section{Rigidity results}

\label{sectionrigidity}

Due to the important role played by~\eqref{ellipticpsi} in the quest of stationary solutions of the vorticity equation~\eqref{Eomega}, and given the 
accessibility of symmetries for ${\mathbb S}^2$, it is natural to expect some concrete answers to the question about solutions 
capturing symmetries of the spherical domain. 

\begin{theorem} \label{theo1} Consider 
classical solutions $\psi$ of~\eqref{ellipticpsi}, for a continuously differentiable function $F: {\mathbb R} \to {\mathbb R}$.
\begin{itemize}
\item[(i)] If $\omega=0$ and $F'>-2$, then $\psi$ is a constant. This statement is optimal since the non-constant elements of the first eigenspace 
$\mathbb{E}_1$ of the Laplace-Beltrami operator satisfy $\Delta \psi = - 2 \psi$.
\item[(ii)] If $\omega=0$ and $F'>-6$, then $\psi$ is zonal up to a transformation in $\mathbb{SO}(3)$. This statement is optimal since the non-zonal elements 
of the second eigenspace $\mathbb{E}_2$ of the Laplace-Beltrami operator satisfy $\Delta \psi = - 6 \psi$.
\item[(iii)] For $\omega \neq 0$ there exists $\alpha \in \mathbb{R}$ such that $\mathbb{P}_2 \psi = \alpha \sin \theta$.
\item[(iv)] If $\omega \neq 0$ and $F'>-6$, then $\psi$ is zonal up to a transformation in $\mathbb{SO}(3)$. This statement is optimal since the 
Rossby-Haurwitz wave \eqref{srhw} with $j=2$ and $Y$ a non-trivial linear combination of all spherical harmonics $\{Y_2^m\}_{|m| \le 2}$ 
satisfies $\Delta\psi = -6 \psi$.
\end{itemize}
\end{theorem}

\begin{proof} (i) Pairing~\eqref{ellipticpsi} for $\omega=0$ with $\Delta \psi$ and integrating by parts leads to the identity
$$
\iint_{{\mathbb S}^2} |\Delta \psi|^2 \,{\rm d}\sigma = - \iint_{{\mathbb S}^2} F'(\psi) |\operatorname{grad} \psi|^2 \,{\rm d}\sigma\,.
$$
Assuming that $F'>2$, the fact that the smallest nonzero eigenvalue of $-\Delta$ is $2$ gives
$$
\iint_{{\mathbb S}^2} |\Delta \psi|^2 \,{\rm d}\sigma < 2 \iint_{{\mathbb S}^2} |\operatorname{grad} \psi|^2 \,{\rm d}\sigma 
\leq \iint_{{\mathbb S}^2} |\Delta \psi|^2\,{\rm d}\sigma \,
$$
which is a contradiction unless $\Delta \psi =0$, which forces $|\operatorname{grad} \psi|=0$.

\medskip

\noindent
(ii) Up to the action of $\mathbb{O}(3)$, we can assume that $\mathbb{P}_2 \psi = \alpha \sin \theta$ for some $\alpha \in \mathbb{R}$. Differentiating~\eqref{ellipticpsi} for $\omega=0$ with respect to $\varphi$, and pairing the result with $\partial_\varphi \psi$ leads to the identity
$$
\iint_{{\mathbb S}^2} |\nabla \partial_\varphi \psi|^2 \,{\rm d}\sigma = - \iint_{{\mathbb S}^2} F'(\psi) |\partial_\varphi \psi|^2 \,{\rm d}\sigma\,.
$$
Under the assumption that $F'>-6$, this implies that 
$$
\iint_{{\mathbb S}^2} |\nabla \partial_\varphi \psi|^2 \,{\rm d}\sigma < 6  \iint_{{\mathbb S}^2} |\partial_\varphi \psi|^2 \,{\rm d}\sigma\,,
$$
unless $\partial_\varphi \psi \equiv 0$. Since $\partial_\varphi \mathbb{P}_2 \psi =0$, and $\partial_\varphi$ commutes with the spectral projectors of the Laplacian, this gives $\mathbb{P}_2\partial_\varphi \psi =0$. Therefore, we can use the fact that the third eigenvalue of $-\Delta$ is $6$ to conclude that
$$
\iint_{{\mathbb S}^2} |\Delta \partial_\varphi \psi|^2 \,{\rm d}\sigma 
< 6  \iint_{{\mathbb S}^2} |\partial_\varphi \psi|^2 \,dx \leq \int |\Delta \partial_\varphi \psi|^2 \,{\rm d}\sigma\,,
$$
which is a contradiction unless $\partial_\varphi \psi \equiv 0$.

\medskip

\noindent
(iii) If $\psi$ solves~\eqref{ellipticpsi}, then
$$
\left[ -\partial_\theta \psi \partial_\varphi + \partial_\varphi \psi \partial_\theta \right] \left(\Delta \psi + 2\omega\sin\theta\right) = 0.
$$
Multiplying the above by by ${\rm e}^{{\rm i}\varphi}$ and integrating the result on the sphere, we get
$$
0 = \int_{-\frac{\pi}{2}}^{\frac{\pi}{2}}\int_0^{2\pi} {\rm e}^{{\rm i} \varphi}  (-\partial_\theta \psi \,\partial_\varphi + \partial_\varphi \psi \,\partial_\theta) \left[ \partial_\theta^2 \psi - \tan \theta \, \partial_\theta \psi + \frac{1}{(\cos \theta)^2}\, \partial_\varphi^2 \psi  + 2 \omega \sin \theta \right] \cos \theta\, {\rm d\varphi} {\rm d}\theta= I + II + III + IV,
$$
where $I$, $II$, $III$ and $IV$ correspond to the four terms in the bracketed expression. Integrating by parts repeatedly, and using the fact that all boundary terms vanish since $\cos\theta=\partial_\varphi \psi=0$ at $\theta= \pm \frac{\pi}{2}$, we obtain 
\begin{align*}
& I = \int_{-\frac{\pi}{2}}^{\frac{\pi}{2}}\int_0^{2\pi} {\rm e}^{{\rm i}\varphi}\Big( {\rm i} \sin \theta (\partial_\theta \psi)^2 - \cos \theta \, \partial_\varphi \psi \,\partial_\theta \psi \Big) \, {\rm d\varphi} {\rm d}\theta\,, \\
& II = \int_{-\frac{\pi}{2}}^{\frac{\pi}{2}}\int_0^{2\pi} {\rm e}^{{\rm i}\varphi}\left( - {\rm i} \sin \theta (\partial_\theta \psi)^2 - \frac{(\sin \theta)^2}{\cos \theta} \, \partial_\varphi \psi \,\partial_\theta \psi \right) \, {\rm d\varphi} {\rm d}\theta \\
& III = \int_{-\frac{\pi}{2}}^{\frac{\pi}{2}}\int_0^{2\pi} {\rm e}^{{\rm i}\varphi}\frac{1}{\cos \theta}  \,\partial_\varphi \psi \,\partial_\theta \psi \, {\rm d\varphi} {\rm d}\theta\,,
\end{align*}
from which it follows that 
$$
I+II + III = 0\,.
$$ 
Turning to $IV$, repeated integration by parts show that
$$
IV = -2{\rm i}\omega \int_{-\frac{\pi}{2}}^{\frac{\pi}{2}}\int_0^{2\pi} \psi {\rm e}^{{\rm i}\varphi} (\cos \theta)^2\, {\rm d\varphi} {\rm d}\theta \,.
$$
Furthermore the above implies that $IV = 0$, or in other words
$$
\int_{-\frac{\pi}{2}}^{\frac{\pi}{2}}\int_0^{2\pi}  \psi ( \cos \theta \sin \varphi ) \cos \theta \, {\rm d\varphi} {\rm d}\theta = 
\int_{-\frac{\pi}{2}}^{\frac{\pi}{2}}\int_0^{2\pi}  \psi ( \cos \theta \cos \varphi ) \cos \theta \, {\rm d\varphi} {\rm d}\theta\,,
$$
which is the desired result.

\medskip

\noindent
(iv) Differentiating~\eqref{ellipticpsi} with respect to $\varphi$, and pairing the result with $\partial_\varphi \psi$ leads to the identity
$$
\iint_{{\mathbb S}^2} |\nabla \partial_\varphi \psi|^2 \,{\rm d}\sigma = - \iint_{{\mathbb S}^2} F'(\psi) |\partial_\varphi \psi|^2 \,{\rm d}\sigma\,.
$$
Assuming that $F'>-6$, this implies that 
$$
\iint_{{\mathbb S}^2} |\nabla \partial_\varphi \psi|^2 \,{\rm d}\sigma < 6  \iint_{{\mathbb S}^2} |\partial_\varphi \psi|^2 \,{\rm d}\sigma\,.
$$
However, we know from (iii) that $\partial_\varphi \mathbb{P}_2 \psi =0$. Since $\partial_\varphi$ commutes with the spectral projectors of the Laplacian, this gives $\mathbb{P}_2\partial_\varphi \psi =0$. Therefore, the third eigenvalue of $-\Delta$ being $6$, we conclude that
$$
\iint_{{\mathbb S}^2} |\Delta \partial_\varphi \psi|^2 \,{\rm d}\sigma < 6  \iint_{{\mathbb S}^2} |\partial_\varphi \psi|^2 \,{\rm d}\sigma \leq 
\iint_{{\mathbb S}^2} |\Delta \partial_\varphi \psi|^2 \,{\rm d}\sigma\,,
$$
which is a contradiction unless $\partial_\varphi \psi \equiv 0$.
\end{proof}

Finding explicit smooth solutions of equation~\eqref{ellipticpsi} for nonlinear functions $F$ is far from obvious. The case 
$$F(s)=a\,{\rm e}^{bs} + \frac{2}{b}\qquad\text{with}\quad a\,,b \in {\mathbb R}\setminus \{0\}\,,$$
is of great interest in statistical mechanics and Riemannian geometry (see \cite{ck1,ck2,kaz}), with the specific value of the additive constant 
(given the exponential term) determined by the Gauss constraint. For $ab > 0$ there are no smooth solutions, while for 
$ab<0$ the general solution is available (see \cite{crowdy}), and can be expressed in terms of a single analytic function by performing the stereographic 
projection of the 2-sphere ${\mathbb S}^2$ which maps the North Pole to infinity and the South Pole to the origin of the compactifed complex plane 
${\mathbb C} \cup \{\infty\}$. Unfortunately, singularities are ubiquituous\footnote{A notable exception being the solution
$\psi_0(\varphi,\theta)=\frac{2}{b}\, \ln (1 + |\zeta|^2) + \frac{1}{b}\, \ln\Big( \frac{2\alpha^2}{-ab(1+
|\alpha \zeta + \beta|^2)^2}\Big)$ with free parameters $\alpha \in {\mathbb R} \setminus \{0\}$ and $\beta \in {\mathbb C}$, where 
$\zeta=\tan\big( \frac{\theta}{2} + \frac{\pi}{4}\big)\,{\rm e}^{{\rm i}\varphi} \in {\mathbb C} \cup \{\infty\}$. This solution is generated by the 
polynomial $\zeta \mapsto \alpha \zeta + \beta$ and all streamlines are circles (the streamline pattern being identical to that depicted in Fig. 3). A suitable rotation 
in ${\mathbb SO}(3))$ transforms this solution into a zonal one.}, so that this type of nonlinearity is not suitable for regular 
global flows but rather for regular flows in regions bounded by a streamline. 

\begin{figure}[h]
 \centering
  \includegraphics[width=0.33\textwidth]{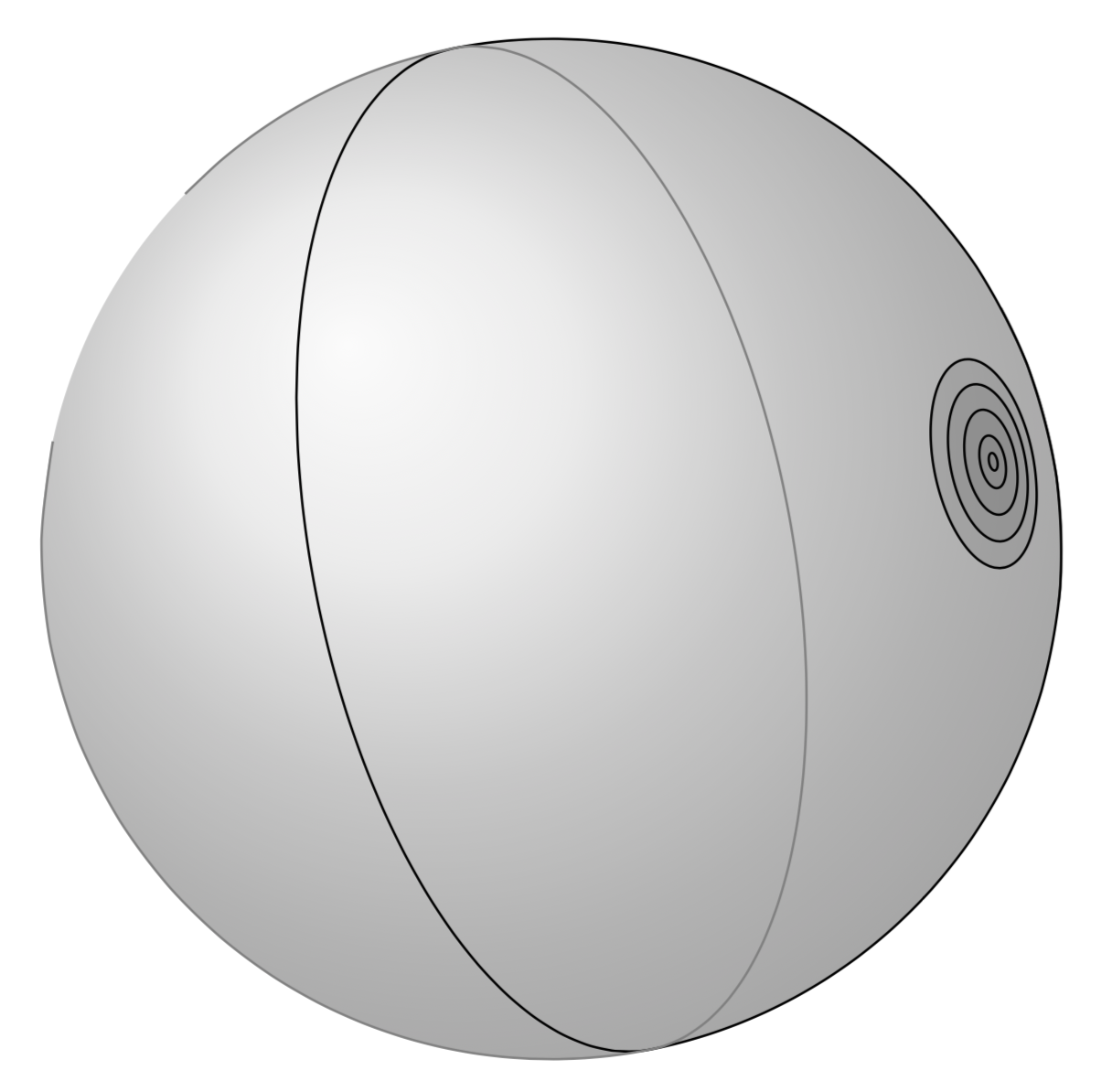}
   \captionsetup{width=.97\linewidth}
  \caption{Depiction of the streamlines for the solution \eqref{ahm2}.}
  \label{fig3}
\end{figure}

Theorem \ref{theo1} permits us to construct nonlinearities $F$ for which the equation~\eqref{ellipticpsi} admits classical non-zonal solutions. 
Denoting for $\alpha \in (0,1)$ and $k \in {\mathbb N} \cup \{0\}$ 
by $C^{k,\alpha}({\mathbb S}^2)$ the Banach space of $k$-times H\"older continuously differentiable functions $\psi: {\mathbb S}^2 \to {\mathbb R}$ 
with exponent $\alpha$, we can write \eqref{ellipticpsi} for $\omega=0$ and any continuously differentiable function $F: {\mathbb R} \to {\mathbb R}$ in the form 
\begin{equation}\label{spde1}
{\mathcal F}(\psi)=0 \quad\text{for 
the nonlinear operator}\quad {\mathcal F}: C^{2,\alpha}({\mathbb S}^2) \to C^{0,\alpha}({\mathbb S}^2)\,,\qquad {\mathcal F}(\psi)= \Delta \psi- F(\psi)\,.
\end{equation}
Note that \eqref{spde1} is ${\mathbb O}(3)$-equivariant:
$${\mathcal F}(G \psi) = G {\mathcal F}(\psi) \quad\text{for all}\quad G \in  {\mathbb O}(3) \quad\text{and all}\quad \psi \in C^{2,\alpha}({\mathbb S}^2)\,,$$
the natural action of ${\mathbb O}(3)$ on functions $\psi: {\mathbb S}^2 \to {\mathbb R}$ being given by
$$(G\psi)(X):= \psi (G^TX)\,,\qquad G \in  {\mathbb O}(3)\,,$$
where $G^T$ is the transpose of the matrix $G$ in the matrix representation of the orthogonal group ${\mathbb O}(3)$. 
Therefore, if $\psi \in C^{2,\alpha}({\mathbb S}^2)$ solves \eqref{spde1}, then so does $G\psi$ for any  $G \in  {\mathbb O}(3)$. In particular, 
since the zonal solutions of \eqref{spde1} are those symmetric with respect to the polar axis, if $G \in  {\mathbb O}(3)$ breaks this symmetry, then $G \psi$ 
is a non-zonal solution of \eqref{spde1} whenever $\psi$ is a non-constant zonal solution. We may take $G$ to be the rotation that transforms 
rotations about the polar axis into rotations about a fixed horizontal axis. By computing the expression 
$$\frac{{\rm d}^2 f}{{\rm d}\theta^2} - \tan(\theta) \,\frac{{\rm d} f}{{\rm d}\theta}$$ 
for a suitable function of the variable $\sin\theta$, we obtain specific zonal solutions of \eqref{spde1}. 
For example, one can check that for every $\varepsilon >0$ the function 
$$\psi_0(\theta)= \ln\Big(\frac{1+\varepsilon \sin\theta}{1-\varepsilon \sin\theta} \Big)\,,\qquad -\frac{\pi}{2} \le \theta \le \frac{\pi}{2}\,,$$
is a zonal solution of \eqref{ellipticpsi} with $\omega=0$, for 
\begin{equation}\label{fahm1}
F(\psi)= - \frac{1-\varepsilon^2}{2} \, [2 \sinh(\psi) + \sinh(2\psi)]\,.
\end{equation}
Consequently, 
\begin{equation}\label{ahm1}
\psi(\varphi,\theta)= \ln [1 + \varepsilon \cos^2(\theta) \sin^2(\varphi-\varphi_0)]
\end{equation}
is a non-zonal solution of \eqref{spde1} for every fixed $\varphi_0 \in [0,2\pi)$. 
While the solution \eqref{ahm1} is known (see \cite{ahm}), the above symmetry approach 
rather than the {\it ad-hoc Ansatz} made in \cite{ahm} explains how it comes about and also provides a procedure leading to the construction of other explicit non-zonal solutions. Indeed, the zonal solution 
$$\psi_0(\theta)={\rm e}^{\varepsilon \sin\theta} -1\,,\qquad -\frac{\pi}{2} \le \theta \le \frac{\pi}{2}\,,$$
of \eqref{spde1} for
\begin{equation}\label{fahm2}
F(\psi)=\varepsilon^2(1+ \psi) - (1+\psi) \ln^2(1+\psi) - 2(1+\psi) \ln(1+\psi)\,.
\end{equation}
leads to the non-zonal solution
\begin{equation}\label{ahm2}
\psi(\varphi,\theta)= {\rm e}^{\varepsilon \cos\theta \sin(\varphi-\varphi_0)} - 1
\end{equation}
for every fixed $\varphi_0 \in [0,2\pi)$. Note that the streamlines of the solutions \eqref{ahm1}-\eqref{ahm2} are circles 
with collinear centres along a segment lying in the equatorial 
plane and passing through the centre of the sphere (obtained by suitably rotating Figure \ref{fig3}), since by passing to spherical coordinates in 
${\mathbb R}^3$ we see that the level sets $[\cos\theta \sin(\varphi-\varphi_0)=d]$ are precisely the points on the sphere at distance $|d|$ 
from the line obtained rotating the $x$-axis by $\varphi_0$-degrees in the equatorial plane.

\section{Stability of stationary solutions}

\label{SOSS}

To gain insight into the complicated dynamics of the atmosphere it is advantageous to regard it as a background zonal flow presenting 
fluctuations, some transient or at least unable to develop significantly while other are enhanced in time (instabilities). One distinguishes between 
barotropic instability, in which perturbations draw strength at the expense of the kinetic energy of the background flow, and baroclinic instability, 
fed by the release potential energy through the lifting (sinking) of warm, relatively light (respectively of cold, relatively dense) fluid. Barotropic 
flows are prevalent in the stratosphere of the outer planets, so that the investigation of the stability of the explicitly known stationary 
solutions of \eqref{Eomega} is of great physical relevance. 

\subsection{Linear stability of zonal flows}

The giant planet atmospheres feature alternating prograde (eastward) and retrograde (westward) jets of different speeds and widths (see Fig. \ref{fig2}). 
Moreover, observations of the zonal flow patterns of Jupiter and Saturn indicate the development of eddies around the peaks of the westward jets. 
This type of instability is reminiscent of Rayleigh stability criterion for shear flows.

\label{ZF1}

Consider a zonal flow with a stream function $\psi_0(\theta)$ of class $C^3$, representing a stationary solution 
of the vorticity equation \eqref{Eomega}. To study perturbations of this zonal flow, it is convenient to use the 
variable $s=\sin\theta \in [-1,1]$ instead of the latitude $\theta$. Note that the conjugate variable to the longitude $\varphi$ is not the latitude 
$\theta$ but the axial component in Cartesian coordinates, $s=\sin\theta$, 
for which Hamilton's equations hold:
$$\begin{cases}
D_t\varphi = -\Psi_s\,,\\
D_t s = \Psi_\varphi\,.
\end{cases}$$ Setting
\begin{equation}\label{nots}
\Psi(\varphi,s)=\psi(\varphi,\theta)  \,,
\end{equation}
the expressions for the velocity and vorticity are
\begin{equation}\label{vvs}
u=- \sqrt{1-s^2}\,\partial_s\Psi\,,\qquad v= \frac{1}{\sqrt{1-s^2}}\,\partial_\varphi\Psi\,,\qquad \Omega=(1-s^2)\,\partial_s^2\Psi  - 2s\,\partial_s\Psi + 
\frac{1}{1-s^2}\,\partial_\varphi^2 \Psi:=\Upsilon(\varphi,s,t)\,,
\end{equation}
and \eqref{Eomega} takes the form
\begin{equation}\label{Eomegas}
\partial_t \Upsilon + \partial_\varphi\Psi\, (\partial_s \Upsilon + 2\omega) - \partial_s\Psi\,\partial_\varphi\Upsilon=0\,.
\end{equation}
The linearization of the vorticity equation \eqref{Eomegas} around the zonal flow with stream function $\Psi_0(s)=\psi_0(\theta)$ and vorticity $\Upsilon_0$ is the equation
\begin{equation}\label{zlin}
\partial_t \Upsilon -\Psi_0'\, \partial_\varphi \Upsilon + (\Upsilon_0'+2\omega)\, \partial_\varphi \Psi=0
\end{equation}
for the infinitesimal perturbation $\Psi(\varphi,s,t)$ with vorticity $\Upsilon$, where $\Psi_0'=\partial_s\Psi_0$ and $\Upsilon_0'=\partial_s\Upsilon_0$. In accordance to the discussion in 
the preamble of relation \eqref{bu}, the boundary conditions associated to \eqref{zlin} are
\begin{equation}\label{zlinbc}
\partial_\varphi\Psi=0 \quad\text{at}\quad s=\pm 1\,.
\end{equation}
We can write \eqref{zlin} in the form 
\begin{equation}\label{zlins}
\partial_t \Upsilon = {\mathcal L}_0 \Upsilon\,,
\end{equation}
with the linear operator
$${\mathcal L}_0=\Psi_0'\, \partial_\varphi - (\Upsilon_0'+2\omega)\, \partial_\varphi \Delta^{-1}$$
acting in the space 
$$L_0^2=\{\Upsilon \in L^2({\mathbb S}^2): \, \int_0^{2\pi}\int_{-1}^1 \Upsilon(\varphi,s)\,{\rm d}s{\rm d}\varphi=0\}$$
of vorticities subject to the Gauss constraint \eqref{gaussc}. Since the operator ${\mathcal L}_0$ commutes with translations in 
the azimuthal direction, by means of the Fourier modes
$$f(\varphi,s,t)=\sum_{k \in {\mathbb Z} \setminus \{0\}} f_k(s,t) \,{\rm e}^{{\rm i}k \varphi}\,,$$
for $f= \Delta^{-1}\Upsilon$, we can decompose \eqref{zlins} into 
$$\partial_t \Upsilon_k = {\rm i}k\,{\mathcal L}_0^k \Upsilon_k\,,\qquad k \in {\mathbb Z} \setminus \{0\}\,,$$
where the operators
$${\mathcal L}_0^k=\Psi_0' - (\Upsilon_0'+2\omega)\, \Delta_k^{-1}$$
act in $L^2[-1,1]$ subject to the boundary conditions $\mu(s)=0$ at $s=\pm 1$ for $\mu=\Delta_k^{-1}\Upsilon_k$, which ensure that
$$ \Upsilon_k = \Delta_k \mu=(1-s^2) \mu'' - 2s \mu' - \tfrac{k^2}{1-s^2}\,\mu\,,\qquad k \in {\mathbb Z} \setminus \{0\}\,,$$
is regular at $s=\pm 1$. Since for every $k \in {\mathbb Z} \setminus \{0\}$ the operator ${\mathcal L}_0^k$ is a compact perturbation of the multiplication operator $\Psi_0'$ with the purely essential spectrum 
$$\Sigma_0=\Big[ \inf_{ s \in [-1,1]} \{  \Psi_0'(s)\}\,,\, 
\sup_{ s \in [-1,1]} \{  \Psi_0'(s)\}\Big]\,,$$
the essential spectrum of ${\mathcal L}_0^k$ coincides with the closed real interval $\Sigma_0$, and the 
rest of the spectrum of ${\mathcal L}_0^k$ consists of at most countably many isolated eigenvalues of finite multiplicity. The discrete spectrum of ${\mathcal L}_0^k$ is symmetric about the real axis, 
since the complex conjugate of an eigenfunction for the eigenvalue $\lambda \in {\mathbb C} \setminus {\mathbb R}$ is an eigenfunction for $\overline{\lambda}$. The Fourier mode 
decomposition thus yields the spectrum of the operator ${\mathcal L}_0$: the essential spectrum 
$$\{\lambda \in {\mathbb C}: \ \lambda ={\rm i}kr \quad\text{with}\quad k \in {\mathbb Z} \setminus \{0\} \quad\text{and}\quad r \in \Sigma_0\}
$$ is located on the imaginary axis, and the discrete spectrum is symmetric about the imaginary axis and comprises at most countable many isolated eigenvalues of finite multiplicity. Linear instability 
thus amounts to the existence of an eigenvalue of ${\mathcal L}_0$ with non-zero real part.  Indeed, due to the symmetry of the discrete spectrum 
about the imaginary axis, this means the existence of an eigenvalue $\xi={\rm i}k \lambda$ with strictly positive real part, for some $k \in {\mathbb Z} \setminus \{0\}$ and 
some eigenvalue $\lambda \not \in \Sigma_0$ of ${\mathcal L}_0^k$. If $\Upsilon_k(s)$ is the 
corresponding eigenfunction of ${\mathcal L}_0^k$, then $\Upsilon_k(s)\,{\rm e}^{{\rm i}k (\varphi + \lambda t)}$ solves \eqref{zlins} and its amplitude grows indefinitely for $t \to \infty$. 

The spectral analysis is very challenging, and only few thorough investigations appear to be available:
\begin{itemize} 
\item  For $\psi_0(\theta)=\alpha \sin\theta$ with $\alpha \in {\mathbb R} \setminus \{0\}$, the essential spectrum of ${\mathcal L}_0^k=\alpha \partial_\varphi + 2(\omega - \alpha)\partial_\varphi \Delta_k^{-1}$ 
is the single point $\alpha$ and by expanding in spherical harmonics we see that the 
discrete spectrum consists of the real eigenvalues $\Big\{\Big(\alpha - \frac{2(\alpha -\omega)}{|k|(|k|+1)}\Big)\Big\}$ of multplicity $2|k|$. This zonal flow is thus linearly stable.
\item The zonal flows 
\begin{equation}\label{zrhw}
\psi_0(\theta)=\alpha Y_1^0(\theta) + \beta Y_2^0(\theta)\,,
\end{equation}
where $\alpha,\,\beta \in {\mathbb R} \setminus \{0\}$ and $Y_1^0(\theta)=\sqrt{\tfrac{3}{4\pi}}\,\sin\theta$, $Y_2^0(\theta)=\sqrt{\tfrac{5}{16\pi}}\,(3\sin^2\theta -1)$, are the zonal spherical harmonics of degree $1$ and $2$ (see the Appendix), turn out to be linearly stable (see \cite{skiba, Taylor}).
\item The zonal flows
\begin{equation}\label{zrhw3}
\psi_0(\theta)=\beta Y_3^0(\theta)\,, \qquad \beta \in {\mathbb R} \setminus \{0\}\,,
\end{equation}
with three jets (defined as extrema of the azimuthal velocity, $-\partial_\theta \psi_0$, and equal to the number of nodes of the zonal spherical harmonics $Y_3^0$ of degree $3$) are known 
(see \cite{skiba, Taylor}) to be linearly unstable if their amplitude $|\beta|$ exceeds a critical value $\beta^\ast(\omega)$.
\item Numerical simulations (see \cite{ben}) indicate that the zonal flows
\begin{equation}\label{zrhwj}
\psi_0(\theta)=\beta Y_j^0(\theta)\,, \qquad \beta \in {\mathbb R} \setminus \{0\}\,,
\end{equation}
with $j > 3$ jets are linearly unstable if their amplitude $|\beta|$ is sufficiently large.
\end{itemize} 

Let us now briefly describe a classical approach -- typically pursued within the setting of flat-space geometry (see \cite{mp}) -- that provides insight into the linear 
stability without actually finding eigenvalues. Motivated by the considerations made above, 
we consider normal mode perturbations of zonal flows of the form
\begin{equation}\label{nmp}
\Psi(\varphi,s,t)= \Phi(s)\,{\rm e}^{{\rm i} k (\varphi - \mu t)}\,,
\end{equation}
subject to the boundary condition \eqref{zlinbc}, where $k \in {\mathbb Z} \setminus \{0\}$ is the wave number and $\mu \in {\mathbb C}$ is the wave speed. The zonal flow with stream function $\Psi_0$ is 
linearly unstable if there exists a nontrivial solution \eqref{nmp} to \eqref{zlin} and \eqref{zlinbc} with $\frak{Im}(k\mu) > 0$, since 
in this case the amplitude of the perturbation \eqref{nmp} grows indefinitely with time. For $\frak{Im}(\mu) \neq 0$ the equation 
\eqref{zlin} reduces to 
\begin{equation}\label{raye}
\Big( (1-s^2) \Phi'  \Big)' - \Big\{\frac{k^2}{1-s^2} + \frac{\Upsilon_0'+2\omega}{\Psi_0'+\mu}\Big\}\,  \Phi=0\,,\qquad s \in (-1,1)\,,
\end{equation}
for the amplitude $\Phi$. Multiplying \eqref{raye} by the complex-conjugate $\overline{\Phi}$ and integrating the outcome on 
$[-1,1]$, an integration by parts yields
\begin{equation}\label{rayei}
\int_{-1}^1  \Big\{ (1-s^2) |\Phi'(s)|^2 + \frac{k^2}{1-s^2}\,|\Phi(s)|^2\Big\} \,{\rm d}\theta 
+ \int_{-1}^1  \frac{\Big(\Upsilon_0'(s)+2\omega\Big)\Big(\Psi_0'(s) + \overline{\mu}\Big)}{|\Psi_0'(s) +\mu|^2} \,|\Phi(s)|^2\,{\rm d}s =0\,.
\end{equation}
Since the imaginary part of \eqref{rayei} is
$$-\frak{Im}(\mu)\,\int_{-1}^1  \frac{\Upsilon_0'(s)+2\omega}{|\Psi_0'(s)+\mu|^2} \,|\Phi(s)|^2\,{\rm d}s\,,$$
we obtain the Rayleigh necessary condition for linear instability 
\begin{equation}\label{raynec}
\int_{-1}^1  \frac{\Upsilon_0'(s)+2\omega}{|\Psi_0'(s) +\mu|^2} \,|\Phi(s)|^2\,{\rm d}s =0\,,
\end{equation}
which requires $\Big(\Upsilon_0'(s)+2\omega\Big)$ to change sign on $(-1,1)$. A further necessary condition, 
due to Fjortoft, is obtained by taking also into account the real part of \eqref{rayei}: if \eqref{raynec} holds, then \eqref{rayei} yields
\begin{equation}\label{fji}
\int_{-1}^1  \frac{\Big(\Upsilon_0'(s)+2\omega\Big)\Psi_0'(s)}{|\Psi_0'(s)+\mu|^2} \,|\Phi(s)|^2\,{\rm d}s=-\int_{-1}^1  \Big\{ (1-s^2) |\Phi'(s)|^2 + \frac{k^2}{1-s^2}\,|\Phi(s)|^2\Big\} \,{\rm d}\theta <0\,.
\end{equation}
For any $\gamma \in {\mathbb R}$, adding \eqref{raynec} multiplied by $\gamma$ to \eqref{fji}, we get
$$\int_{-1}^1  \frac{\Big(\Upsilon_0'(s)+2\omega\Big)\Big(\Psi_0'(s) + \gamma\Big)}{|\Psi_0'(s) +\mu|^2} \,|\Phi(s)|^2\,{\rm d}s=-\int_{-1}^1  \Big\{ (1-s^2) |\Phi'(s)|^2 + \frac{k^2}{1-s^2}\,|\Phi(s)|^2\Big\} \,{\rm d}\theta <0 \,.$$
Consequently, we must have  $\Big(\Upsilon_0'(s_0)+2\omega\Big)\Big(\Psi_0'(s_0) + \gamma\Big) <0$ 
at some $s_0 \in (-1,1)$. In terms of latitude, these necessary conditions for the linear instability of a zonal flow with azimuthal velocity $U_0$ and vorticity $\Omega_0$ read:
\begin{itemize}
\item ({\it Rayleigh's criterion}) the meridional gradient of the total vorticity of the zonal flow, 
$$\Omega_0'(\theta)+2\omega\cos\theta=\Big(\Upsilon_0'(\sin\theta)+2\omega\Big)\cos\theta\,,$$ 
changes sign on the interval $\big(-\tfrac{\pi}{2},\tfrac{\pi}{2}\big)$;
\item ({\it Fjortoft's criterion}) for every $\gamma \in {\mathbb R}$ we must have   
$$\Big(\Omega_0'(\theta)+2\omega \cos\theta\Big)\Big( \frac{U_0(\theta)}{\cos\theta}- \gamma\Big)>0$$ 
at some $\theta_0 \in \big(-\tfrac{\pi}{2},\tfrac{\pi}{2}\big)$.
\end{itemize}
Note that Fjortoft's criterion implies that of Rayleigh since the existence of real numbers $\gamma_1>\gamma_2$ with 
$$\frac{U_0(\theta)}{\cos\theta} -\gamma_1 < 0 < \frac{U_0(\theta)}{\cos\theta} -\gamma_2\,,\qquad \theta \in \big(-\tfrac{\pi}{2},\tfrac{\pi}{2}\big)\,,$$ 
combined with Fjortoft's constraint for $\gamma=\gamma_1$ and $\gamma=\gamma_2$ ensures that the expression $\Omega_0'(\theta)+2\omega \cos\theta$ 
has opposite signs at two points in $\big(-\tfrac{\pi}{2},\tfrac{\pi}{2}\big)$. 
Both criteria fail for the linearly stable zonal flow $\psi_0(\theta)=\alpha\theta$ with $\alpha \in {\mathbb R} \setminus \{0\}$ since in this case
$$\Omega_0'(\theta)+2\omega \cos\theta=2(\omega- \alpha)\cos\theta\,,\qquad \frac{U_0(\theta)}{\cos\theta} - \gamma =-\alpha-\gamma\,.$$
However, they are generally far from sufficient to ensure linear instability: e.g., both hold for the linearly stable zonal flow $\psi_0(\theta)=\frac{\omega}{3}\,\sin^2\theta$, with
$$\Omega_0'(\theta)+2\omega \cos\theta=2\omega \cos\theta ( 1- 2\sin\theta)\,,\qquad \frac{U_0(\theta)}{\cos\theta} - \gamma =-\frac{2\omega}{3}\, \Big(\sin\theta + \frac{3\gamma}{2\omega}\Big)\,.$$
Rayleigh's criterion appears not to be sufficient for linear instability: it holds for the latitudinal profile of the persistent zonal jets on Jupiter and Saturn (see the data in \cite{read}).

\subsection{Nonlinear stability of stationary solutions: the Arnold approach} Due to the intricate nature of the investigation of linear stability and to its rather inconclusive nature\footnote{The question whether linearization is conclusive with 
respect to the stability of solutions to the nonlinear vorticity equation is open.}, we pursue the issue of nonlinear stability directly, relying on linear results (whenever available) to make an informed guess.

We consider a solution $\psi_0$ of~\eqref{ellipticpsi0}, which is thus a stationary solution of \eqref{Eomega}.  We implement Arnold's method (see \cite{Arnold1, Arnold2}) by defining the functional
$$
\mathcal{E}(\psi) = \iint_{{\mathbb S}^2} \left[ \frac{1}{2} |U_0 + U|^2 + K( \Omega_0 + \Omega +2\omega \sin \theta) + A \sin \theta \Omega + B (\mathbb{P}_{1} \psi)^2 \right]\,{\rm d}\sigma\,,
$$
which is a sum of conserved quantities for the flow of \eqref{Eomega}. Our aim is to choose the function $K$ and the constants $A$ and $B$ so that $\psi=0$ is a critical point of $\mathcal{E}$, and, 
furthermore, the second variation of $\mathcal{E}$ at $0$ is a definite quadratic form.  

With this in mind, we now compute the  first variation of $\mathcal{E}$ at $0$:
\begin{align*}
d\mathcal{E}_{0} (\delta {\psi}) & =\iint_{{\mathbb S}^2} \left[  U_0 \cdot \delta  U + K'( \Omega_0 + 2 \omega \sin \theta) \delta  \Omega + A \sin \theta \, \delta   \Omega  \right]\,{\rm d}\sigma\,.
\end{align*}
Expressing $\delta  \Omega$ in terms of $\delta  U$ and integrating by parts using that $\operatorname{grad} \sin \theta = \cos \theta \, \mathbf{e}_\theta$, this becomes
\begin{align*}
d\mathcal{E}_{{ \psi}} (\delta {\psi}) & = \iint_{{\mathbb S}^2}  \left[  U_0 \cdot \delta   U - K'( \Omega_0 + 2 \omega \sin \theta) \operatorname{div} J \delta  U  - A \sin \theta \operatorname{div} J \delta  U  \right]\, {\rm d}\sigma\\
& =  \iint_{{\mathbb S}^2}  \left[ U_0 \cdot \delta  U - K''( \Omega_0 + 2 \omega \sin \theta) J \operatorname{grad} ( \Omega_0 + 2 \omega \sin \theta) \cdot \delta  U   + A \cos \theta \, \mathbf{e}_\varphi \cdot \delta   U   \right]\,{\rm d}\sigma \,.
\end{align*}
The stationary solution $\psi_0$ being such that $ \Omega_0 + 2 \omega \sin \theta = F(\psi_0)$ implies, after taking the gradient, that
$$
\operatorname{grad}  ( \Omega_0 + 2 \omega \sin \theta) = F'(\psi_0) \operatorname{grad} \psi_0 =  - F'(\psi_0) ( J U_0 + \omega \cos \theta \,\mathbf{e}_\theta).
$$
Therefore
$$
d\mathcal{E}_{\psi_0} (\delta {\psi}) = \iint_{{\mathbb S}^2} \delta {U} \cdot U_0 \left[  1 - K''(F(\psi_0)) F'(\psi_0) \right] \,dx + \int [A - \omega K''(F(\psi_0)) F'(\psi_0) ] \cos \theta\, \mathbf{e}_\varphi \cdot \delta {U} \,{\rm d}\sigma\,.
$$
It remains to choose $K''(F(x)) F'(x)=1$ and $A = \omega$ to obtain that the first variation is zero.

Turning to the second variation,
$$
d^2\mathcal{E}_{{ \psi_0}} (\delta {\psi}) =  \iint_{{\mathbb S}^2} \left[ |\delta U|^2 + \frac{1}{F'(\psi_0)} (\delta \Omega)^2 + 2 B (\mathbb{P}_{1} \delta \psi)^2 \right]\,{\rm d}\sigma\,,
$$
we expand $\psi$ in spherical harmonics: letting $\mathbb{P}_k$ denote the projector on the $k$-th eigenspace of $\Delta$ associated to the eigenvalue $-k(k+1)$, we write
$$
\delta \psi = \sum_{k \geq 1} \alpha_k \,\mathbb{P}_k \delta \psi \,,
$$
since we can subtract a constant from $\psi$ to ensure that $\mathbb{P}_0 \psi = 0$. Then
\begin{align*}
\label{dsquaredE}
d^2\mathcal{E}_{0}(\delta \psi) & = \iint_{{\mathbb S}^2} \left[ \sum_{k\geq 1} k(k+1)\, |\mathbb{P}_k \delta \psi|^2 + \frac{1}{F'(\psi_0)} \sum_{k\geq 1}k^2(k+1)^2 |\mathbb{P}_k \delta \psi|^2 + 2B\, |\mathbb{P}_1 \delta \psi|^2 \right] \,{\rm d}\sigma \\
& = \iint_{{\mathbb S}^2}\left[ \left( 2 + \frac{4}{F'(\psi_0)} + 2B\right) (\mathbb{P}_1 \delta \psi)^2 + \sum_{k \geq 2}\left( k(k+1) + \frac{k^2(k+1)^2}{F'(\psi_0)} \right) (\mathbb{P}_k \delta \psi)^2 \right] \, {\rm d}\sigma\,.
\end{align*}
The question of the coercivity of $d^2 \mathcal{E}_0$ reduces to determining eigenvalues of the Schr\"odinger operator $- \frac{\Delta}{F'(\psi)} + 1$ on the orthogonal complement of $\mathbb{E}_1$. We 
distinguish several cases, according to the the range of $F'(\psi_0)$:
\begin{itemize}
\item If $F'>0$, the quadratic form is positive-definite if one chooses $B=0$, for instance. This corresponds to Type II Arnold stable solutions (see Subsection~\ref{MR} for the definition). 
However, the condition $F'>0$  is only satisfied by constant solutions (see Theorem~\ref{theo1}).
\item If $-6 < F'(\psi_0) < 0$, then the quadratic form is negative-definite: indeed, $k(k+1) +  \frac{k^2(k+1)^2}{F'(\psi_0)} > 0$ for all $k \geq 2$, and the mode $k=1$ can be handled by choosing $B=-10$. This 
corresponds to Type I Arnold stable solutions, with the difference that the first eigenvalue of the Laplacian is replaced by the second eigenvalue -- a modification related to conservation laws which are 
specific to the sphere. Note that, due to Theorem~\ref{theo1}, the condition $-6 < F'(\psi_0) < 0$ forces solutions to be zonal up to a rotation.
\end{itemize}
Combining the above considerations with standard arguments leads to the following result.

\begin{theorem} \label{thmarnold}
For $0> F' > -6$ any solution $\psi_0$ of~\eqref{ellipticpsi0} is stable in $H^2({\mathbb S}^2)$: if $\widehat{\psi}(t)$ is the solution of~\eqref{Eomega} with initial data $\psi(0)=\tilde{\psi}_0$, then
$
\| \widehat{\psi}(t) - \psi_0 \|_{H^2} \lesssim \| \widehat{\psi}_0 - \psi_0 \|_{H^2}\,.
$
\end{theorem}

The limiting case in the above theorem is given by $F' = -6$, which corresponds to Rossby-Haurwitz solutions in $\mathbb{E}_1 + \mathbb{E}_2$. They will be the focus of Section 5 below.

Theorem~\ref{thmarnold} applies to the explicit stationary solutions discussed in Section 3. Indeed, since 
\eqref{fahm1}-\eqref{ahm1} yield
$$\frac{1}{F'(\psi)}= -\,\frac{1}{2(1-\varepsilon^2)}\,\frac{[1 - \varepsilon^2 \cos^2(\theta) \sin^2(\varphi-\varphi_0)]^2}{1 +3 \,\varepsilon^2 \cos^2(\theta) \sin^2(\varphi-\varphi_0)}\,,$$
while \eqref{fahm2}-\eqref{ahm2} lead to
$$\frac{1}{F'(\psi_0)}= -\,\frac{1}{2-\varepsilon^2 + 4\varepsilon  \cos(\theta) \sin(\varphi-\varphi_0) + \varepsilon^2 \cos^2(\theta) \sin^2(\varphi-\varphi_0)}\,,$$
we see that for $\varepsilon>0$ small enough the solutions \eqref{ahm1} and \eqref{ahm2} are stable. Regarding the physical relevance of 
these stable stationary solutions, note that with the exception of
the equatorial regions containing a broad eastward zonal jet, vortices
are generally found on Jupiter and Saturn at all latitudes, preferentially in regions of westward zonal flow (see the data in \cite{yada}).

\subsection{Nonlinear stability of zonal flows}

\subsubsection{Arnold's approach}

\label{subsectionzonal}

\label{ZF2}

We now investigate the nonlinear stability of smooth zonal flows $\psi_0 = f(\theta)$, with associated azimuthal velocity and vorticity given by
$$
U_0 = - f'(\theta) \mathbf{e}_\varphi\,, \qquad \Omega_0 = f''(\theta) - \tan \theta  f'(\theta) = g(\theta)\,,
$$
respectively. The main result (Theorem~\ref{theoremzonal}) was proved in~\cite{Taylor} in greater generality (for rotating surfaces that are not necessarily spherical), but 
we include it here for ease of reference, along with a more straightforward proof in the case of the sphere.

\begin{theorem} \label{theoremzonal}
If there exists $\epsilon,\, A \in {\mathbb R}$ such that 
\begin{equation}\label{szc}
\left| \frac{f'(\theta) - A \cos \theta}{g'(\theta)} \right| > \epsilon >0 \qquad \mbox{on $\quad\displaystyle \left(-\frac{\pi}{2},\frac{\pi}{2} \right)$}\,,
\end{equation}
stable in $H^2({\mathbb S}^2)$: if $\widehat{\psi}(t)$ is the solution of~\eqref{Eomega} with initial data $\psi(0)=\widehat{\psi}_0$, then $\| \widehat{\psi}(t) - \psi_0 \|_{H^2} \lesssim \| \widehat{\psi}_0 - \psi_0 \|_{H^2}$.
\end{theorem}

Note that the condition \eqref{szc} is satisfied if, for instance,
\begin{itemize}
\item $|g'(\theta)|>0$ on $(-\frac{\pi}{2},\frac{\pi}{2})$,
\item $|f'(\theta)| \lesssim |\theta \pm \frac{\pi}{2}|$ for $\theta$ close to $\mp \frac{\pi}{2}$,
\item $|g'(\theta)| \gtrsim |\theta \pm \frac{\pi}{2}|$ for $\theta$ close to $\mp \frac{\pi}{2}$,
\end{itemize}
hold simultaneously.

\begin{proof} We set
$$
\mathcal{E}(\psi) = \iint_{{\mathbb S}^2} \left[ \frac{1}{2} | U|^2 + K( \Omega) + A \sin \theta \Omega \right]\,{\rm d}\sigma\,.
$$
Then, since $\operatorname{grad} \Omega_0 = g'(\theta) \mathbf{e}_\theta$, we have
\begin{align*}
d\mathcal{E}_{{ \psi_0}} (\delta {\psi}) & =  \iint_{{\mathbb S}^2} \left[  U_0 \cdot \delta U + K'( \Omega_0) \delta  \Omega  + A \sin \theta \delta \Omega \right]\,{\rm d}\sigma \\
& =   \iint_{{\mathbb S}^2}  \left[ U_0 \cdot \delta  U - K''( \Omega_0) J \operatorname{grad}  \Omega_0 \cdot \delta  U + A \operatorname{grad} \sin  \theta  \cdot J \delta U \right]\,{\rm d}\sigma \\
& = \iint_{{\mathbb S}^2} \left[ -f'(\theta) + K''(g(\theta)) g'(\theta) + A \cos \theta \right] \mathbf{e}_\varphi \cdot \delta  U_0 \, {\rm d}\sigma\,.
\end{align*}
In order for this first variation to vanish, we choose $K$ such that $-f'(\theta) + K''(g(\theta)) g'(\theta)=-A \cos \theta$. The second variation of $\mathcal{E}$ is then
$$
d^2\mathcal{E}_{{ \psi_0}} (\delta {\psi}) =  \iint_{{\mathbb S}^2} \left[ |\delta U|^2 + \frac{f'(\theta) - A \cos \theta}{g'(\theta)} (\delta \Omega)^2 \right]\,{\rm d}\sigma\,,
$$
which ensures stability.
\end{proof}

\subsubsection{Application to the atmospheres of Uranus and Neptune}
To illustrate the applicability of Theorem~\ref{theoremzonal}, let us now investigate the zonal wind profiles of Uranus and Neptune, consisting of one broad retrograde equatorial jet flanked by two 
prograde jets at higher latitudes (see Figure \ref{fig2}). The zonal flow is symmetric about the Equator for both planets, but there are noticeable differences of the latitudinal flow profiles:
\begin{itemize}
\item on Uranus the equatorial jet is located within the latitude band between 30$^\circ$N and  30$^\circ$S, while on 
Neptune it extends over 50$^\circ$;
\item the prograde/retrograde (eastward/westward) zonal flows on Uranus, measured relative to the planet's rotation speed about its polar axis, peak at about 200 m/s, respectively at 80 m/s, the corresponding values for Neptune being about 200 m/s and 400 m/s, respectively.
\end{itemize}
Recalling \eqref{changeofframe}, if the latitudinal profile of the zonal flow with respect to the rotation about the planet's polar axis (with zonal velocity $\theta \mapsto \omega\,\cos\theta$), 
depicted in Figure \ref{fig2}, is given by the function
\begin{equation}\label{zf}
U_0(\theta)=\alpha\, \cos^5\theta + \beta\, \cos^3\theta + \gamma\, \cos\theta\,,\qquad \theta \in \Big( -\frac{\pi}{2},\frac{\pi}{2}\Big)\,,
\end{equation}
for some real constants $\alpha \neq 0$, $\beta$ and $\gamma$, then
$$U_0(\theta)=- f'(\theta) - \omega\, \cos\theta \,,\qquad \theta \in \Big[ -\frac{\pi}{2},\frac{\pi}{2}\Big]\,,$$
with the notation of Theorem \ref{theoremzonal}. We can now compute
\begin{equation}\label{zf2}
\frac{-f'(\theta) + A\,\cos\theta}{g'(\theta)}= \frac{\cos^4\theta + \frac{\beta}{\alpha}\,\cos^2\theta + \frac{\gamma-\omega +A}{\alpha}}{30\Big(\cos^4\theta + \frac{2(\beta-2\alpha)}{5\alpha}\,\cos^2\theta + \frac{\gamma-\omega -8\beta}{15\alpha}\Big)}\,,\qquad 
 \theta \in \Big( -\frac{\pi}{2},\frac{\pi}{2}\Big)\,.
 \end{equation}
so that the stability criterion provided in Theorem \ref{theoremzonal} applies if the quadratic polynomials 
\begin{equation}\label{zf3}
x^2 + \frac{\beta}{\alpha}\,x + \frac{\gamma-\omega +A}{\alpha} \quad\text{and}\quad x^2 + \frac{2(\beta-2\alpha)}{5\alpha}\,x + \frac{\gamma-\omega -8\beta}{15\alpha}
\end{equation}
have the same roots in the interval $(0,1)$. 
\begin{itemize}
\item Since the unit of the non-dimensional zonal speed scaling for Uranus corresponds to 150 m/s (see the first 
table in Section 7), the latitudinal profile of the zonal flow on Uranus depicted in Figure \ref{fig2} is well-approximated by a function of the form \eqref{zf} if we require $U_0$ to 
have the minimum $U_0(0)=-\frac{8}{15}$ and the maximum $U_0(\frac{\pi}{3})=\frac{4}{3}$ on $\big[0,\frac{\pi}{2}\big]$, these being the non-dimensional counterparts of an eastward speed of 
80 m/s and a westward speed of 200 m/s, respectively. The condition that $\theta=\frac{\pi}{3}$ is a critical
 point and the above two specific values of the function $U_0$ yield the linear system
$$\begin{cases}
&\alpha + \beta + \gamma = -\frac{8}{15}\,,\\
& \tfrac{5}{16}\,\alpha + \tfrac{3}{4}\,\beta + \gamma = 0\,,\\
& \tfrac{1}{32}\,\alpha + \tfrac{1}{8}\,\beta + \tfrac{1}{2}\,\gamma = \tfrac{4}{3}\,,
\end{cases}$$
whose unique solution is 
$$\alpha=\frac{64}{45}\,, \quad\beta=-\frac{272}{45}\,,\quad \gamma=\tfrac{184}{45}\,.$$
With $\omega=18$ the relevant value for Uranus (see the first 
table in Section 7), we obtain that the second quadratic polynomial from \eqref{zf3}, expressed in terms of $x=\cos^2\theta$, is 
$$x \mapsto x^2-\tfrac{5}{2}\,x + \tfrac{155}{96}\,,$$
with no real roots. Choosing $A \in {\mathbb R}$ so that the first quadratic polynomial in \eqref{zf3} has no roots, we can apply Theorem \ref{theoremzonal} 
and we conclude that the zonal flow pattern of Uranus is stable. 
\item For Neptune the non-dimensional unit for the zonal speed corresponds to 200 m/s (see the first 
table in Section 7), the latitudinal profile of the zonal flow depicted in Figure \ref{fig2} is well-approximated by a function of the form \eqref{zf} if we require $U_0$ to 
have the minimum $U_0(0)=-2$ and the maximum $U_0(\frac{5\pi}{12})=1$ on $\big[0,\frac{\pi}{2}\big]$. The condition that $\theta=\frac{5\pi}{12}$ with $\cos\theta \approx \tfrac{1}{4}$  is a critical
 point and the above two specific values of the function $U_0$ yield the linear system
$$\begin{cases}
&\alpha + \beta + \gamma = -2\,,\\
& \tfrac{5}{4^4}\,\alpha + \tfrac{3}{4^2}\,\beta + \gamma = 0\,,\\
& \tfrac{1}{4^5}\,\alpha + \tfrac{1}{4^3}\,\beta + \tfrac{1}{4}\,\gamma = 1\,,
\end{cases}$$
whose unique solution is $\alpha=\frac{2048}{75}$, $\beta=-\frac{2656}{75}$, $\gamma=\tfrac{458}{75}$. With $\omega=13$ the relevant value for Neptune (see the first 
table in Section 7), we obtain that the second quadratic polynomial in $x=\cos^2\theta$ from \eqref{zf3} is 
$$x \mapsto x^2-\tfrac{211}{160}\,x + \tfrac{20731}{30720}\,,$$
with no real roots as $\tfrac{211}{160} \approx 1.31$ and $\tfrac{20731}{30720} \approx 0.674$. Consequently, choosing $A \in {\mathbb R}$ so that the first quadratic polynomial in 
\eqref{zf3} has no roots, Theorem \ref{theoremzonal} yields that the zonal flow pattern of Neptune is also stable. 
\end{itemize}
Unfortunately this approach is not applicable to the likely stability of the zonal jets on Jupiter and Saturn to breaking up into meanders and vortex-like eddies, but both 
zonal jet patterns (see \cite{read} for their detailed profiles) are not far off from entering the framework of Theorem~\ref{theoremzonal}. In contrast to this, the profiles of terrestrial 
stratospheric jets derived from observational data (see \cite{mns}) are well beyond condition \eqref{szc}, as is to be expected since the Earth's polar jet stream is known to be unstable.

\section{Stability results for degree 2 Rossby-Haurwitz waves}

\label{sectionRH}

Due to the considerable physical relevance of the largest-scale Rossby-Haurwitz waves \eqref{eotw} (that is, those with low degree), their stability properties are of great interest.

\begin{theorem}\label{thmRH} 
(i) (Stability of zonal Rossby-Haurwitz flows of degree $n \le 2$) 
Zonal solutions $\psi_0$ to \eqref{Eomega} of the form
\begin{equation}\label{rh2}
\psi_0(\theta) = \alpha\, \sin \theta + \beta Y_2^0(\theta)\,,\qquad \alpha \in {\mathbb R}\,,\quad \beta \in {\mathbb R}\setminus \{0\}\,,
\end{equation}
are stable in $H^2({\mathbb S}^2)$ under perturbations with bounded vorticity.

(ii) (Instability of non-zonal Rossby-Haurwitz waves) Non-zonal Rossby-Haurwitz waves $\psi_0$ of the form \eqref{eotw} are unstable. To be more specific, there exists $\epsilon>0$ 
and a sequence $\widehat{\psi}^n_0 \to \psi_0$ in $H^2({\mathbb S}^2)$, so that for the solutions $\widehat{\psi}^n(t)$ of \eqref{Eomega} with initial data $\widehat{\psi}^n(0)=\widehat{\psi}^n_0$ we have
$$
\sup_{t>0} \| \psi^n(t) - \psi_0(t) \|_{L^2({\mathbb S}^2)} > \epsilon > 0.
$$

(iii) (Stability of $\mathbb{E}_1 + \mathbb{E}_2$) The instability of a Rossby-Haurwitz wave of degree 2 can only occur by energy transfers between spherical harmonics of degree $2$. More precisely, 
if we consider a perturbation $\widehat{\psi}_0 \in H^2({\mathbb S}^2)$ of 
\begin{equation}\label{rh12}
\psi_0(\varphi,\theta) = -\frac{\omega}{2}\, \sin \theta + \beta_0 Y(\varphi,\theta)\,,
\end{equation}
where $Y \in \mathbb{E}_2$ with $\Vert Y \Vert_{L^2({\mathbb S}^2)}=1$ and $\beta_0 \in {\mathbb R}\setminus \{0\}$, the solution $\widehat{\psi}(t)$ of \eqref{Eomega} with initial data $\widehat{\psi}(0)=\widehat{\psi}_0$ 
can be written as
$$
\widehat{\psi}(t) =- \alpha\, \sin \theta + c_1^{-1}\,{\rm e}^{{\rm i}\omega t} \,Y_1^{-1}+ c_1^1 \,{\rm e}^{-{\rm i}\omega t} \,Y_1^1 + \beta(t)\, Y(t) +  \tilde{\psi}(t)\,,\qquad t \ge 0\,,
$$
where
\begin{equation*}
\left\{
\begin{array}{l}
c_1^{\pm 1} \in {\mathbb C} \quad\text{and}\quad \alpha,\, \beta(t) \in \mathbb{R}\,, \\
Y(t) \in \mathbb{E}_2 \quad\text{with}\quad \Vert Y(t) \Vert_{L^2({\mathbb S}^2)}=1\,,\\
\mathbb{P}_1 \tilde{\psi}(t) = \mathbb{P}_2\tilde{\psi}(t)  = 0 \quad\text{for all}\quad  t \ge 0\,, \\
|\alpha -\tfrac{\omega}{2}| + |c_1^{-1}| + |c_1^{1}| +  |\beta(t) - \beta_0| + \| \tilde{\psi}(t) \|_{H^2({\mathbb S}^2)} \lesssim \|  \psi_0 - \widehat{\psi}_0 \|_{H^2({\mathbb S}^2)} \,.
\end{array}
\right.
\end{equation*}
\end{theorem}

\begin{remark}
It seems natural to conjecture that all Rossby-Haurwitz waves of degree 2 are \textit{orbitally} stable. It should be possible to adapt the proof of assertion $(i)$ to prove this conjecture, 
though this would entail considerable technical complications.
\end{remark}

\begin{proof} $(i)$ We begin with a reduction: by using the symmetries of the problem (see Subsection~\ref{subsectionsymmetries}), more specifically the scaling and change-of-frame symmetries, it suffices to prove $(i)$ in the case $\omega=0$, $\beta=1$.

Consider a smooth perturbation 
\begin{equation}\label{rh2p}
\psi(\varphi,\theta,t)= \sum_{l= 1}^\infty \Big\{ \sum_{m=-l}^l c_l^m(t)\, Y_l^m(\varphi,\theta)\Big\}
\end{equation}
of the zonal flow \eqref{rh2}, expressed in terms of the spherical harmonics $Y_l^m$ by means of the time-dependent 
coefficients $c_l^m(t) \in {\mathbb C}$. Since $\psi$ is real-valued, \eqref{a-mm} yields
\begin{equation}\label{m-m}
c_l^{-m}(t)=\int_{{\mathbb S}^2} \psi\,\overline{Y_l^{-m}}\,d\sigma 
=(-1)^m \int_{{\mathbb S}^2} \psi\,Y_l^{m}\,d\sigma=(-1)^m \overline{\int_{{\mathbb S}^2} \psi\,\overline{Y_l^{m}}\,d\sigma}
=(-1)^m\,\overline{c_l^{m}(t)}\,,\quad |m| \le l\,.
\end{equation}
Furthermore, we know that
\begin{equation}\label{c1}
c_1^0(t)=c_1^0(0)\,,\quad |c_1^{\pm 1}(t)=|c_1^{\pm 1}(0)|\,,\qquad t \ge 0\,.
\end{equation}

Assuming that initially (at time $t=0$) the solution \eqref{rh2p} of \eqref{Eomega} is $\varepsilon$-close 
to the Rossby flow \eqref{rh2} in the $H^2$-norm, with $\varepsilon>0$ small, we have
\begin{align}\label{sin1}
\Vert \Delta [\psi_0 -\psi(\cdot,0)]\Vert_{L^2}^2 =& 4\,| \alpha - c_1^0(0)|^2 + 4\,|c_1^{-1}(0)|^2 + 
4\,|c_1^{1}(0)|^2 + 36\, |\beta - c_2^0(0)|^2 \\
&\quad+ 36 \sum_{0<|m| \le 2}^l |c_2^m(0)|^2+ \sum_{l=3}^\infty l^2(l+1)^2\Big\{ \sum_{m=-l}^l |c_l^m(0)|^2\Big\} < \varepsilon^2\,.\nonumber
\end{align}
On the other hand, the conservation of energy for the solution \eqref{rh2p} to \eqref{Eomega} reads
$$\sum_{l= 1}^\infty l(l+1) \Big\{ \sum_{m=-l}^l |c_l^m(t)|^2 \Big\}=\sum_{l= 1}^\infty l(l+1) \Big\{ \sum_{m=-l}^l |c_l^m(0)|^2 \Big\}\,,\qquad t \ge 0\,,$$
which, using \eqref{c1}, can be re-written in the form
\begin{equation}\label{seq1}
\sum_{l= 2}^\infty l(l+1) \Big\{ \sum_{m=-l}^l |c_l^m(t)|^2 \Big\}=\sum_{l=2}^\infty l(l+1) \Big\{ \sum_{m=-l}^l |c_l^m(0)|^2 \Big\}\,,\qquad t \ge 0\,.
\end{equation}
The time-invariance of the integral $\int_{{\mathbb S}^2} |\Delta \psi |^2\,{\rm d}\sigma$ gives the equality
\begin{align*}
& 4 \sum_{m=-1}^1 |c_1^{m}(t)|^2 + 
\sum_{l=2}^\infty l^2(l+1)^2 \Big\{ \sum_{m=-l}^l  |c_l^m(t)|^2 \Big\} =
4 \sum_{m=-1}^1 |c_1^{m}(0)|^2 + 
\sum_{l=2}^\infty l^2(l+1)^2 \Big\{ \sum_{m=-l}^l  |c_l^m(0)|^2 \Big\}\,,\qquad t \ge 0\,.\nonumber
\end{align*}
Using \eqref{c1}, we infer that 
\begin{equation}\label{tim0}
\sum_{l=2}^\infty l^2(l+1)^2 \Big\{ \sum_{m=-l}^l |c_l^m(t)|^2 \Big\}=\sum_{l=2}^\infty l^2(l+1)^2 \Big\{ \sum_{m=-l}^l |c_l^m(0)|^2 \Big\} \,,\qquad t \ge 0\,.
\end{equation}
The identities \eqref{seq1} and \eqref{tim0} ensure
$$\sum_{l=2}^\infty [l^2(l+1)^2 - 6l(l+1)] \Big\{ \sum_{m=-l}^l |c_l^m(t)|^2 \Big\}=
\sum_{l=2}^\infty [l^2(l+1)^2 - 6l(l+1)] \Big\{ \sum_{m=-l}^l |c_l^m(0)|^2 \Big\}\,,\qquad t \ge 0\,,$$
which, since the $l=2$ coefficient vanishes, can be written
\begin{equation}\label{ineq3}
\sum_{l=3}^\infty [l^2(l+1)^2 - 6l(l+1)] \Big\{ \sum_{m=-l}^l |c_l^m(t)|^2 \Big\} < \varepsilon^2\,,\qquad t \ge 0\,.
\end{equation}
Recalling \eqref{c1}, we conclude that the instability of the zonal Rossby flow \eqref{rh2} can only be caused by   
a substantial energy transfer between the spherical harmonic components of mode $l=2$. To rule this out, we 
rely on the time-invariance of the integrals
\begin{equation}\label{opsi}
I_k(\psi(\cdot,t))= \int_{{\mathbb S}^2}  \Big(\Delta\psi(\cdot,t) \Big)^k\,{\rm d}\sigma\,,\qquad k \in \{2,3,5\}\,,
\end{equation}
which ensures, using the Cauchy-Schwarz inequality and the boundedness of the vorticity,
\begin{equation}\label{opsie}
| I_k(\psi(\cdot,t)) - I_k(\psi_0)| = | I_k(\psi(\cdot,0)) - I_k(\psi_0)| \\
\lesssim \Vert \Delta \psi(\cdot,0) - \Delta \psi_0 \Vert_{L^2} < \varepsilon\,.
\end{equation}

We now take advantage of \eqref{ineq3} and of the specific structure of the spherical harmonics to elucidate the leading order 
of the integrals in \eqref{opsi} as $\varepsilon \to 0$. For this, note that integration by parts yields the recurrence formula 
$$\int_{-\frac{\pi}{2}}^{\frac{\pi}{2}} \cos^{2k+1}\theta\,{\rm d}\theta= \frac{2k}{2k+1} \int_{-\frac{\pi}{2}}^{\frac{\pi}{2}} \cos^{2k-1}\theta\,{\rm d}\theta\,,
\qquad k \ge 1\,,$$
which, since $\int_{-\frac{\pi}{2}}^{\frac{\pi}{2}} \cos\theta\,{\rm d}\theta=2$, yields the value of the Wallis integrals
$$\int_{-\frac{\pi}{2}}^{\frac{\pi}{2}} \cos^3\theta\,{\rm d}\theta=\tfrac{4}{3}\,,\quad 
\int_{-\frac{\pi}{2}}^{\frac{\pi}{2}} \cos^5\theta\,{\rm d}\theta=\tfrac{16}{15}\,,\quad
\int_{-\frac{\pi}{2}}^{\frac{\pi}{2}} \cos^7\theta\,{\rm d}\theta=\tfrac{32}{35}\,,\quad 
\int_{-\frac{\pi}{2}}^{\frac{\pi}{2}} \cos^9\theta\,{\rm d}\theta=\tfrac{256}{315}\,,\quad 
\int_{-\frac{\pi}{2}}^{\frac{\pi}{2}} \cos^{11}\theta\,{\rm d}\theta=\tfrac{512}{693}\,.$$
Taking into account the explicit formulas for the spherical harmonics (see the Appendix), we can now compute
\begin{align*}
&\int_{{\mathbb S}^2}  (Y_2^0)^2 \,{\rm d}\sigma =1\,,\quad \int_{{\mathbb S}^2}  Y_2^{-1}Y_2^1 \,{\rm d}\sigma =-1\,,\quad 
\int_{{\mathbb S}^2}  Y_2^{-2}Y_2^2 \,{\rm d}\sigma =1\,,\quad \int_{{\mathbb S}^2}  (Y_2^0)^3 \,{\rm d}\sigma =\tfrac{\sqrt{5}}{7\sqrt{\pi}}\,,\\
&\int_{{\mathbb S}^2}  Y_2^0Y_2^{-1}Y_2^1 \,{\rm d}\sigma =-\tfrac{\sqrt{5}}{14\sqrt{\pi}}\,,\quad 
\int_{{\mathbb S}^2}  Y_2^0Y_2^{-2}Y_2^2 \,{\rm d}\sigma =-\tfrac{\sqrt{5}}{7\sqrt{\pi}}\,,\quad
\int_{{\mathbb S}^2}  Y_2^2(Y_2^{-1})^2 \,{\rm d}\sigma =\tfrac{\sqrt{15}}{7\sqrt{2\pi}}\,,\quad \int_{{\mathbb S}^2}  Y_2^{-2}(Y_2^1)^2 \,{\rm d}\sigma =\tfrac{\sqrt{15}}{7\sqrt{2\pi}}\,,\\
& \int_{{\mathbb S}^2}  (Y_2^0)^4 \,{\rm d}\sigma =\tfrac{15}{28\pi}\,,\quad
\int_{{\mathbb S}^2}  (Y_2^0)^2Y_2^{-2} Y_2^2 \,{\rm d}\sigma =\tfrac{5}{28\pi}\,,\quad 
\int_{{\mathbb S}^2}   (Y_2^0)^2Y_2^{-1} Y_2^1\,{\rm d}\sigma =-\tfrac{5}{28\pi}\,,\quad 
\int_{{\mathbb S}^2}  Y_2^0 Y_2^2 (Y_2^{-1})^2 \,{\rm d}\sigma =0\,,\\ 
&\int_{{\mathbb S}^2}  Y_2^0 Y_2^{-2} (Y_2^{1})^2 \,{\rm d}\sigma =0\,,\quad 
\int_{{\mathbb S}^2}   (Y_2^2)^2(Y_2^{-2})^2\,{\rm d}\sigma =\tfrac{5}{14\pi}\,,\quad 
\int_{{\mathbb S}^2}  (Y_2^1)^2 (Y_2^{-1})^2 \,{\rm d}\sigma =\tfrac{5}{14\pi}\,,\quad 
\int_{{\mathbb S}^2}  Y_2^{-2} Y_2^2 Y_2^{-1}Y_2^1 \,{\rm d}\sigma =-\tfrac{5}{28\pi}\,,\\ 
&\int_{{\mathbb S}^2}   (Y_2^0)^5\,{\rm d}\sigma =\tfrac{5^2 \cdot 199 \sqrt{5}}{154\pi\sqrt{\pi}}\,,\quad 
\int_{{\mathbb S}^2}  (Y_2^0)^3 Y_2^{-2} Y_2^2 \,{\rm d}\sigma =-\tfrac{5\sqrt{5}}{154\pi\sqrt{\pi}}\,,
\quad \int_{{\mathbb S}^2}  (Y_2^0)^3 Y_2^{-1} Y_2^1 \,{\rm d}\sigma =-\tfrac{5^2\sqrt{5}}{4\cdot 154\pi\sqrt{\pi}}\,.
\end{align*}
We now use \eqref{opsi} and the multinomial formula
$$(x_1 + \dots + x_n)^k = \sum_{k_1+ \dots + k_n=k} \frac{k!}{k_1! \cdots k_n!} \prod_{j=1}^n x_j^{k_j}\,,$$
under the assumption \eqref{sin1}, which ensures \eqref{ineq3}. Noticing that the $\varphi$-dependence of the 
integrand shows that the integral $\int_{{\mathbb S}^2}  \prod_{j=1}^n (Y_2^{m_j})^{k_j} \,{\rm d}\sigma$  vanishes 
if $\sum_{j=1}^n m_j k_j \neq 0$, from \eqref{m-m} we obtain
\begin{align}
 I_2((\psi(\cdot,t)) & =   4 |c_1^0(t)|^2 + 8\, |c_1^1(t)|^2 
+  \sum_{l=2}^\infty l^2(l+1)^2 \Big\{|c_l^0(t)|^2+ 2 \sum_{m=1}^l |c_l^m(t)|^2 \Big\} \,,\label{exi2p}\\
 I_2(\psi_0) &  =4  \alpha^2 + 36 \,.\label{exi20}
\end{align}
Since \eqref{c1} and \eqref{sin1} ensure
\begin{equation} \label{ap2}
\Big| |c_1^0(t)|^2 -  \alpha^2 \Big| = 
| c_1^0(t)- \alpha| \cdot \Big|   c_1^0(t) + \alpha \Big| = | c_1^0(0)- \alpha| \cdot \Big|  c_1^0(0) + \alpha \Big| \lesssim \varepsilon \,,
\end{equation}
we see that \eqref{ineq3} and \eqref{exi2p}-\eqref{exi20} yield
$$I_2((\psi(\cdot,t)) - I_2(\psi_0)  = 36\,\Big\{ [c_2^0(t)]^2 + 2 |c_2^2(t)|^2 + 2 |c_2^1(t)|^2  - 1\Big\} + {\rm O}(\varepsilon)\,.$$
From \eqref{opsie} we therefore get
\begin{equation}\label{opsi2}
[c_2^0(t)]^2 + 2\, |c_2^2(t)|^2 + 2\, |c_2^1(t)|^2 =  1 + {\rm O}(\varepsilon)\,,
\end{equation}

We now show that
\begin{align}
\label{opsi3}
\tfrac{7}{15} \alpha^2 c_2^0(t) + [c_2^0(t)]^3 +  6\, c_2^0(t) |c_2^2(t)|^2 - 3\,c_2^0(t) |c_2^1(t)|^2 + 3 \sqrt{6} \,{\frak Re} \Big[ c_2^2(t)\,\Big(c_2^{1}(t)\Big)^2\Big]  =\tfrac{7}{15}\alpha^2 + 1 + {\rm O}(\varepsilon)\,,
\end{align}
where we denote by ${\frak Re}(z)$ the real part of the complex number $z$. For this, note that $c_1^{\pm 1}(t)$ and $c_l^m(t)$ for $l \ge 3$ are 
${\rm O}(\varepsilon)$, in view of \eqref{c1}, \eqref{sin1} and \eqref{ineq3}. Since $\int_{{\mathbb S}^2}  \prod_{j=1}^3 (Y_2^{m_j})^{k_j} \,{\rm d}\sigma$ vanishes if $\sum_{j=1}^3 m_j k_j \neq 0$, and
$$ \int_{{\mathbb S}^2} (Y_1^0)^3 \,{\rm d}\sigma=\int_{{\mathbb S}^2} Y_1^0 (Y_2^0)^2 \,{\rm d}\sigma=\int_{{\mathbb S}^2} Y_1^0 Y_2^{-2}Y_2^2 \,{\rm d}\sigma=
\int_{{\mathbb S}^2} Y_1^0 Y_2^{-1}Y_2^1 \,{\rm d}\sigma=0$$
because in each case we integrate an odd function of $\theta$ over $\big[ -\frac{\pi}{2},\,\frac{\pi}{2}\big]$, while
$$\int_{{\mathbb S}^2} (Y_1^0)^2 Y_2^0 \,{\rm d}\sigma=\tfrac{3\sqrt{5}}{8\sqrt{\pi}} \,\int_{-\frac{\pi}{2}}^{\frac{\pi}{2}} (3\sin^2\theta-1) \sin^2\theta\cos\theta\,{\rm d}\theta=\tfrac{1}{\sqrt{5\pi}}\,,$$
we obtain 
\begin{align*}
I_3(\psi(\cdot,t))  &=- \tfrac{72}{\sqrt{5\pi}}|c_1^0(t)|^2 c_2^0(t) \\
&\qquad -\tfrac{216 \sqrt{5}}{7\sqrt{\pi}} \, \Big\{ [c_2^0(t)]^3 +  6\, c_2^0(t) |c_2^2(t)|^2 - 3\,c_2^0(t) |c_2^1(t)|^2 + 3 \sqrt{6} \,{\frak R}e 
\Big[ c_2^2(t)\,\Big(c_2^{1}(t)\Big)^2\Big]\Big\} + {\rm O}(\varepsilon) \,.
\end{align*}
Due to \eqref{c1} and \eqref{opsie} for $k=3$, the relation \eqref{opsi3} now emerges by subtracting from the above
$$I_3(\psi_0)
=-\tfrac{72}{\sqrt{5\pi}} \alpha^2 -\tfrac{216 \sqrt{5}}{7\sqrt{\pi}} $$
and taking \eqref{ap2} into account.

We now investigate the leading order of $I_4((\psi(\cdot,t))$. Again, since $c_1^{\pm 1}(t)$ and $c_l^m(t)$ for $l \ge 3$ 
are ${\rm O}(\varepsilon)$, for this we need only to keep track of the integrals 
$\int_{{\mathbb S}^2}  \prod_{j=1}^3 (Y_1^0)^k(Y_2^{m_j})^{k_j} \,{\rm d}\sigma$ with $\sum_{j=1}^3 m_j k_j =0$ and $k+ \sum_{j=1}^3 k_j=4$. 
We compute
\begin{align*}
&\int_{{\mathbb S}^2}  (Y_1^0)^4 \,{\rm d}\sigma = \tfrac{9}{20\pi}\,,\quad \int_{{\mathbb S}^2}  Y_1^0(Y_2^0)^3 \,{\rm d}\sigma =0\,,\quad 
\int_{{\mathbb S}^2}  (Y_1^0)^2(Y_2^0)^2 \,{\rm d}\sigma = \tfrac{11}{28\pi}\,,\\
&\int_{{\mathbb S}^2}  (Y_1^0)^2Y_2^2Y_2^{-2} \,{\rm d}\sigma = \tfrac{3}{14\pi}\,,\qquad  
\int_{{\mathbb S}^2}  (Y_1^0)^2 Y_2^1 Y_2^{-1} \,{\rm d}\sigma = \tfrac{9}{28\pi}\,,
\end{align*}
and infer that
\begin{align*}
I_4((\psi(\cdot,t)) &= \tfrac{36}{5\pi}\, | c_1^0(t)|^4 + \tfrac{6^3 \cdot 11}{7\pi}| c_1^0(t)|^2 [c_2^0(t)]^2 + \tfrac{6^4 \cdot 4}{7\pi}\,|c_1^0(t)|^2 |c_2^2(t)|^2  
- \tfrac{6^4 \cdot 3}{7\pi}|c_1^0(t)|^2 |c_2^1(t)|^2\,,\\
&\qquad + \tfrac{6^4 \cdot 15}{28\pi} \Big\{ [c_2^0(t)]^2 + 2  \Big( |c_2^2(t)|^2 + |c_2^1(t)|^2 \Big)\Big\}^2 + {\rm O}(\varepsilon)\,.\end{align*}
Subtracting from this the relation
$$I_4(\psi_0)=\tfrac{36}{5\pi} \alpha^4 + \tfrac{6^3 \cdot 11}{7\pi}\alpha^2 + \tfrac{6^4 \cdot 15}{28\pi}\,,$$
invoking \eqref{opsie} with $k=4$ and \eqref{c1}, we get
\begin{equation}\label{opsi4} 
 22 \alpha^2 [c_2^0(t)]^2 + 48 \alpha^2 |c_2^2(t)|^2   
- 36 \alpha^2 |c_2^1(t)|^2  + 45 \Big\{ [c_2^0(t)]^2 + 2  \Big( |c_2^2(t)|^2 + |c_2^1(t)|^2 \Big)\Big\}^2 = 22 \alpha^2 + 45 + {\rm O}(\varepsilon)\,,
\end{equation}
using \eqref{ap2} and the fact that \eqref{sin1} ensures
\begin{align}
\Big| |c_1^0(t)|^4 -  \alpha^4\Big| 
\lesssim \varepsilon\,.\nonumber
\end{align}

We can now prove the stability of the zonal solution \eqref{rh2} for $\alpha \neq 0$. In this case, from 
\eqref{opsi2} and \eqref{opsi4} we get
\begin{equation}\label{opsi44}
[c_2^0(t)]^2 + \tfrac{24}{11} \, |c_2^2(t)|^2 - \tfrac{18}{11} \, |c_2^1(t)|^2 = 1  + {\rm O}(\varepsilon)\,,
\end{equation}
which, subtracted from \eqref{opsi2}, yields
\begin{equation}\label{opsigc1}
|c_2^2(t)|^2 - 15 \, |c_2^1(t)|^2 ={\rm O}(\varepsilon)\,.
\end{equation}
In view of the continuous dependence on data guaranteed by the well-posedness of \eqref{Eomega}, 
invoking \eqref{c1} and the fact that \eqref{ineq3} is guaranteed by \eqref{sin1}, we see that the zonal Rossby 
flow $\psi_0$ given by \eqref{rh2} is unstable only if there exists some $\delta_0>0$ and a sequence of initial data $\{\psi_N(\cdot,0)\}_{N \ge 1}$ 
converging to $\Psi$ in $H^2$ and such that for every $N_0 \ge 1$ large enough there exists a time $t_{N_0}>0$ with
\begin{equation}\label{|instr0}
|c_{N_0,2}^2(t_{N_0})|^2 + |c_{N_0,2}^1(t_{N_0})|^2 = 2\delta_0^2 \,,
\end{equation}
where $\{c_{N_0,l}^m(t)\}_{l \ge 1,\,|m| \le l}$ are the coefficients of the expansion of the solution $\psi_{N_0}(\cdot,t)$ to \eqref{Eomega} with 
initial data $\psi_{N_0}(\cdot,0)$. The validity of \eqref{|instr0} for some $\delta_0>0$ ensures a similar relation for any $\delta \in (0,\delta_0)$, at 
some time $T>0$. We can thus analyse the limit $\delta \to 0$: from \eqref{opsi2}, \eqref{opsigc1} and \eqref{|instr0} we get
\begin{equation}\label{cas1}
\begin{cases}
&[c_{N,2}^0(t_{N})]^2 = 1 - 4\delta^2 + {\rm O}(\varepsilon)\,,\\
&|c_{N,2}^2(t_{N})|^2 = \tfrac{15}{8}\,\delta^2 + {\rm O}(\varepsilon)\,,\\
&|c_{N,2}^1(t_{N})|^2 = \tfrac{1}{8}\,\delta^2 + {\rm O}(\varepsilon)\,.
\end{cases}
\end{equation}
Thus
$$c_{N,2}^0(t_{N})=(1-2\delta^2 - \tfrac{1}{2}\,\delta^4) + {\rm O}(\delta^6) + {\rm O}(\varepsilon)\,,$$
which, together with the last two relations in \eqref{cas1} makes the matching at ${\rm O}(\delta^3)$ in \eqref{opsi3} impossible. This 
contradiction proves the stability of \eqref{rh2} if $\alpha \neq 0$. 

It remains to deal with the case $\alpha= 0$, a setting in which \eqref{opsi3} simplifies to 
\begin{equation}\label{opsi3s}
[c_2^0(t)]^3 +  6\, c_2^0(t) |c_2^2(t)|^2 - 3\,c_2^0(t) |c_2^1(t)|^2 + 3 \sqrt{6} \,{\frak R}e 
\Big[ c_2^2(t)\,\Big(c_2^{1}(t)\Big)^2\Big]  =1 + {\rm O}(\varepsilon)\,.
\end{equation}
but \eqref{opsi4} does not provide additional information with respect to \eqref{opsi2}, so that we can not rely on \eqref{opsigc1}. 
To compensate for the ineffectiveness of \eqref{opsi4} we have to take advantage of \eqref{opsie} with $k=5$. Assuming instability, we find 
for every $\delta>0$ small enough an initial data $\varepsilon$-close to \eqref{rh2} such that for the corresponding solution $\psi_N(\cdot,t)$ 
at some time $T_N>0$ we have
\begin{equation}\label{|instr}
|c_{N,2}^2(t_{N})|^2 + |c_{N,2}^1(t_{N})|^2 = 2\delta^2\,,
\end{equation}
From \eqref{opsi2} we then get 
$$ [c_{N,2}^0(t_N)]^2 = (1 - 4\delta^2) +  {\rm O}(\varepsilon)\,,$$
so that
\begin{equation}\label{asy1}
c_2^0(t_N) = (1 - 2\delta^2) +{\rm O}(\delta^4) +  {\rm O}(\sqrt{\varepsilon})\,,
\end{equation}
and by writing
$$6\, c_2^0(t) |c_2^2(t)|^2 - 3\,c_2^0(t) |c_2^1(t)|^2 =3\,c_2^0(t) \,\Big\{ |c_2^2(t)|^2 + |c_2^1(t)|^2 \Big\} + 3\,c_2^0(t) \,\Big\{ |c_2^2(t)|^2 -2 |c_2^1(t)|^2 \Big\} $$
the relations \eqref{opsi3s}-\eqref{|instr} yield
\begin{equation}\label{asy2}
|c_{N,2}^2(t_N)|^2 -  2\, |c_{N,2}^1(t_N)|^2 =  {\rm O}(\delta^3) + {\rm O}(\sqrt{\varepsilon})\,.
\end{equation}
We now notice that $I_5(\psi_N(\cdot,t_N)$ is asymptotically an additive ${\rm O}(\varepsilon)$-correction of linear combinations of integrals of the type 
$\int_{{\mathbb S}^2}  \prod_{j=1}^5 (Y_2^{m_j})^{k_j} \,{\rm d}\sigma$ with $\sum_{j=1}^5 m_jk_j=0$, so that by computing the eleven integrals
\begin{align*}
&\int_{{\mathbb S}^2}  (Y_2^0)^5 \,{\rm d}\sigma\,,\quad \int_{{\mathbb S}^2}  (Y_2^0)^3Y_2^{-2}Y_2^2 \,{\rm d}\sigma\,,\quad 
\int_{{\mathbb S}^2}  (Y_2^0)^3 Y_2^{-1}Y_2^1 \,{\rm d}\sigma\,,\quad \int_{{\mathbb S}^2}  (Y_2^0)^2 Y_2^2 (Y_2^{-1})^2 \,{\rm d}\sigma\,,\quad \int_{{\mathbb S}^2}  (Y_2^0)^2Y_2^{-2}(Y_2^1)^2 \,{\rm d}\sigma\,,\\ 
&\int_{{\mathbb S}^2}  Y_2^0(Y_2^{-2})^2(Y_2^2)^2 \,{\rm d}\sigma\,,\quad \int_{{\mathbb S}^2}  Y_2^0(Y_2^{-1})^2 \,{\rm d}\sigma\,,\quad 
\int_{{\mathbb S}^2}  Y_2^0(Y_2^{-1})^2(Y_2^{1})^2 \,{\rm d}\sigma\,,\quad \int_{{\mathbb S}^2}  Y_2^0Y_2^{-2}Y_2^2 Y_2^{-1}Y_2^1 \,{\rm d}\sigma\,,\\
& \int_{{\mathbb S}^2}  (Y_2^2)^2 Y_2^{-2} (Y_2^{-1})^2 \,{\rm d}\sigma \,,
\quad \int_{{\mathbb S}^2}  (Y_2^{-2})^3 Y_2^{2} (Y_2^1)^2 \,{\rm d}\sigma \,,
\end{align*}
we can determine the leading order with respect to $\varepsilon$. If we also bring $\delta \gg \varepsilon$ into play, then 
\eqref{|instr} and \eqref{asy1} ensure that $I_5(\psi_N(\cdot,t_N)$ consists of a linear combination of the first three integrals (already computed) and 
an additive ${\rm O}(\delta^3,\varepsilon)$-correction:
\begin{equation}\label{i52}
I_5(\psi_N(\cdot,t_N)=-\tfrac{6^5 \cdot5^2\cdot 199 \sqrt{5}}{154\pi\sqrt{\pi}} \,\Big\{ [c_2^0(t_N)]^5 + 
\frac{1}{199}\,[c_2^0(t_N)]^3 [c_{N,2}^1(t_N)]^2 -  \frac{4}{199}\,[c_2^0(t_N)]^3 |c_{N,2}^2(t_N)|^2\Big\} + {\rm O}(\delta^3)+ {\rm O}(\varepsilon)\,,
\end{equation}
while
\begin{equation}\label{i520}
I_5(\Psi)=-\tfrac{6^5 \cdot5^2\cdot 199 \sqrt{5}}{154\pi\sqrt{\pi}} \,.
\end{equation}
Due to \eqref{opsie} with $k=5$, and taking into account \eqref{asy1} and \eqref{|instr}, from \eqref{i52}-\eqref{i520} we get
\begin{equation}\label{asy3}
[c_{N,2}^1(t_N)]^2 -  4\, |c_{N,2}^2(t_N)|^2 = 80 \cdot 199\delta^2 + {\rm O}(\delta^3) + {\rm O}(\sqrt{\varepsilon})\,.
\end{equation}
But \eqref{|instr}, \eqref{asy2} and \eqref{asy3} clearly cannot hold simultaneously. The obtained contradiction proves that the zonal 
flow \eqref{rh2} is stable also for $\alpha=0$.\medskip

\noindent $(ii)$ Let
$$
\widehat{\psi}^n_0 = \psi_0 + \tfrac{1}{n}\, \sin \theta = \big( \alpha + \tfrac{1}{n} \big)\, \sin \theta + \beta\, Y(\varphi,\theta)
$$
with $Y \in {\mathbb E}_j$. According to \eqref{eotw}, the solution $\widehat{\psi}^n(t)$ of \eqref{Eomega} with initial data $\widehat{\psi}^n(0)=\widehat{\psi}^n_0$ is the travelling wave
$$\widehat{\psi}^n(t) =\big( \alpha + \tfrac{1}{n} \big)\, \sin \theta + \beta\, Y(\varphi- \widehat{c}\,t,\theta)$$
with propagation speed $\widehat{c}=\frac{2\omega}{j(j+1)} + \left( \alpha + \frac{1}{n} \right)\,\frac{j(j+1)-2}{j(j+1)}$. On the other hand,
$$\psi_0(t) =\alpha \, \sin \theta + \beta\, Y(\varphi-ct,\theta)$$
with $c=\frac{2\omega}{j(j+1)} + \alpha\,\frac{j(j+1)-2}{j(j+1)}$. If (see the Appendix)
$$Y(\varphi,\theta)=\sum_{|m| \le j} c_j\,P_j^m(\sin\theta)\,{\rm e}^{{\rm i}m\varphi}\,,$$
we get
\begin{align*}
\sup_{t >0} \| \widehat{\psi}^n(t) - \psi(t) \|_{L^2({\mathbb S}^2,{\rm d}\sigma)}^2 &= \sup_{t >0} \left\| \tfrac{1}{n} \, \sin \theta + \beta\, Y(\varphi - \widehat{c}\,t, \theta) - \beta\, Y(\varphi-ct, \theta) \right\|_{L^2({\mathbb S}^2,{\rm d}\sigma)}^2 \\
&= \sup_{t >0} \left\| \tfrac{1}{n} \, \sin \theta + \beta\, \sum_{|m| \le j} c_j\,P_j^m(\sin\theta)\,\Big( {\rm e}^{{\rm i}m(\varphi - \widehat{c}\,t)}- {\rm e}^{{\rm i}m(\varphi - c\,t)}\Big)\right\|_{L^2({\mathbb S}^2,{\rm d}\sigma)}^2\\
&= \sup_{t >0}  \Big\{ \frac{4\pi}{3n}  + \beta^2 \sum_{0<|m| \le j} |c_j|^2 \Big( \int_{-\tfrac{\pi}{2}}^{\tfrac{\pi}{2}} \Big|P_j^m(\sin\theta)\Big|^2 \cos\theta \,{\rm d}\theta\Big)
  \Big(\int_0^{2\pi} \Big| {\rm e}^{{\rm i}m(\varphi - \widehat{c}\,t)}- {\rm e}^{{\rm i}m(\varphi - c\,t)}\Big|^2{\rm d}\varphi \Big)\Big\} \\
  &= \sup_{t >0}  \Big\{ \frac{4\pi}{3n}  + 2\pi \beta^2 \sum_{0<|m| \le j} |c_j|^2 \Big( \int_{-\tfrac{\pi}{2}}^{\tfrac{\pi}{2}} \Big|P_j^m(\sin\theta)\Big|^2 \cos\theta \,{\rm d}\theta\Big)
  \Big| 1- {\rm e}^{{\rm i}m(\widehat{c} - c)\,t}\Big|^2\Big\} \\
   &= \sup_{t >0}  \Big\{ \frac{4\pi}{3n} + 4\pi \beta^2 \sum_{0<|m| \le j} |c_j|^2 \Big( \int_{-\tfrac{\pi}{2}}^{\tfrac{\pi}{2}} \Big|P_j^m(\sin\theta)\Big|^2 \cos\theta \,{\rm d}\theta\Big)
  \Big( 1- \cos\big( m\,\tfrac{j(j+1)-2}{nj(j+1)}\,t\big)\Big)\Big\} \\
  &=  \frac{4\pi}{3n} + 4\pi \beta^2 \sum_{0<|m| \le j} |c_j|^2 \Big( \int_{-\tfrac{\pi}{2}}^{\tfrac{\pi}{2}} \Big|P_j^m(\sin\theta)\Big|^2 \cos\theta \,{\rm d}\theta\Big) \\
  &> 4\pi \beta^2 \sum_{0<|m| \le j} |c_j|^2 \Big( \int_{-\tfrac{\pi}{2}}^{\tfrac{\pi}{2}} \Big|P_j^m(\sin\theta)\Big|^2 \cos\theta \,{\rm d}\theta\Big)
\end{align*}
for all $n \ge 1$.
\smallskip

\noindent $(iii)$ Inspired by the approach used in the proof of Theorem \ref{theoremzonal}, we define
$$
\mathcal{E}(\psi) = \iint_{{\mathbb S}^2} \Big[ \tfrac{1}{2}\,|U|^2 - \tfrac{1}{12} \,|\Omega + 2 \omega \sin \theta|^2 + \omega \sin \theta\, \Omega -  |\mathbb{P}_1 \psi|^2 \Big] \,{\rm d}\sigma\,.
$$
This functional is constant along solutions $\psi$ of~\eqref{Eomega}. Expanding a solution $\psi=\psi_0 + \delta\psi$ of~\eqref{Eomega} in terms of spherical harmonics, 
$$\begin{cases}
\psi_0(\varphi,\theta)= -\omega \sqrt{\tfrac{\pi}{3}}\,Y_1^0(\theta) + \beta_0\displaystyle\sum_{|m| \le 2} a_m\,Y_2^m(\varphi,\theta) \quad\text{with}\quad \sum_{|m| \le 2} |a_m|^2=1\,,\\
\delta\psi(\varphi,\theta,t)= \displaystyle\sum_{j \ge 1} \Big\{ \sum_{|m| \le j} c_j^m(t)\, Y_j^m(\varphi,\theta)\Big\} \quad\text{with}\quad c_1^0(t)=c_1^0 \quad\text{and}\quad c_1^{\pm 1}(t)=c_1^{\pm 1}(0)\,{\rm e}^{\mp {\rm i}\omega t} \quad\text{for}\quad t \ge 0\,,
\end{cases}$$
we can write (see the Appendix)
\begin{align*}
\mathcal{E}(\psi) &= -\tfrac{1}{3} \,\Big\{ \Big( c_1^0(0) -\omega \sqrt{\tfrac{\pi}{3}} \Big)^2 + |c_1^{-1}(0)|^2 + |c_1^{1}(0)|^2\Big\} + \tfrac{1}{12} \sum_{j \ge 3} \Big\{ \sum_{|m| \le j} |c_j^m(t)|^2 [6 - j(j+1)] \,j(j+1)\Big\}\\
& \le  -\tfrac{1}{3} \,\Big\{ \Big( c_1^0(0) -\omega \sqrt{\tfrac{\pi}{3}} \Big)^2 + |c_1^{-1}(0)| + |c_1^{1}(0)|\Big\} - \tfrac{1}{2} \sum_{j \ge 3} j(j+1) \Big\{ \sum_{|m| \le j} |c_j^m(t)|^2 \Big\}\,.
\end{align*}
Exploiting the conservation of the kinetic energy, which ensures that the expression
$$\Big( c_1^0(0) -\omega \sqrt{\tfrac{\pi}{3}} \Big)^2 + |c_1^{-1}(0)|^2 + |c_1^{1}(0)|^2 + 3\sum_{|m| \le 2} |c_2^m(t) + \beta_0 a_m|^2 + \sum_{j \ge 3} \tfrac{j(j+1)}{2} \Big\{ \sum_{|m| \le j} |c_j^m(t)|^2 \Big\}$$
is time-independent, setting
$$\alpha=\tfrac{1}{2}\,c_1^0(0)\,\sqrt{\tfrac{\pi}{3}}\,,\quad \beta(t)=\sqrt{\sum_{|m| \le 2} |c_2^m(t) + \beta_0 a_m|^2}\,,\quad Y(t)=\sum_{|m| \le 2} \tfrac{c_2^m(t) + \beta_0 a_m}{\beta(t)}\,Y_2^m\,,\quad \tilde{\psi}(t)=\sum_{j \ge 3} \Big\{ \sum_{|m| \le j} c_j^m(t)\,Y_j^m\Big\}\,,$$
proves the claim.\end{proof}

\section{Bifurcation from Rossby-Haurwitz waves}

This section is dedicated to constructing stationary and travelling-wave solutions of \eqref{Eomega} which are different from the explicit solutions of~\eqref{ellipticpsi} studied above. For this, 
we seek non-zonal solutions to a suitably modified form of equation \eqref{ellipticpsi} 
which bifurcate from Rossby-Haurwitz waves. 

\subsection{The case $\omega=0$.} 

To implement a bifurcation approach, it is convenient to introduce the parameter $\lambda \in {\mathbb R}$, seeking solutions $\psi \in  C^{2,\alpha}(\mathbb{S}^2)$ of the nonlinear elliptic equation
\begin{equation}\label{spdec}
-\Delta \psi + F(\lambda,\psi) =0
\end{equation}
for functions $F \in C^2({\mathbb R}^2,{\mathbb R})$. Note that a solution $\psi \in  C^{2,\alpha}(\mathbb{S}^2)$ to \eqref{spdec} satisfies the Gauss constraint 
\begin{equation}\label{gc2}
\iint_{{\mathcal S}^2} F(\lambda,\psi) \,{\rm d}\sigma = 0\,.
\end{equation}
On the other hand, any solution $\psi \in  C^{2,\alpha}(\mathbb{S}^2)$ of
\begin{equation}\label{spdecc}
-\Delta \psi + F(\lambda,\psi) = \frac{1}{4\pi} \iint_{{\mathcal S}^2} F(\lambda,\psi)\,{\rm d}\sigma
\end{equation}
provides us with a stationary solution of \eqref{Eomega}. Note that the zero function $\psi \equiv 0$ solves \eqref{spdecc}, and the vorticity $\Delta\psi$ of a 
solution $\psi \in  C^{2,\alpha}(\mathbb{S}^2)$ to \eqref{spdecc} satisfies the Gauss constraint \eqref{gaussc}. For this reason, rather than solving \eqref{spdec} with the constraint 
\eqref{gc2}, we will seek solutions to \eqref{spdecc}. 

Since the Laplace-Beltrami operator is bijective from 
$C^{2,\alpha}(\mathbb{S}^2)$ onto  
$$C^{0,\alpha}_0(\mathbb{S}^2)=\Big\{f \in C^{0,\alpha}(\mathbb{S}^2):\ \iint_{{\mathcal S}^2} f\,{\rm d}\sigma = 0\Big\}\,,$$
with a compact inverse that we denote by $\Delta^{-1}$, we can recast equation \eqref{spdecc} in the form
\begin{equation}\label{gp}
{\mathbb F}(\lambda,f)=0 \quad\text{with}\quad {\mathbb F}: {\mathbb R} \times C^{0,\alpha}_0(\mathbb{S}^2) \to C^{0,\alpha}_0(\mathbb{S}^2)\,,\quad 
{\mathbb F}(\lambda,f)= f - \Delta^{-1} \Big\{ F(\lambda,f) - \frac{1}{4\pi} \iint_{{\mathcal S}^2} F(\lambda,f)\,{\rm d}\sigma \Big\}\,.
\end{equation}

Invariant spherical harmonic basis functions, tabulated by their respective subgroups of ${\mathbb O}(3)$, 
are listed in \cite{gss}. For a finite subgroup ${\mathbb G}$ of ${\mathbb O}(3)$ with the property that the subspace of 
${\mathbb G}$-invariant spherical harmonics of degree $n \ge 1$ is one-dimensional, the fact that \eqref{spdecc} is 
equivariant with respect to the natural action of the orthogonal group enables us to consider this problem restricted to ${\mathbb G}$-equivariant functions.
This way, we can take advantage of the symmetries to analyse the formation of regular flow patterns using the Rabinowitz global 
bifurcation approach (see \cite{kiel, rabi}).

\begin{lemma}[The Rabinowitz global bifurcation theorem]\label{gbl}
Let ${\mathbb Y}$ be  a real Banach space and let ${\mathbb F} \in C^2({\mathbb R}  \times {\mathbb Y},{\mathbb Y})$ be such that 
$f \mapsto f-{\mathbb F}(\lambda,f)$ is compact operator from ${\mathbb R}  \times {\mathbb Y}$ to ${\mathbb Y}$ and:

(i) ${\mathbb F}(\lambda,0)=0$ for all $(\lambda,0) \in {\mathbb R}  \times {\mathbb Y}$;

(ii) $\partial_f {\mathbb F}(\lambda^\ast,0) \in {\mathcal L}({\mathbb Y},{\mathbb Y})$ is a 
linear Fredholm operator of index zero and one-dimensional kernel 
${\mathcal N}(\partial_f {\mathbb F}(\lambda^\ast,0))$ generated by some $f^\ast \in {\mathbb Y} \setminus \{0\}$;

(iii) the transversality condition holds, in the sense that $[\partial_{\lambda, f}^2\,{\mathbb F}(\lambda^\ast,0)]\,(1,f^\ast)$ does not belong to the range of the operator 
$\partial_f {\mathbb F}(\lambda^\ast,0))$, where $\partial^2_{\lambda, f}{\mathbb F}(\lambda^\ast,0)=\partial_\lambda[\partial_f
{\mathbb F}(\lambda,0)]\big|_{\lambda=\lambda^\ast} \in {\mathcal L}({\mathbb R},{\mathcal L}({\mathbb Y},{\mathbb Y}))=
{\mathcal L}({\mathbb R} \times {\mathbb Y},{\mathbb Y})$.

Then there exist $\varepsilon >0$, an
open set ${\mathcal O} \subset {\mathbb R}  \times {\mathbb Y}$ with $(\lambda^\ast,0) \in {\mathcal O}$ and a branch of solutions
$$\{(\lambda,f)=(\lambda(s),\,s\,\chi(s)):\ s \in {\mathbb R},\, |s| < \varepsilon\} \subset  {\mathbb R} \times {\mathbb Y}$$
of ${\mathbb F}(\lambda,f)=0$ with $\lambda(0)=\lambda^\ast$, $\chi(0)=f^\ast$,
and such that $s \mapsto \lambda(s) \in {\mathbb R}$, $s \mapsto s\chi(s)\in {\mathbb Y}$ are continuously differentiable on
$(-\varepsilon,\varepsilon)$, with
$$\{(\lambda,f) \in {\mathcal O}:\ {\mathbb F}(\lambda,f)=0,\ f \neq 0\}=\{(\lambda(s),\,s\,\chi(s)):\ 0< |s| < \varepsilon\}.$$

Furthermore, if ${\mathcal S}$ is the closure of the set of nontrivial solutions of ${\mathbb F}(\lambda,f)=0$ in 
${\mathbb R}  \times {\mathbb Y}$, then the connected component ${\mathcal S}^\ast$ of ${\mathcal S}$ to which 
$(\lambda^\ast,0)$ belongs has at least one of the following properties:

(I) ${\mathcal S}^\ast$ is unbounded in ${\mathbb R}  \times {\mathbb Y}$;

(II) there exists some $\lambda \neq \lambda^\ast$ such that $(\lambda,0) \in {\mathcal S}^\ast$.
\end{lemma}

We now prove the following existence result.

\begin{theorem}\label{exi1}
Let $G$ be a finite subgroup ${\mathbb G}$ of ${\mathbb O}(3)$ with the property that the subspace of 
${\mathbb G}$-invariant spherical harmonics of some specific degree $\ell \ge 3$ is one-dimensional and non-zonal.  
If $F: {\mathbb R}^2 \to {\mathbb R}$ is twice continuously differentiable and such that
$$
F_{f}(\lambda^*,0) = \ell (\ell + 1), \qquad F_{f \lambda}(\lambda^*,0) \neq 0,
$$
then there exists a maximal connected component of nontrivial solutions of \eqref{spdecc} such that all solutions  $(\lambda,\psi) \in {\mathcal S}^\ast$ close enough to $(\lambda^\ast,0)$ are non-zonal.
\end{theorem}

\begin{figure}[ht]
\begin{center}
\begin{minipage}{190mm}{
\resizebox*{7cm}{!}{\includegraphics{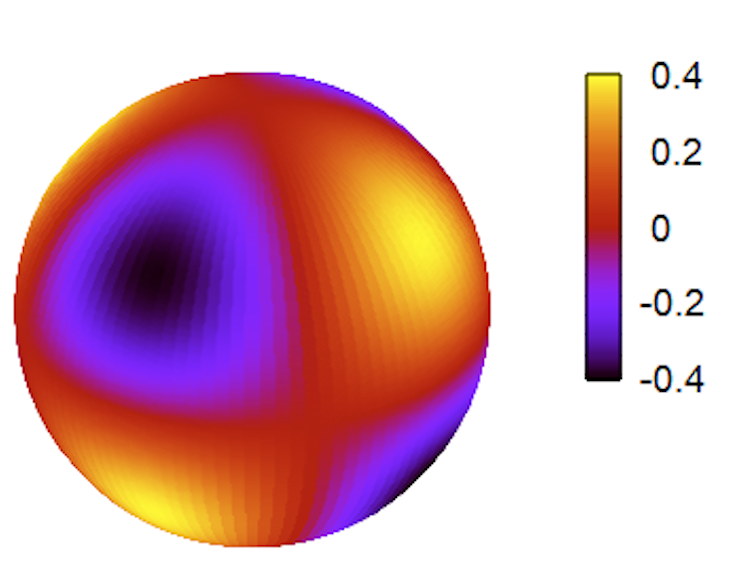}}}\hspace{1.7cm}{
\resizebox*{7cm}{!}{\includegraphics{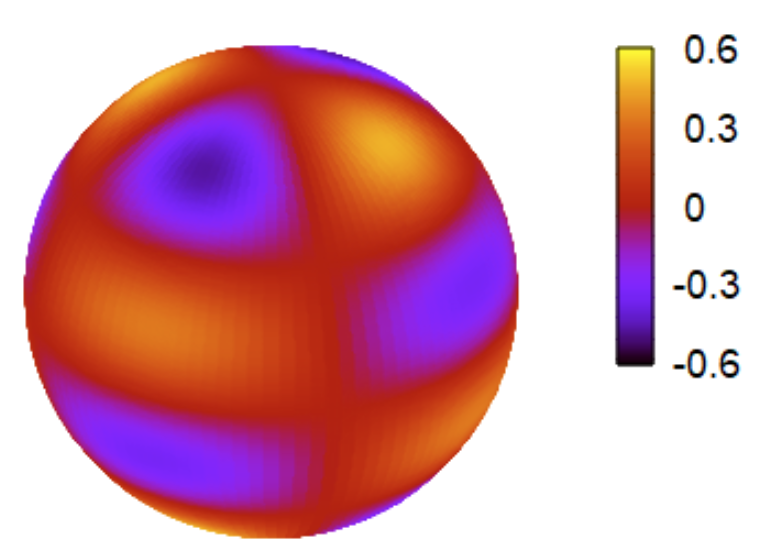}}}%
\caption{\footnotesize{Visualization of the non-axial symmetries of specific spherical harmonics, using different colours to keep track of the geometry of different level sets. Note that the subspace of spherical harmonics of degree 3 that are invariant under the finite 
tetrahedral subgroup (the symmetry group of the methane molecule) is generated by $Y_{3}^{-2}$, while 
the subspace of spherical harmonics of degree 5 that are invariant under the finite subgroup $D_4^d$ 
(the symmetry group of the bicapped square antiprism, describing the shape of the octasulfur molecule) is generated by $Y_5^{-2}$.}}
\label{f4}
\end{minipage}
\end{center}
\end{figure}

\begin{proof} Consider the map ${\mathbb F}: {\mathbb R} \times {\mathbb Y} \to {\mathbb Y}$ defined as in \eqref{gp} by
\begin{equation}\label{gpg}
{\mathbb F}(\lambda,f)= f - \Delta^{-1} \Big\{ F(\lambda,f) - \frac{1}{4\pi} \iint_{{\mathcal S}^2} F(\lambda,f)\,{\rm d}\sigma \Big\}\,
\end{equation}
where the Banach space
$${\mathbb Y}=\{f \in C^{0,\alpha}_0({\mathbb S}^2):\ Gf=f \quad\text{for all}\quad G \in {\mathbb G}\}$$
captures the symmetries associated with the group ${\mathbb G}$. The linearization of the operator ${\mathbb F}$ about the trivial solution $f=0$ of ${\mathbb F}(\lambda,f)=0$ is
\begin{equation}\label{linbif}
\partial_f {\mathbb F}(\lambda,0)[\psi]= \psi - \Delta^{-1} \Big\{ F_f(\lambda,0)\psi - \frac{1}{4\pi} \iint_{{\mathcal S}^2} F_f(\lambda,0)\psi\,{\rm d}\sigma \Big\} 
= \psi - F_f(\lambda,0) \, \Delta^{-1}\psi\,,\qquad \psi \in \mathbb{Y}\,.
\end{equation}
For $F_f(\lambda,0)=-\ell(\ell+1)$ this operator acting on $C^{0,\alpha}_0({\mathbb S}^2)$ has a nontrivial kernel given by the $(2\ell+1)$-dimensional space of spherical harmonics of order $\ell$. Consequently, if the finite subgroup ${\mathbb G}$ of ${\mathbb O}(3)$ has the property that the subspace of 
${\mathbb G}$-invariant spherical harmonics of degree $\ell$ is one-dimensional, then for $\lambda$ such that 
$F_f(\lambda,0)=-\ell(\ell+1)$, it follows that the kernel ${\mathcal N}(\partial_f {\mathbb F}(\lambda,0))$ of the operator ${\mathbb F}$ defined in \eqref{gpg} is one-dimensional, being generated by some $f^* \in \mathbb{E}_\ell$. Note that for a simple eigenvalue, the condition that 
$\partial_f {\mathbb F}(\lambda^\ast,0)$ is a Fredholm operator of index zero means that the range of this operator is closed and has a one-dimensional complement. 
If $\lambda^\ast$ is such that $F_f(\lambda^\ast,0)=-\ell(\ell+1)$, then, due to elliptic regularity and to the self-adjointness of the Laplace-Beltrami operator in $L^2(\mathbb{S}^2)$, we know that 
$\psi \in C^{0,\alpha}(\mathbb{S}^2)$ belongs to the range of $\partial_f {\mathbb F}(\lambda^\ast,0)$ if and only if it belongs to the 
orthogonal complement in $L^2(\mathbb{S}^2)$ of the spherical harmonics of degree $\ell$. From this it follows at once that $\partial_f {\mathbb F}(\lambda^\ast,0)$ acting on 
${\mathbb Y}$ has a closed range with a one-dimensional complement in ${\mathbb Y}$. Furthermore, since 
$$[\partial^2_{\lambda, f}{\mathbb F}(\lambda^\ast,0)]\,(1,f^\ast)= - \, F_{f \lambda}(\lambda^\ast,0)\,\Delta^{-1} f^\ast \,,$$
we see that the transversality condition is equivalent to $F_{f \lambda}(\lambda^\ast,0) \neq 0$. All the hypotheses in Lemma \ref{gbl} hold, and 
we deduce the existence of a curve $\{(\lambda(s),\,s\,\chi(s)):\ s \in {\mathbb R},\, |s| < \varepsilon\}$ of nontrivial solutions that bifurcates at $(\lambda^\ast,0)$ from the curve 
$\{(\lambda,0):\ \lambda \in {\mathbb R}\}$ of trivial solutions. Since the tangent vector of this nontrivial solution curve 
at the bifurcation point is given by $(\lambda'(0),f^\ast)$ and $f^\ast$ is non-zonal, near the bifurcation point all nontrivial 
solutions are non-zonal. 
\end{proof}

The next result describe settings in which it is possible to reveal structural properties of the continuum of solutions found in Theorem \ref{exi1}.

\begin{theorem}\label{bt2}
If $P: {\mathbb R} \to {\mathbb R}$ is twice continuously differentiable and 
\begin{itemize}
\item $\lim_{\lambda \to \infty} P(\lambda)=\infty$ and $\lim_{\lambda \to -\infty} P(\lambda)=-\infty$,
\item there exists $\lambda^\ast \in {\mathbb R}$ with $P'(\lambda^\ast)=- l(l+1)$ and $P''(\lambda^\ast) \neq 0$,
\item there exists $a>0$ with $P'(\lambda)>0$ for $|\lambda|>a$,
\end{itemize}
then the maximal connected component built up in Theorem~\ref{exi1} for $F(\lambda, f)=P(\lambda +f)-P(\lambda)$ has a closure ${\mathcal S}^\ast$ which is bounded in ${\mathbb R} \times {\mathbb Y}$, contains $(\lambda^\ast,0)$ as 
well as some $(\lambda,0)$ with $\lambda \neq \lambda^\ast$.
\end{theorem}

\begin{proof}
We have to prove that all nontrivial solutions $(\lambda,\psi)$ of ${\mathbb F}(\lambda,\psi)=0$ with ${\mathbb F}$ defined by \eqref{gpg} are bounded {\it a priori} 
in ${\mathbb R} \times {\mathbb Y}$, or, equivalently, 
that all nontrivial solutions of \eqref{spdecc} are bounded {\it a priori} in ${\mathbb R} \times C^{2,\alpha}_0({\mathbb S}^2)$. Note that for 
$\psi \in C^{2,\alpha}({\mathbb S}^2)$, $\psi \not \equiv 0$, we have 
$$\psi_m =\inf_{\mathbb{S}^2} \{\psi\} < 0 < \sup_{\mathbb{S}^2} \{\psi\}= \Psi_M$$
since $\iint_{\mathbb{S}^2} \psi \,{\rm d}\sigma=0$. For a function $P$ with the properties specified in the statement one can easily see that 
there exist $a_+>0$ and $a_-<0$ such that 
\begin{itemize}
\item $P(a_+) >0 > P(a_-)$ and $P'(\lambda)>0$ for $\lambda \in (-\infty,a_-) \cup (a_+,\infty)$,
\item  $\sup_{ \lambda \in [a_-,a_+]} \{ |P(\lambda)|\} \le A= \max\{ -P(a_-),\,P(a_+)\}$.
\end{itemize}
If $(\lambda,\psi) \in {\mathbb R} \times {\mathbb Y}$ with $\psi \not\equiv 0$ lies on the continuum of solutions, the weak maximum principle applied to \eqref{spdecc} yields
$$F(\lambda,\psi_M)=P(\lambda + \psi_M)- P(\lambda) \le \frac{1}{4\pi} \iint_{{\mathcal S}^2} F(\lambda,\psi)\,{\rm d}\sigma \le P(\lambda + \psi_m)- P(\lambda)=F(\lambda,\psi_m)\,.$$
Thus $P(\lambda + \psi_m) \ge P(\lambda + \psi_M)$ and since $\lambda + \psi_m < \lambda < \lambda + \psi_M$  we deduce that all three points belong to the interval $[a_-,a_+]$. 
Consequently we have the following $L^\infty(\mathbb{S}^2)$ {\it a priori} bounds for any nontrivial solution $\psi \in C^{2,\alpha}_0({\mathbb S}^2)$ of equation \eqref{spdecc}:
\begin{equation}\label{linf}
|\lambda| + |\psi| \le 2(a_+-a_-) \quad\text{and}\quad |\Delta \psi| \le 2A\,.
\end{equation}
We now invoke for $k>2$ the $L^k$-estimates for elliptic equations on smooth compact manifolds without boundary 
(see \cite{bes}): there are constants $c_0>0$ and $c_2>0$ such that
$$\Vert \psi \Vert_{W^{2,k}(\mathbb{S}^2)} \le c_2 \Vert \Delta \psi \Vert_{L^k(\mathbb{S}^2)} + c_0 \Vert \psi \Vert_{L^1(\mathbb{S}^2)}$$
for every $\psi \in W^{2,k}(\mathbb{S}^2)$. In conjunction with \eqref{linf} and with the Sobolev embedding $W^{2,k}(\mathbb{S}^2) \subset 
C^1(\mathbb{S}^2)$, we obtain {\it a priori} bounds in $C^1(\mathbb{S}^2)$ for the nontrivial solutions $\psi \in C^{2,\alpha}_0({\mathbb S}^2)$ of \eqref{spdecc}. But then 
\eqref{spdecc} yields by differentiation {\it a priori} bounds for $\Vert \Delta \psi \Vert_{C^{0,\alpha}(\mathbb{S}^2)}$. We now conclude the boundedness 
of the nontrivial solutions of \eqref{spdecc}  in ${\mathbb R} \times C^{2,\alpha}_0({\mathbb S}^2)$  from the Schauder estimates (see \cite{bes}): there are constants $C_0>0$ and $C_2>0$ such that
$$\Vert \psi \Vert_{C^{2,\alpha}(\mathbb{S}^2)} \le C_2 \Vert \Delta \psi \Vert_{C^{0,\alpha}(\mathbb{S}^2)} + C_0 \Vert \psi \Vert_{C(\mathbb{S}^2)}$$
for every $\psi \in  C^{2,\alpha}(\mathbb{S}^2)$.\qed
\end{proof}

%

\begin{remark}\label{exir}
Simple examples of functions satisfying the hypotheses of Theorem \ref{exi1} are provided by the polynomials 
$$P(\lambda)=\mu_1 \lambda^3 - [\mu + l(l+1)]\lambda$$
with parameters $\mu>0$ and $\mu_1>0$.\qed
\end{remark}

\subsection{The case $\omega>0$}

We now establish the existence of a global continuum of solutions to \eqref{ellipticpsi} for $\omega >0$ and suitable 
twice continuously differentiable functions $F: {\mathbb R} \to {\mathbb R}$. To take advantage of the fact that the spherical 
harmonics provide a representation of the orthogonal group ${\mathbb O}(3)$, we write \eqref{ellipticpsi} in the form
\begin{equation}\label{eqz}
\Delta \psi(\xi)= F(\psi(\xi)) - 2 \omega z\,,\qquad \xi \in \mathbb{S}^2\,,
\end{equation}
where $z$ is the distance of $\xi \in \mathbb{S}^2$ to the equatorial plane. An adequate class of nonlinear functions $F$ is 
obtained by modifying outside a neighbourhood of zero the linear functions 
$\psi \mapsto -l(l+1)\psi$ with $l \in {\mathbb N}$, aiming at replacing the solution set
$$\{ \omega z + \sum_{m=-l}^l c_m Y_l^m(\varphi,\theta):\, c_m \in {\mathbb R}\ \text{for}\ -l \le m \le l\}$$
which features functions with gradients of all possible sizes by a continuum of nontrivial solutions with an 
{\it a priori} bound on the corresponding velocity fields and vorticities.

\begin{theorem}\label{exi2}
Let $G$ be a finite subgroup ${\mathbb G}$ of ${\mathbb O}(3)$ with the property that the subspace of 
${\mathbb G}$-invariant spherical harmonics of some specific degree $l \ge 3$ is one-dimensional and non-zonal, being generated 
by some $f^\ast \neq 0$. Given $\beta>0$, if $P: {\mathbb R} \to {\mathbb R}$ is twice continuously differentiable and 
\begin{itemize}
\item there exists $\mu > \frac{2\nu}{l(l+1)}$ with $P(\lambda)=-\frac{2\nu}{\mu}\,\lambda$ for $|\lambda| \le 2\mu$, where 
$\nu=\frac{\beta l(l+1)}{l(l+1)-2}>\beta$,
\item $\lim_{\lambda \to -\infty} P(\lambda)=-\infty$, $\lim_{\lambda \to \infty} P(\lambda)=\infty$ and there exists $a>2\mu$ with $P(\lambda) > 0$ for $|\lambda| >  a$,
\end{itemize}
then there exists a maximal connected component ${\mathcal S}_\beta$ of nontrivial solutions of the vorticity equation
\begin{equation}\label{eqzl}
\Delta \psi= P((1+\lambda^2)\psi) - 2 \Big( \nu - \frac{\mu}{1+\lambda^2}\Big)\, z - \frac{1}{4\pi} \,\iint_{{\mathbb S}^2} P((1+\lambda^2)\psi) \,{\rm d}\sigma\,,
\end{equation} 
such that the corresponding velocity fields are uniformly bounded {\it a priori}. Moreover, the continuum ${\mathcal S}_\beta$ 
comprises non-zonal Rossby solutions of the form 
$$\psi=c^\ast f^\ast - \frac{\mu}{1+\lambda^2}\,z \quad\textit{with}\quad c^\ast \neq 0$$ 
close to the zonal solution $\psi_0=-\frac{2\nu}{l(l+1)}\,z$ of the linear equation 
$$\Delta \psi - l (l+1)\psi + 2 \beta z=0\,.$$
\end{theorem}

\begin{proof}
Let us first note that if $f$ solves
\begin{equation}\label{eqza}
\Delta f - P\Big((1+\lambda^2)f - \mu z\Big) + 2 \nu z +  \frac{1}{4\pi} \,\iint_{{\mathbb S}^2} P\Big((1+\lambda^2)f - \mu z\Big) \,{\rm d}\sigma=0\,,
\end{equation}
then 
\begin{equation}\label{eqzs}
\psi= f - \frac{\mu}{1+\lambda^2}\,z
\end{equation}
solves \eqref{eqzl}. We will therefore develop a global bifurcation approach to establish the existence of 
nontrivial solutions of \eqref{eqza}. Using the Banach space ${\mathbb Y}$ introduced in 
the proof of Theorem \ref{exi1}, we transform \eqref{eqza} to ${\mathbb F}(\lambda,f)=0$, where
$${\mathbb F}: {\mathbb R} \times {\mathbb Y} \to {\mathbb Y}\,,\qquad 
{\mathbb F}(\lambda,f)= f - \Delta^{-1} \Big\{ P\Big((1+\lambda^2)f - \mu z\Big) - 2 \nu z\Big\}\,.$$
Note that ${\mathbb F}(\lambda,0)=0$. Since $|z| \le 1$ we have
$$\partial_f {\mathbb F}(\lambda,0)[f_0]=f_0 - (1+\lambda^2)\Delta^{-1} \Big\{ P'( - \mu z) f_0 - \frac{1}{4\pi} \,\iint_{{\mathbb S}^2} P'( - \mu z) f_0\,{\rm d}\sigma \Big\}
=f_0 + \frac{2\nu(1+\lambda^2)}{\mu}\,\Delta^{-1}f_0\,,$$
so that a necessary condition for the kernel of $\partial_f {\mathbb F}(\lambda,0)$ to comprises more than $\{0\}$ is 
$$\frac{2\nu(1+\lambda^2)}{\mu}= n(n+1)$$
for some $n \in {\mathbb N}$. The case $n=l$ corresponds to 
\begin{equation}\label{lbp}
\pm \lambda^\ast = \sqrt{\frac{\mu l (l+1)}{2\nu}-1}
\end{equation}
and in this case the kernel is one-dimensional, being generated by $f^\ast$. Since
$$[\partial^2_{\lambda f}{\mathbb F}(\lambda^\ast),0] (1,f^\ast)=\frac{4\lambda^\ast \nu}{\mu} \Delta^{-1}f^\ast 
\quad\text{with}\quad \frac{4\lambda^\ast \nu}{\mu} \neq 0\,,$$
a reasoning analogous to that in the proof of Theorem \ref{exi1} ensures the existence of a global continuum of 
nontrivial solutions that bifurcate at $(\lambda^\ast,0)$ from the curve of trivial solutions. Performing the transformation 
\eqref{eqzs} we obtain a corresponding continuum of non-trivial solutions to \eqref{eqzl}. Close to the bifurcation point, as long 
as $|(1+\lambda^2)f - \mu z| \le 2 \mu$ throughout $\mathbb{S}^2$, the equation \eqref{eqza} takes the form
\begin{equation}\label{lin0}
\Delta f + \frac{2\nu}{\mu}\,(1+\lambda^2)f =0\,.
\end{equation}
Equation \eqref{lin0} has a nontrivial solution only if $\frac{2\nu}{\mu}\,(1+\lambda^2)=n(n+1)$ for some $n \in {\mathbb N}$. But 
for $\lambda=\lambda^\ast$, equation \eqref{lin0} takes the form $\Delta f + l(l+1)f =0$, due to \eqref{lbp}. 
Consequently, for as long as $|(1+\lambda^2)f - \mu z| \le 2 \mu$ throughout $\mathbb{S}^2$, we must have 
$\lambda= \lambda^\ast$ for the corresponding non-trivial solutions in 
the continuum and the non-trivial function $\psi$ is a spherical harmonic of degree $l$ and thus a multiple of $f^\ast$, as claimed. 
Moreover, taking into account the two alternatives for the continuum of non-trivial bifurcating solutions, we infer the existence 
of non-trivial solutions in this continuum satisfying $\Vert (1+\lambda^2)f - \mu z\Vert_{L^\infty(\mathbb{S}^2)} > 2 \mu$, solutions for which the 
nonlinear adjustment of $P$ comes into play.

It remains to prove that the gradients of solutions to \eqref{eqzl} within the above continuum are 
uniformly bounded {\it a priori}. As in the proof of Theorem \ref{exi1}, this follows at once 
from elliptic $L^k$ estimates with $k>2$ in conjunction with the Sobolev embedding $W^{2,k}(\mathbb{S}^2) \subset C^1(\mathbb{S}^2)$, 
once we establish an $L^\infty(\mathbb{S}^2)$ {\it a priori} bound for the non-trivial solutions $\psi \in C^{2,\alpha}_0({\mathbb S}^2)$ of \eqref{eqzl}. 
For this, choose $b>a$ such that $P(-b)<0<P(b)$ and 
$$\min\{ P(b),\,|P(-b)|\} \ge \max_{|\lambda| \le a} \{ |P(\lambda)|\}\,.$$
Let us now note that if $(\lambda,\psi) \in {\mathbb R} \times C^{2,\alpha}_0({\mathbb S}^2)$ with $\psi \not \equiv 0$ solves \eqref{eqzl}, then 
$$\psi_m =\inf_{\mathbb{S}^2} \{\psi\} < 0 < \sup_{\mathbb{S}^2} \{\psi\}= \Psi_M$$
since $\iint_{\mathbb{S}^2} \psi \,{\rm d}\sigma=0$. From the weak maximum principle in conjunction with \eqref{eqza} we obtain
$$P((1+\lambda^2)\psi_M - \mu z)- 2\nu z \le  \frac{1}{4\pi} \,\iint_{{\mathbb S}^2} P\Big((1+\lambda^2)f - \mu z\Big) \,{\rm d}\sigma \le P((1+\lambda^2)\psi_m - \mu z) - 2\nu z\,,$$
so that
$$P((1+\lambda^2)\psi_M - \mu z) \le P((1+\lambda^2)\psi_m - \mu z)\,.$$
Since $\psi_m < \psi_M$, we must have 
$$-b  \le (1+\lambda^2)\psi_m - \mu z \le (1+\lambda^2)\psi_M - \mu z \le b\,.$$
But then $|z| \le 1$ yields
$$-(b+\mu) \le (1+\lambda^2)\psi_m  \le (1+\lambda^2)\psi_M \le b+\mu\,,$$
and therefore $\Vert \psi\Vert_{L^\infty(\mathbb{S}^2)} \le \mu +b$. This completes the proof.\qed
\end{proof}

\section{Relevance for stratospheric flows}

In this section we show that some of the inviscid flows on a rotating sphere studied hitherto are building blocks for the leading-order dynamics of 3D stratospheric flows.

In the stratosphere the atmospheric flow is practically inviscid (see \cite{cat}), being thus governed 
by the components of the Euler equation (see \cite{gill})
\begin{subequations}\label{eq:euler}
\begin{align}
&\frac{\mathrm D u'}{\mathrm Dt'} + \frac{u'w'-u'v'\tan\theta}{r'} - 2 \varOmega'(v'\sin\theta -w'\cos\theta)
  = -\frac{1}{\rho'}\frac{1}{r'\cos\theta}\frac{\partial p'}{\partial \varphi}\,,\\
&\frac{\mathrm D v'}{\mathrm Dt'} + \frac{v'w'+u'^2\tan\theta}{r'} +2 \varOmega'u'\sin\theta + \varOmega'^2 r'\sin\theta\cos\theta
  = -\frac{1}{\rho'}\frac{1}{r'}\frac{\partial p'}{\partial \theta}\,, \\
&\frac{\mathrm D w'}{\mathrm Dt'} - \frac{u'^2+v'^2}{r'} - 2 \varOmega' u'\cos\theta - \varOmega'^2r'\cos^2\theta
  = -\frac{1}{\rho'}\frac{\partial p'}{\partial r'} -g'\,,
\end{align}
\end{subequations}
where the material derivative $\mathrm D/\mathrm Dt'$ in spherical coordinates is given by 
$$ \frac{\mathrm D\phantom{|}}{\mathrm Dt'} = \frac{\partial}{\partial t'}+ \frac{u'}{r'\cos\theta}\frac{\partial}{\partial \varphi} + \frac{v'}{r'}\frac{\partial}{\partial \theta} + w'\frac{\partial}{\partial r'}\,.$$
Here $p'$ and $\rho'$ are the pressure and density in the atmosphere, $\varOmega'$ is the constant 
rate of rotation of the planet, and $g'$ is the acceleration due to gravity, taken to be a constant. (We use primes to denote physical/dimensional variables; they will be removed 
when we nondimensionalize.) The conservation of mass in spherical coordinates takes the form
\begin{gather}\label{eq:mass-conv}
\frac{\mathrm D\rho'}{\mathrm Dt'} + \rho'\left(\frac{1}{r'\cos\theta}\,\frac{\partial u'}{\partial \varphi} + \frac{1}{r'\cos\theta}\,\frac{\partial}{\partial \theta}(v'\cos\theta) 
+ \frac{1}{r'^2}\,\frac{\partial}{\partial r'}\,(r'^2 w')\right)=0\,,
\end{gather}
while the equation of state for an ideal gas reads
\begin{equation}\label{ig}
p'=\rho' {\frak R}'T'\,,
\end{equation}
where $T'$ is the (absolute) temperature and ${\frak R}'$ is the gas constant. The first law of thermodynamics should also hold:
\begin{equation}\label{fl}
c_p'   \, \frac{{\mathrm D}T'}{\text{D}t'} -\kappa' \nabla'^2  T' - \frac{1}{\rho'}\, \frac{{\mathrm D}p'}{{\mathrm D}t'}=Q'\,,
\end{equation}
where 
$$\nabla'^2 \equiv \frac{\partial^2}{\partial r'^2} + \frac{2}{r'}\,\frac{\partial}{\partial r'}+ 
\frac{1}{r'^2}\Big( \frac{1}{\cos^2\theta}\,\frac{\partial^2}{\partial \varphi^2} 
+ \frac{\partial^2}{\partial \theta^2} -\tan\theta \frac{\partial}{\partial \theta}\Big)\,,$$
$c_p'$ is the specific heat and $\kappa'/c_p'$ is the thermal diffusivity, while $Q'$ is a general heat-source term.
We will mainly work with the pressure and the density, so that in our setting the role of the ideal gas law \eqref{ig} is to specify 
the temperature, while the second law of thermodynamics identifies the associated heat sources. 

We now introduce the following dimensional scales:
 \begin{align}\label{eq:non-dimen}
 \begin{aligned}
  R' &: \text{ radius of the planet (as a distance scale)} \\
  H' &: \text{ mean width of the stratosphere} \\
  U' &: \text{ horizontal velocity scale} \\
  W' &: \text{ vertical velocity scale} \\
  \overline{\rho}' &: \text{ average density of the stratosphere}\,.
  \end{aligned}
 \end{align}
The inverse Rossby number is defined as
 \begin{align}\label{eq:def-omega}
  \omega=\frac{\varOmega'R'}{U'}\,,
 \end{align}
and two further important flow-parameters are the shallowness parameter 
$\mu$ and the ratio $\delta$ between the vertical 
and horizontal velocity scales, given by
\begin{equation}\label{eps}
\mu=\frac{H'}{R'}
\quad\text{and}\quad
\delta = \frac{W'}{U'}\,.
\end{equation}
The relevant data is suggested by the characteristics of persistent large-scale flow patterns in the stratosphere (see \cite{cat, dow, lun}):\bigskip

\begin{tabular}{ |p{1.25cm}||p{1.6cm}|p{1.2cm}|p{1.7cm}|p{2.75cm}|p{1.4cm}|p{1.7cm}|p{0.5cm}|p{1.4cm} |p{1.5cm}| }
 \hline
 Planet & $R'$ & $H'$ & $g'$ & $\varOmega'$ & $U'$ & $W'$ & $\omega$ & $\mu$ & $\delta$ \\
 \hline\hline
 Earth   & 6371 km    & 40 km & 9.8 m/s$^2$ & $7.27\times 10^{-5}\,\text{rad/s}$ & 50 m/s & $10^{-3}$ m/s & 9 & $6 \times 10^{-3}$& $2 \times 10^{-5}$\\
Jupiter &  69911 km   & 270 km & 24.8 m/s$^2$ & $1.76\times 10^{-4}\,\text{rad/s}$ & 150 m/s  & $10^{-2}$ m/s & 82 &$ 4 \times 10^{-3}$ & $6 \times 10^{-5}$ \\
Saturn &  58232 km   & 200 km & 10.4 m/s$^2$ & $1.62\times 10^{-4}\,\text{rad/s}$ & 150 m/s  & $10^{-2}$ m/s & 63 &$ 3 \times 10^{-3}$ & $6 \times 10^{-5}$ \\
Neptune &  24622 km   & 200 km & 11.1 m/s$^2$ & $1.08\times 10^{-4}\,\text{rad/s}$ & 200 m/s  & $10^{-3}$ m/s & 13 &$ 8 \times 10^{-3}$ & $5 \times 10^{-6}$ \\
Uranus &  25362 km   & 150 km & 8.8 m/s$^2$ & $1.04\times 10^{-4}\,\text{rad/s}$ & 150 m/s  & $10^{-5}$ m/s & 18 &$ 6 \times 10^{-3}$ & $6 \times 10^{-8}$ \\
 \hline
\end{tabular}

\bigskip

Since typically $\delta \le 10^{-4}$, the time scale $R'/U'$ is determined by the horizontal flow (the values for Earth, Jupiter, Saturn, Neptune and Uranus being about 
1.5, 5.5, 4.5, 1.4 and 2 days, respectively).  We now define the dimensionless variables $t$, $z$, $u$, $v$, $w$, $\rho$, and $p$ by
\begin{equation}\label{nd}
t' =\displaystyle\frac{R'}{U'}\,t\,,\quad r' =R'+H'z\,,\quad(u',v') =U'(u,v)\,,\quad 
 w' =W'w\,,\quad \rho' = \overline{\rho'}\rho\,, \quad p'=\overline{\rho'}U'^2 p\,,
\end{equation}
and obtain from \eqref{eq:euler}-\eqref{eq:mass-conv} the components of the nondimensional Euler equation 
\begin{align}
&\displaystyle\frac{\mathrm D u}{\mathrm Dt} + \frac{\delta uw-uv\tan\theta}{1+\mu z} - 2 \omega (v\sin\theta - \delta w\cos\theta)
= -\frac{1}{\rho}\frac{1}{(1+\mu z)\cos\theta}\, \frac{\partial p}{\partial \varphi}   \label{elongnd} \\
&\displaystyle\frac{\mathrm D v}{\mathrm Dt} + \frac{\delta vw+u^2\tan\theta}{1+\mu z} +2 \omega u \sin\theta + 
\omega^2 (1+\mu z) \sin\theta \cos\theta
  = -\frac{1}{\rho}\frac{1}{1+\mu z} \, \frac{\partial p}{\partial \theta}  \label{elatnd} \\
&\displaystyle \mu\delta\, \frac{\mathrm D w}{\mathrm Dt} - \mu\,\frac{u^2+v^2}{1+\mu z} - 2 \mu\omega u \cos\theta - 
\mu\omega^2 (1+\mu z)^2 \cos^2\theta
 = -\frac{1}{\rho}\frac{\partial p}{\partial z} -g \,,\label{ervertnd}
\end{align}
and the nondimensional equation of mass conservation
\begin{equation}\label{mcnd}
\frac{\mathrm D\rho}{\mathrm Dt} + \rho\Big\{\frac{1}{(1+\mu z)\cos\theta}\,\frac{\partial u}{\partial \varphi} + \frac{1}{(1+\mu z)\cos\theta}\,\frac{\partial}{\partial \theta}(v\cos\theta)  + \frac{\delta}{\mu}\frac{1}{(1+\mu z)^2}\,\frac{\partial}{\partial z}\,\Big((1+\mu z)^2 w\Big)\Big\}=0\,,\end{equation}
where 
$$\frac{\mathrm D\phantom{|}}{\mathrm Dt} = \frac{\partial}{\partial t}+ \frac{u}{(1+\mu z) \cos\theta}\frac{\partial}{\partial \varphi} + 
\frac{v}{1+\mu z}\frac{\partial}{\partial \theta} + \frac{\delta}{\mu}\,w\frac{\partial}{\partial z} \quad\text{and}\quad 
g=\frac{g' H'}{U'^2}\,,$$
with $g \approx 157$ for Earth, $g \approx 297$ for Jupiter, $g \approx 92$ for Saturn, $g \approx 58$ for Uranus, and $g \approx 55$ for Neptune. 

We are interested in the leading-order dynamics as $\mu \to 0$, 
the physically relevant regime for the thin-shell stratosphere being characterised by
\begin{equation}\label{ts}
\delta \ll \mu \ll 1\,,
\end{equation}
so that the flow dynamics is governed at leading-order by the non-dimensional equations \eqref{elongnd}-\eqref{mcnd} in the limit $\mu \to 0$:
\begin{align}
&\displaystyle\frac{\partial u_0}{\partial t}+ \frac{u_0}{\cos\theta}\frac{\partial u_0}{\partial \varphi}+ v_0\frac{\partial u_0}{\partial \theta} 
- u_0v_0\tan\theta - 2\omega\, v_0\sin\theta
= -\frac{1}{\rho_0}\frac{1}{\cos\theta}\, \frac{\partial p_0}{\partial \varphi} \,,  \label{elongl} \\
&\displaystyle\frac{\partial v_0}{\partial t}+\frac{u_0}{\cos\theta}\frac{\partial v_0}{\partial \varphi}+ v_0\frac{\partial v_0}{\partial \theta} 
+ u_0^2\tan\theta + 2\omega\, u_0\sin\theta  + \omega^2  \sin\theta \cos\theta
= -\frac{1}{\rho_0} \, \frac{\partial p_0}{\partial \theta} \,,  \label{elatl} \\
&0 = \displaystyle\frac{1}{\rho_0}\frac{\partial p_0}{\partial z} +g \,,\label{ervertl}\\
&\displaystyle\frac{\partial \rho_0}{\partial t}+ \frac{u_0}{\cos\theta}\frac{\partial \rho_0}{\partial \varphi}+ v_0\frac{\partial \rho_0}{\partial \theta}  + 
\frac{\rho_0}{\cos\theta} \Big(\frac{\partial u_0}{\partial \varphi} + \frac{\partial}{\partial \theta}(v_0\cos\theta) \Big)=0\,.\label{mcl}
\end{align}
Throughout the stratosphere the main changes in density are in the vertical direction, with the density decreasing with height 
(e.g., from about 100 g/cm$^3$ at the bottom of the Earth's stratosphere to about 1 g/cm$^3$ at its top), so that we restrict our attention 
to the setting
\begin{equation}\label{dh}
\rho_0=\rho_0(z)\,.
\end{equation}
The flow dynamics is then governed at leading-order by the system
\begin{subequations}\label{eel}
\begin{align}
&\frac{\partial u_0}{\partial t}+ \frac{u_0}{\cos\theta}\frac{\partial u_0}{\partial \varphi}+ v_0\frac{\partial u_0}{\partial \theta} 
- u_0v_0\tan\theta - 2\omega\, v_0\sin\theta = -\frac{1}{\rho_0 \cos\theta}\frac{\partial p_0}{\partial \varphi}\,,\\ 
&\frac{\partial v_0}{\partial t}+\frac{u_0}{\cos\theta}\frac{\partial v_0}{\partial \varphi}+ v_0\frac{\partial v_0}{\partial \theta} 
+ u_0^2\tan\theta + 2\omega\, u_0\sin\theta + \omega^2  \sin\theta \cos\theta = - \frac{1}{\rho_0}\frac{\partial p_0}{\partial \theta}\,,  \\
&0 = \frac{1}{\rho_0}\frac{\partial p_0}{\partial z} +g \,,\\
& \frac{\partial u_0}{\partial \varphi} + \frac{\partial}{\partial \theta}(v_0\cos\theta) =0\,.
\end{align}
\end{subequations}
For $\rho_0$ constant the system \eqref{eel} particularizes to that describing inviscid flow on the 
surface of a rotating sphere. This feature 
is related to the fact that, due to an ascending temperature with height, the stratosphere is stably stratified and vertical motion is suppressed.  
 
Equation (\ref{eel}c) yields the existence of a stream function, $\psi(\varphi,\theta,z,t)$, satisfying
\begin{equation}\label{strg}
  u_0=-\frac{\partial \psi}{\partial \theta}
  \quad\text{and}\quad
  v_0=\frac{1}{\cos\theta}\frac{\partial \psi}{\partial \varphi}\,,
\end{equation}
while the elimination of the dynamic pressure $P_0$ between the equations (\ref{eel}a)-(\ref{eel}b) gives the vorticity equation
\begin{equation}\label{vorteq}
\frac{\partial}{\partial t}\,\Delta\psi + \frac{1}{\cos\theta}\, \left[\frac{\partial\psi}{\partial\varphi}\,\frac{\partial}{\partial\theta}
  - \frac{\partial\psi}{\partial\theta}\,\frac{\partial}{\partial\varphi}\right]\left(\nabla^2_\Sigma\psi + 2\omega\sin\theta\right) = 0\,,
\end{equation}
in which $\Delta= \frac{\partial^2}{\partial\theta^2} - \tan\theta\frac{\partial}{\partial\theta} + \frac{1}{\cos^2\theta}\frac{\partial^2}{\partial\varphi^2}$ is the Laplace-Beltrami operator on 
the surface of the unit sphere ${\mathbb S}^2$ and $\Delta \psi$ is the vorticity of the flow. 

\begin{lemma}\label{lem}
If $\psi_0(\varphi,\theta)$ solves
\begin{equation}\label{li}
\Delta \psi_0 = F(\psi_0) -\frac{1}{4\pi}\, \iint_{{\mathbb S}^2} F(\psi_0)\,{\rm d}\sigma
\end{equation}
for some $F \in C^1({\mathbb R},{\mathbb R})$, then 
\begin{equation}\label{lirot}
\psi(\varphi,\theta,z,t)=\omega \sin\theta + \psi_0(\varphi +\omega\,t, \theta)
\end{equation}
is a solution of the vorticity equation \eqref{vorteq}.
\end{lemma}

\begin{proof}
Since $\Delta \sin\theta = - 2 \sin\theta$, we have
$$\Delta(\psi_0+\omega \sin\theta) =F(\psi_0)- 2\omega \sin\theta -\frac{1}{4\pi}\, \iint_{{\mathbb S}^2} F(\psi_0)\,{\rm d}\sigma\,,$$
so that 
$$\Sigma\Delta+2\omega \sin\theta =F(\widehat{\psi})  -\frac{1}{4\pi}\, \iint_{{\mathbb S}^2} F(\widehat{\psi})\,{\rm d}\sigma\,,$$
for 
$$\Psi(\varphi,\theta,t)=\omega \sin\theta + \widehat{\psi}\,,\qquad \widehat{\psi}(\varphi,\theta,t)=\psi_0(\varphi+\omega t,\theta)\,.$$
We now compute
$$\frac{1}{\cos\theta}\, \left[\frac{\partial\Psi}{\partial\varphi}\,\frac{\partial}{\partial\theta}
  - \frac{\partial\Psi}{\partial\theta}\,\frac{\partial}{\partial\varphi}\right]\left(\Delta\Psi + 2\omega\sin\theta\right) = 
  \frac{1}{\cos\theta}\, \left[\frac{\partial\widehat{\psi}}{\partial\varphi}\,\frac{\partial}{\partial\theta}
  - \Big(\omega \cos\theta +\frac{\partial\widehat{\psi}}{\partial\theta} \Big)\,\frac{\partial}{\partial\varphi}\right] \,F(\widehat{\psi}) 
  =-\omega F'(\widehat{\psi})\,\frac{\partial \widehat{\psi}}{\partial \varphi} \,.$$
On the other hand,
$$\frac{\partial}{\partial t}\,\Delta\Psi = \frac{\partial}{\partial t}\,\Big(\Delta\widehat{\psi}- 2\omega \sin\theta\Big)
=\frac{\partial}{\partial t}\, F(\widehat{\psi})= F'(\widehat{\psi})\,\frac{\partial \widehat{\psi}}{\partial t}=\omega F'(\widehat{\psi})\,\frac{\partial \widehat{\psi}}{\partial \varphi}\,,$$
so that $\psi$ solves \eqref{vorteq}.
\end{proof}

\begin{theorem}\label{zf}
Given the vertical density stratification of the stratosphere $\rho_0(z)$, if $\psi_0(\varphi,\theta)$ solves
\begin{equation}\label{li}
\Delta \psi_0 = F(\psi_0)
\end{equation}
for some $F \in C^1({\mathbb R},{\mathbb R})$, then 
\begin{equation}\label{lirot}
\psi(\varphi,\theta,z,t)=\omega \sin\theta + \frac{1}{\sqrt{\rho_0(z)}}\, \psi_0(\varphi +\omega\,t, \theta)
\end{equation}
with the associated pressure 
\begin{equation}\label{lipres}
p_0(\varphi,\theta,z,t)={\mathcal F}(\psi_0(\varphi +\omega\,t, \theta)) - \frac{1}{2}\,\Big( \frac{\partial \psi_0}{\partial \theta}\,(\varphi +\omega\,t, \theta) \Big)^2 
- \frac{1}{2\cos^2\theta}\,\Big( \frac{\partial \psi_0}{\partial \varphi} \,(\varphi +\omega\,t, \theta)\Big)^2 - g \int_0^z \rho_0(s)\,{\rm d}s\,,
\end{equation}
where ${\mathcal F}$ is a primitive of $F$, is a solution of the system \eqref{eel}-\eqref{strg}, describing height-dependent stratospheric planetary flows  
that propagate zonally westwards.
\end{theorem}

\begin{proof}
Since
$$\Delta \Big(\frac{1}{\sqrt{\rho_0(z)}}\, \psi_0\Big) =\frac{1}{\sqrt{\rho_0(z)}}\,\Delta \psi_0
=\frac{1}{\sqrt{\rho_0(z)}}\,F(\psi_0)=G\Big(z,\,\frac{1}{\sqrt{\rho_0(z)}}\,\psi_0\Big)\quad\text{with}\quad G(z,s)=\frac{1}{\sqrt{\rho_0(z)}}\, F\big(s\sqrt{\rho_0(z)}\,\big)\,,$$
we infer from Lemma \ref{lem} that $\psi$ defined by \eqref{lirot} solves \eqref{vorteq} for every fixed $z$ since \eqref{li} ensures
$$ \iint_{{\mathbb S}^2} F(\psi_0)\,{\rm d}\sigma=0\,.$$
Using \eqref{strg}, the {\it Ansatz} \eqref{lirot} has the following effect on the equations (\ref{eel}a)-(\ref{eel}b): on the left sides, 
only the quadratic terms in $\psi_0$ remain and the factor $\frac{1}{\rho_0}$ cancels out:
\begin{equation}\label{grad}\begin{cases}
\displaystyle\frac{\partial\psi_0}{\partial\theta} \frac{\partial^2\psi_0}{\partial\theta \partial\varphi}  - 
\frac{\partial\psi_0}{\partial\varphi} \frac{\partial^2\psi_0}{\partial\theta^2} + 
\tan \theta \,\frac{\partial\psi_0}{\partial\varphi}\frac{\partial\psi_0}{\partial\theta} =- \frac{\partial p_0}{\partial\varphi}\,,\\[0.3cm]
-\displaystyle\frac{1}{\cos^2\theta}\,\frac{\partial\psi_0}{\partial\theta} \frac{\partial^2\psi_0}{\partial\varphi^2}  + \frac{1}{\cos^2\theta}\,
\frac{\partial\psi_0}{\partial\varphi} \frac{\partial^2\psi_0}{\partial\theta \partial\varphi} + \frac{\sin\theta}{\cos^3\theta}\,
\Big(\frac{\partial\psi_0}{\partial\varphi}\Big)^2 +\tan\theta\,\Big(\frac{\partial\psi_0}{\partial\theta}\Big)^2 = -\frac{\partial p_0}{\partial\theta}\,.
\end{cases}
\end{equation}
Taking \eqref{li} into account we see that the left side of \eqref{grad} is precisely the gradient of the expression
\begin{equation}\label{bern}
\frac{1}{2}\,\Big( \frac{\partial \psi_0}{\partial \theta} \Big)^2 
+ \frac{1}{2\cos^2\theta}\,\Big( \frac{\partial \psi_0}{\partial \varphi} \Big)^2 - {\mathcal F}(\psi_0) 
\end{equation}
with respect to the $(\varphi,\theta)$-variables, and (\ref{eel}c) is easily integrated to yield \eqref{lipres}.
\end{proof}

It is of interest to investigate the stratospheric temperature distribution associated to the vortices \eqref{lirot}. 
With the temperature normalisation 
\begin{equation}\label{ndt}
T'=\frac{U'^2}{{\frak R}'}\, T\,,
\end{equation}
the equation of state \eqref{ig} takes the nondimensional form
\begin{equation}\label{ign}
p=\rho T\,.
\end{equation}
Consequently, for the realistic density distribution $\rho_0(z)=a \,{\rm e}^{-bz}$, where $a>0$ is the (nondimensional) average density 
of the tropopause and $b > 2$ for the atmospheres of our solar system (see the data in \cite{cat}), from \eqref{lipres} we obtain 
the associated stratospheric temperature at leading order:
\begin{equation}\label{st}
T_0(\varphi,\theta, z,t)= \frac{ag}{b} + {\rm e}^{bz}\,\Big( \frac{1}{a} \,\widehat{p_0}(\varphi,\theta,t) + \frac{g}{b}\Big)\,,
\end{equation}
where $\widehat{p_0}(\varphi,\theta,t)$ is the atmospheric pressure at the tropopause. 
Note that \eqref{st} captures the increase of the stratospheric temperature with height.

\bigskip

\begin{tabular}{ |p{1.25cm}||p{2.5cm}|p{6.5cm}| p{6cm}|}
 \hline
 Planet & ${\frak R}'$  & rounded normalisation factor $U'^2/{\frak R}'$ & stratospheric temperature range  \\
 \hline\hline
 Earth   & 287  m$^2$/(s$^2$K)  & 9 K & 220 K to 260 K \\
Jupiter &  3745  m$^2$/(s$^2$K)  & 6 K & 90 K to 150 K\\
Saturn &  3892  m$^2$/(s$^2$K) & 6 K & 110 K to 170 K\\
Uranus &  3615  m$^2$/(s$^2$K)  & 6 K & 55 K to 115 K\\
Neptune &  3615  m$^2$/(s$^2$K) & 11 K & 55 K to 125 K\\
 \hline
\end{tabular}

\bigskip

\begin{remark}
(i) Using the kinematic equations for the material derivative (see \cite{holton})
\begin{equation}\label{ke3d}
u'= \frac{{\rm D}\varphi}{{\rm D}t'}\,r'\,\cos\theta\,,\qquad v'=\frac{{\rm D}\theta}{{\rm D}t'}\,r'\,,\qquad w'=\frac{{\rm D}r'}{{\rm D}t'}\,,
\end{equation}
the Euler equation \eqref{eq:euler} and the equation of mass conservation \eqref{eq:mass-conv} lead to the axial angular momentum conservation law (see \cite{whrs})
\begin{equation}\label{am}
\rho'\, \frac{{\rm D}\ }{{\rm D}t'}\,\Big\{ \big(u'+ \varOmega' r'\cos\theta \big)r'\cos\theta \Big\} = - \frac{\partial p'}{\partial \varphi} \,.
\end{equation}
The kinematic equations for the material derivative of a flow on the unit sphere are (see \cite{cj2})
\begin{equation}\label{kes}
u_0=\cos\theta\, \frac{D}{Dt}\,\varphi\,,\qquad v_0= \frac{D}{Dt}\,\theta\,,
\end{equation}
where
$$\frac{D}{Dt}=\frac{\partial}{\partial t}+ \frac{u_0}{\cos\theta}\frac{\partial }{\partial \varphi} + 
v_0\,\frac{\partial}{\partial \theta}\,.$$
The equations \eqref{kes} are precisely the non-dimensional version of equation 
\eqref{ke3d} with $r' \equiv 1$ (and $w' \equiv 0$). One can now see that (\ref{eel}a) is precisely 
the non-dimensional form of \eqref{am} for flow on a sphere.

(ii) The motion of individual particles of the flow associated to \eqref{lirot} occurs on a sphere determined by the 
initial location, and its evolution is therefore determined by the spherical coordinates $(\varphi(t),\,\theta(t))$. Setting
\begin{equation}\label{Phi}
\Phi(t)= \varphi(t)+ \omega t\,,
\end{equation}
from \eqref{lirot}, \eqref{strg} and \eqref{kes} we get
\begin{align*}
\frac{{\rm d}\ }{{\rm d}t}\, \psi_0(\Phi(t),\theta(t)) &= 
\frac{1}{\sqrt{\rho(z)}}\,\Big\{ \frac{\partial \psi_0}{\partial \varphi}\,\frac{{\rm d}\Phi}{{\rm d}t} 
+ \frac{\partial \psi_0}{\partial \theta}\,\frac{{\rm d}\theta}{{\rm d}t} \Big\} 
= \frac{1}{\sqrt{\rho(z)}}\,\Big\{ \frac{\partial \psi_0}{\partial \varphi}\,\Big(\frac{D_\Sigma\varphi}{Dt} + \omega\Big) + \frac{\partial \psi_0}{\partial \theta}\,\frac{D_\Sigma\theta}{Dt} \Big\} \\
&=\frac{1}{\sqrt{\rho(z)}}\,\Big\{ \frac{\partial \psi_0}{\partial \varphi}\,\Big(\frac{u_0}{\cos\theta} + \omega\Big) + \frac{\partial \psi_0}{\partial \theta}\,v \Big\} \\
&=\frac{1}{\sqrt{\rho(z)}}\,\Big\{ \frac{\partial \psi_0}{\partial \varphi}\,\Big(-\frac{\partial \psi_0}{\partial \theta}\,\frac{1}{\sqrt{\rho(z)}\,\cos\theta}\Big) + \frac{\partial \psi_0}{\partial \theta}\,\Big(\frac{\partial \psi_0}{\partial \varphi}\,\frac{1}{\sqrt{\rho(z)}\,\cos\theta}\Big) \Big\}=0\,.
\end{align*}
Consequently the flow occurs along the level sets of $\psi_0$, translated westward at the speed of rotation of the planet.  
Westward moving persistent flow patterns that are nearly stationary in the rotating frame of reference 
were observed in Saturn's stratosphere at 40$^\circ$, 55$^\circ$ and 70$^\circ$ N and S, these flows being remarkably symmetric 
about Saturn's equator (see \cite{greg}). Similar coherent high-latitude bands of westward flows in Jupiter's stratosphere, persisting 
for 70 days, were captured in 2000 during the Cassini mission (see \cite{kaspi} for data and Figure 2 for a visualisation). 
Terrestrial patterns of a similar nature also occur but are rather rare events, e.g., major stratospheric warmings 
may disrupt the eastward polar vortex and give rise to westward winds lasting typically a few days (see \cite{holton}); these 
attain the planet's speed of rotation at high latitudes. Thus the flow induced by \eqref{lirot} captures physically realistic patterns in 
suitable latitude bands -- alternating eastward and westward traveling belts being typical for Jupiter and Saturn. On the other hand, the stratospheric flow for Uranus and Neptune 
is highly zonal, featuring a broad retrograde equatorial jet and high-latitude prograde jets (see Fig. \ref{fig2}). Geostationary flow patterns lasting for decades 
occur at about 20$^\circ$ latitude on Uranus, while on Neptune they can be observed near  50$^\circ$ latitude but appear to be rather short-lived (see the data in \cite{ga}).

(iii) The effect of replacing \eqref{li} by
$$\Delta \psi_0 = F(\psi_0) - \frac{1}{4\pi}\,\iint_{{\mathbb S}^2} F(\psi_0)\,{\rm d}\sigma$$
(so that the Gauss contraint \eqref{gaussc} is satisfied) brings about the additive correction term
$$\frac{\psi_0(\varphi+\omega t)}{4\pi\sqrt{\rho_0(z)}}\,\iint_{{\mathbb S}^2} F(\psi_0)\,{\rm d}\sigma$$
in \eqref{bern} and the compatibility of the horizontal gradient with (\ref{eel}c) is not granted unless we allow for a forcing term as a perturbation of gravity acting in the radial direction. This feature 
is replicated if we start with \eqref{ellipticpsi} rather than \eqref{li}. Thus an interesting direction for further investigations is opened up since such forcing terms appear naturally 
if one accounts for oblateness: rapidly rotating planets deviate from a perfect sphere by flattening at the poles and bulging at the Equator (see \cite{cj2} for further details in the terrestrial setting).\qed
\end{remark}

\section*{Appendix: Spherical harmonics}
\setcounter{equation}{0}
\renewcommand\theequation{A.\arabic{equation}}

We collect some properties of spherical harmonics relied upon throughout the paper.

The eigenvalues of the Laplace-Beltrami operator on the unit sphere $\Sigma$, acting on functions with vanishing spherical average\footnote{Without this restriction, zero would be an eigenvalue with constant eigenvectors.} are $ \{-j(j+1),\ j \in {\mathbb N} \}$.  
We will denote by $\mathbb{E}_j$ the $j$-th eigenspace, of dimension $2j + 1$, associated to the eigenvalue $-j(j+1)$, and by $\mathbb{P}_j$ the corresponding spectral projector. 
A basis of $\mathbb{E}_j$ is provided by the $(2j+1)$ spherical harmonics\footnote{While in geophysics it convenient to use the latitude 
$\theta \in [-\tfrac{\pi}{2},\tfrac{\pi}{2}]$, in quantum mechanics one typically uses instead the co-latitude or polar 
angle $\Theta=\tfrac{\pi}{2}-\theta \in [0,\pi]$. Passing from one set of coordinates to the other requires only an interchange of $\sin$ and $\cos$ in all 
explicit expressions and to keep track of the range of the corresponding angles.}
$$
Y_{j}^m (\varphi,\theta)=(-1)^m \sqrt{\tfrac{(2j+1)(j-m)!}{4\pi(j+m)!}} \,P_{j}^m (\sin \theta) {\rm e}^{{\rm i} m \varphi}\,,\qquad m=-j,\dots,j\,,
$$
of degree $j$ and zonal number $m$ ($-j \le m \le j$), where 
$$P_{j}^m(x)=\frac{1}{2^j j!}\,(1-x^2)^{m/2}\, \frac{{\rm d}^{j+m}}{{\rm d}^{j+m}x}\,(x^2-1)^j \,,\qquad m=-j,\dots,j\,,$$ 
are the associated Legendre polynomials, satisfying (see \cite{mull})
\begin{equation}\label{a-mm}
Y_{j}^{-m}=(-1)^m \,\overline{Y_{j}^{m}} \,,\qquad m=-j,\dots,j\,,
\end{equation} 
where the overline means complex conjugation. The only zonal spherical harmonics of degree $j$ are $Y_j^0$, 
$Y_j^{\pm j}$ are called sectoral and change only in the longitudinal direction, while for $1 \le |m| \le l-1$ the spherical harmonics $Y_j^m$ 
are called tesseral and vary in both the longitudinal and latitudinal directions. A real orthonormal basis of spherical harmonics 
$\{ R_j^m\}$ can be defined in terms of their complex analogues by setting
$$
R_j^m=\begin{cases}
\frac{{\rm i}[Y_j^m - (-1)^m Y_j^{-m}]}{\sqrt{2}}=(-1)^m \sqrt{\tfrac{(2j+1)(j-|m|)!}{2\pi(j+|m|)!}} \,P_{j}^{|m|} (\sin \theta) \sin(| m| \varphi)\,,\quad & m<0\,,\\[0.2cm]
Y_j^0=\sqrt{\tfrac{(2j+1)}{4\pi}}\,P_{j}^{0} (\sin \theta)\,,\quad &  m=0\,,\\[0.2cm]
\frac{Y_j^m + (-1)^m Y_j^{-m}}{\sqrt{2}}=(-1)^m  \sqrt{\tfrac{(2j+1)(j-|m|)!}{2\pi(j+m)!}} \,P_{j}^{|m|} (\sin \theta) \cos( m\varphi)\,,\quad & m>0\,.
\end{cases}$$
The first eigenspaces are
\begin{itemize}
\item $\mathbb{E}_1$, which admits the orthonormal complex basis (with respect to the spherical surface element ${\rm d}\sigma=\cos\theta\,{\rm d}\theta {\rm d}\varphi$)
$$Y_1^{-1}(\varphi,\theta)=\tfrac{1}{2}\sqrt{\tfrac{3}{2\pi}}\,\cos\theta\,{\rm e}^{-{\rm i} \varphi} \,,\qquad Y_1^0(\theta)=\tfrac{1}{2}\sqrt{\tfrac{3}{\pi}}\,\sin\theta\,,\qquad Y_1^1(\varphi,\theta)=-\tfrac{1}{2}\sqrt{\tfrac{3}{2\pi}}\,\cos\theta\,{\rm e}^{{\rm i} \varphi}\,,$$
with the corresponding orthonormal real basis
$$R_1^{-1}(\varphi,\theta)=\tfrac{1}{2}\sqrt{\tfrac{3}{\pi}}\,\cos\theta\,\sin \varphi \,,\qquad R_1^0(\theta)=\tfrac{1}{2}\sqrt{\tfrac{3}{\pi}}\,\sin\theta\,,\qquad R_1^1(\varphi,\theta)=-\tfrac{1}{2}\sqrt{\tfrac{3}{\pi}}\,\cos\theta\,\cos \varphi\,;$$
\item $\mathbb{E}_2$, which admits the orthonormal complex basis
$$Y_2^{\pm1}(\varphi,\theta)=\mp\tfrac{1}{2}\sqrt{\tfrac{15}{2\pi}}\,\sin\theta \cos\theta\,{\rm e}^{\pm {\rm i} \varphi} \,,\qquad Y_2^0(\theta)=\tfrac{1}{4}\sqrt{\tfrac{5}{\pi}}\,(3\sin^2\theta-1)\,,
\qquad Y_2^{\pm 2}(\varphi,\theta)=\tfrac{1}{4}\sqrt{\tfrac{15}{2\pi}}\,\cos^2\theta\,{\rm e}^{\pm 2{\rm i} \varphi}\,,$$
with the corresponding orthonormal real basis
\begin{align*}
&R_2^{-2}(\varphi,\theta)=\tfrac{1}{4}\sqrt{\tfrac{15}{\pi}}\,\cos^2\theta\,\sin(2\varphi)\,,\qquad &R_2^{-1}(\varphi,\theta) =\tfrac{1}{2}\sqrt{\tfrac{15}{\pi}}\,\sin\theta \cos\theta\,\sin \varphi \,,\qquad &R_2^0(\theta)=\tfrac{1}{4}\sqrt{\tfrac{5}{\pi}}\,(3\sin^2\theta-1)\,,\\
&R_2^1(\varphi,\theta) =-\tfrac{1}{2}\sqrt{\tfrac{15}{\pi}}\,\sin\theta \cos\theta\,\cos \varphi \,,\qquad 
&R_2^2(\varphi,\theta)=\tfrac{1}{4}\sqrt{\tfrac{15}{\pi}}\,\cos^2\theta\,\cos(2\varphi)\,.\qquad&
\end{align*}
\end{itemize}
The only spherical harmonics with modes $j \ge 3$ that we refer to in this paper are
$$Y_3^0(\theta)=\tfrac{1}{4}\sqrt{\tfrac{7}{\pi}}\,(5\sin^3\theta-3\sin\theta) \quad\text{and}\quad Y_5^0(\theta)=\tfrac{1}{16}\sqrt{\tfrac{11}{\pi}}\,(63\sin^5\theta- 705\sin^3\theta +15\sin\theta)\,.$$
Generally we have
$$Y_j^0(\theta)=\frac{1}{2^{j+1}j!} \sqrt{\frac{2j+1}{\pi j!}}\,P_j^0(\sin\theta)\,,$$
with $P_j^0(-x)=(-1)^j P_j^0(x)$ on $(-1,1)$ and $P_j^0(1)=1$, while $P_j^0(0)=0$ for $j$ odd and $P_j(0) = \frac{(-1)^{j/2} j!}{2^j (j/2)!}$ for $j$ even. 
For $j \ge 1$ the Legendre polynomial $P_j^0$, of degree $j$, has $j$ distinct simple roots and $j-1$ local 
minima and maxima in the interval $(-1,1)$, while 
$\pm 1$ are global extrema in $[-1,1]$ with $(P_j^0)'(1)=\frac{j(j+1)}{2}$. While a general formula for the roots $s_k$ of $P_j$ in $(-1,1)$ is, to 
the best of our knowledge, still elusive, their location is quite accurately described by 
$$s_k \approx \cos \Big( \frac{(4k-1)\pi}{4j+2}\Big)\,,\qquad k=1,\dots,j\,.$$
The fact that  for $j \ge 1$ the zonal spherical harmonics $R_j^0$ has exactly $j$ nodal domains (connected components of the complement of the set of zeros) on the sphere $\Sigma$ is generally not replicated 
by the real spherical harmonics $R_j^m$ with $j \ge 2$ and $m \neq 0$, each of the three real spherical harmonics of degree one having 
two nodal domains (see \cite{ley}).

Representation theory highlights the relevance of symmetries in the study of spherical harmonics. An element ${\frak g}$ of the 
group $SO(3)$ of the rotations of the sphere $\Sigma$ can be parametrized by the Euler angles as
$${\frak g}(\alpha,\beta,\gamma)={\frak R}_z(\gamma){\frak R}_y(\beta) {\frak R}_z(\alpha)\,,$$
where ${\frak R}_z(\gamma)$ and ${\frak R}_y(\beta)$ represent a rotation around the $z$-axis by $\gamma$ radians and 
a rotation around the $y$-axis by $\beta$ radians. To any ${\frak g} \in SO(3)$ we can associate a rotation $\Lambda_{\frak g}$ 
on $L^2(\Sigma)$, defined by $(\Lambda_{\frak g}f)(\eta)=f({\frak g}^{-1}\eta)$. The mapping ${\frak g} \mapsto \Lambda_{\frak g}$ 
is a unitary representation of $SO(3)$, and restricting $\Lambda_{\frak g}$ to the finite-dimensional space 
of spherical harmonics of degree $j  \ge 0$, consisting of the linear combination of the spherical harmonics of degree $j$, one 
obtains all the irreducible representations of $SO(3)$, in the sense that there are no genuine invariant subspaces (see \cite{s}). 
A rotated spherical harmonic of degree $j$ can be written as a linear combination of spherical harmonics of degree $j$ by means of 
the formula
$$\Lambda_{\frak g} Y_j^m = \sum_{|k| \le j} {\mathbb U}_l^{mk}({\frak g}) Y_l^k\,,$$
where
$${\mathbb U}_l^{mk}({\frak g}(\alpha,\beta,\gamma))={\rm e}^{-{\rm i} (m\gamma+ k\alpha)}\, P_l^{mk}(\sin\beta)\,,$$
with $P_j^{mk}$ being the generalized associated Legendre polynomials, given for $m,\, k \in \{-j,\dots,j\}$ by (see \cite{vil})
$$P_{j}^{mk}(x)=\frac{(-1)^{j-m}}{2^j}\,\sqrt{\frac{(j+m)!}{(j-k)!(j+k)!(j-m)!}}
(1+x)^{-(m+k)/2}\, (1-x)^{(k-m)/2}\,\frac{{\rm d}^{j-m}}{{\rm d}^{j-m}x}\,[(1-x)^{j-k}(1+x)^{j+k}]\,.$$
Note that the spherical harmonics of degree one can be obtained one from any another by a rotation. 

Expanding a stream function $\psi \in H^2({\mathbb S}^2)$ in spherical harmonics 
$$\psi(\varphi,\theta)= \sum_{j \ge 1}  \sum_{m=-j}^j  \alpha_j^m(t)  Y_j^m(\varphi,\theta) \,,$$
the associated velocity $U$ and vorticity $\Omega$ are represented by
\begin{align*}
U &= \Big(- \sum_{j \ge 1}  \sum_{m=-j}^j \alpha_j^m(t)  \frac{\partial Y_j^m}{\partial \theta}\,,\ \frac{1}{\cos\theta} \sum_{l \ge 1}  \sum_{k=-l}^l 
\alpha_l^k(t)  \frac{\partial Y_l^k}{\partial \varphi}\Big)\,,\\
\Omega &= - \sum_{j \ge 1}  \sum_{m=-j}^j j(j+1) \,\alpha_j^m(t)  Y_j^m(\varphi,\theta)\,.
\end{align*}
Since the surface gradients of spherical harmonics are also orthogonal, we have (see \cite{khol})
$$
\iint_{{\mathbb S}^2} |U|^2 \,d\sigma =\sum_{j \ge 1}  \sum_{m=-j}^j j(j+1)\,|\alpha_j^m(t)|^2  \,,\qquad
\iint_{{\mathbb S}^2} |\Omega|^2 \,d\sigma =\sum_{j \ge 1}  \sum_{m=-j}^j j^2(j+1)^2\,|\alpha_j^m(t)|^2\,.$$
As a consequence one infers the validity of the sharp Poincar\'e inequality 
\begin{equation}\label{poin}
\iint_{{\mathbb S}^2} |\Omega|^2 \,d\sigma \ge (n+1)(n+2) \iint_{{\mathbb S}^2} |U|^2 \,d\sigma\,,\qquad 
\psi \in H^2({\mathbb S}^2) \cap \Big(\bigcap\limits_{j=1}^n {\mathbb E}_j^\perp\Big)
\end{equation}
where ${\mathbb E}_j^\perp$ is the orthogonal complement in $L^2({\mathbb S}^2)$ of the $(2j+1)$-dimensional eigenspace ${\mathbb E}_j$  
of the eigenvalue $-j(j+1)$ of the Laplace-Beltrami operator.


\end{document}